\documentclass{article}

\usepackage[T1]{fontenc}
\usepackage[htt]{hyphenat}

\usepackage{microtype}
\usepackage{graphicx}
\usepackage{subfigure}
\usepackage{booktabs} 
\usepackage{natbib}
\usepackage{algorithm}
\usepackage{algorithmic}
\usepackage{siunitx}
\usepackage{subfigure}
\usepackage[subfigure]{tocloft}

\let\oldaddcontentsline\addcontentsline

\newcommand{\starttocentries}{\let\addcontentsline\oldaddcontentsline}

\usepackage{amsmath}
\usepackage{amssymb}
\usepackage{graphicx}
\usepackage{mathtools}
\usepackage{amsfonts}
\usepackage{amsthm,bm}
\usepackage{color}
\usepackage{enumitem}
\usepackage{dsfont}
\usepackage{pgfplots}
\pgfplotsset{width=10cm,compat=1.9}
 \usepackage{pifont}
\usepackage{thmtools} 
\usepackage{thm-restate}
\usepackage{sidecap}

\newcommand{\R}{\mathbb{R}}

\renewcommand{\phi}{\varphi}

\graphicspath {{figures/}}

\usepackage{hyperref}
\usepackage{caption}
\usepackage[super]{nth}
\newtheorem{lemma}{Lemma}

\newtheorem{theorem}{Theorem}

\newtheorem{proposition}{Proposition}

\newtheorem{remark}{Remark}
\newtheorem{corollary}{Corollary}

\mathtoolsset{showonlyrefs}

\DeclareMathOperator*{\argmin}{argmin}

\usepackage{breakcites}
\usepackage{authblk}
\usepackage[margin=1.29in]{geometry}
\usepackage{enumitem}

\usepackage{hyperref}

\hypersetup{
     colorlinks = true,
     linkcolor = brown,
     anchorcolor = blue,
     citecolor = teal,
     filecolor = blue,
     urlcolor = black
     }

\title{Computing the Variance of Shuffling Stochastic Gradient Algorithms via Power Spectral Density Analysis}

\author[a]{Carles Domingo-Enrich}
\affil[a]{Courant Institute of Mathematical Sciences, New York University}

\begin{document}

\maketitle


\begin{abstract}%
  When solving finite-sum minimization problems, two common alternatives to stochastic gradient descent (SGD) with theoretical benefits are random reshuffling (SGD-RR) and shuffle-once (SGD-SO), in which functions are sampled in cycles without replacement. Under a convenient stochastic noise approximation which holds experimentally, we study the stationary variances of the iterates of SGD, SGD-RR and SGD-SO, whose leading terms decrease in this order, and obtain simple approximations. To obtain our results, we study the power spectral density of the stochastic gradient noise sequences. Our analysis extends beyond SGD to SGD with momentum and to the stochastic Nesterov's accelerated gradient method. We perform experiments on quadratic objective functions to test the validity of our approximation and the correctness of our findings.
\end{abstract}

\addtocontents{toc}{\protect\setcounter{tocdepth}{0}}

\section{Introduction} \label{sec:intro}
We consider the finite-sum minimization problem
\begin{align}
    x^{\star} = \argmin_{x \in \R^d} \bigg\{ f(x) = \frac{1}{n} \sum_{i=1}^n f_i(x) \bigg\}.
\end{align}
This setting is ubiquitous in machine learning; the standard formulation of supervised learning problems is of this form. Stochastic first-order algorithms, where an iterate $(x_k)_k$ is successively updated using the gradients $\nabla f_i(x_k)$ of a subset of the functions $f_i$, are popular among practitioners thanks to their scalability and low memory cost \citep{bottou2016optimization}. The number of functions used in each update is known as the mini-batch size. 

The simplest stochastic first-order algorithm is \textit{stochastic gradient descent} (SGD). For mini-batch size 1, its update reads $x_{k+1} = x_{k} - \gamma \nabla f_{i_k}(x_k)$, where $(i_k)_k$ are uniformly random, independent indices in $\{1, \dots, n\}$, and $\gamma > 0$ is known as the stepsize. A related algorithm which is heavily used in practice is \textit{stochastic gradient descent with momentum} (SGDM), also known as stochastic heavy-ball momentum. The update is of the form $x_{k+1} = y_k - \gamma \nabla f_{i_k}(x_k)$, $y_k = x_k + \alpha (x_k - x_{k-1})$, where $\alpha \in [0,1)$ is the momentum weight. Computing the gradient at $y_k$ instead of $x_k$ gives rise to the \textit{stochastic Nesterov's accelerated gradient} method (SNAG): $x_{k+1} = y_k - \gamma \nabla f_{i_k}(y_k)$, $y_k = x_k + \alpha (x_k - x_{k-1})$. The deterministic counterparts of SGDM and SNAG are Polyak's heavy-ball method \citep{polyak1964some} and Nesterov's accelerated gradient method \citep{nesterov1983amethod,nesterov2003introductory}.

There is a vast array of works showing convergence rates for stochastic gradient algorithms. When $f_i$ are $L$-smooth and $\mu$-strongly convex\footnote{A function $f$ is $L$-smooth when it is differentiable and its gradient is $L$-Lipschitz: $\|\nabla f(x) - \nabla f(y)\| \leq L\|x-y\|$. It is $\mu$-strongly convex when $f(x) \geq f(y) + \langle \nabla f(y), x - y\rangle + \frac{\mu}{2} \|x-y\|^2$.}, the minimizer $x^{\star}$ is unique and when $\gamma \leq \frac{1}{2L}$ the SGD iterates fulfill \citep{needell2014stochastic,stich2019unified,gower2019sgd}
\begin{align} \label{eq:iterate_error}
\mathbb{E}[\|x_k - x^{\star}\|^2] \leq (1-\gamma \mu)^{k} + \frac{2 \gamma \sigma^2}{\mu}.
\end{align}
Here, $\sigma^2 = \frac{1}{n} \sum_{i=1}^{n} \|\nabla f_i(x^{\star}) \|^2$ is the variance of the gradients at the minimizer. The first term of the bound (the \textit{bias term}) decreases exponentially fast, but the second one (the \textit{variance term}) is stationary and can only be reduced by tuning down the stepsize. Other works \citep{rakhlin2012making,drori2020thecomplexity,nguyen2019tight} use a similar bias-variance decompositions; upon optimizing the bound with respect to $\gamma$ at a fixed horizon $k$, they show matching upper and lower bounds on the quantity $\mathbb{E}[f(x_k) - f(x^{\star})]$ of order $\Theta(\frac{1}{k})$. 

In practice, it is common to use alternative schemes to sample the function $f_{i_k}$ used in the $k$-th update, c.f. \cite{mishchenko2020random}. A very popular variation is \textit{random reshuffling} (RR), in which training is divided into epochs of $n$ iterations. At the beginning of the epoch $T$, the indices $i_{nT}, i_{nT + 1}, \dots, i_{(n+1)T - 1}$ are sampled \textit{without replacement} from $\{1, \dots, n\}$, i.e. $(i_{nT}, i_{nT + 1}, \dots, i_{(n+1)T - 1})$ is a random permutation of $\{1, \dots, n\}$. Another frequent heuristic is \textit{shuffle-once} (SO), which is like RR but such that the functions are shuffled only before the first epoch, and the permutation is reused in subsequent epochs. 
Note that RR and SO may be combined with any first-order stochastic algorithm, giving rise for example to SGDM-RR or SNAG-SO.

Compared to SGD, the quantity $\mathbb{E}[f(x_k) - f(x^{\star})]$ has been harder to analyze theoretically for these variants, as gradient estimates are not conditionally unbiased. The early attempt by \cite{recht2012toward} to study RR relied on a conjecture that is false in general \citep{lai2020rechtre}. Later on, \cite{gurbuzbalaban2015whyrandom, haochen2019random, nagaraj2019sgd, mishchenko2020random} showed that SGD-RR enjoys a faster convergence rate than SGD under smoothness (without smoothness SGD is optimal). More recently, \cite{safran2020howgood,rajput2020closing} manage to get matching upper lower and bounds for SGD-RR of order $O(\frac{1}{k^2} + \frac{n^2}{k^3})$, which is faster than the rate $\Theta(\frac{1}{k})$ for SGD. \cite{safran2020howgood} also prove matching bounds for SGD-SO of order $\Theta(\frac{n}{k^2})$, showing that it performs worse than SGD-RR but better than GD. 



\textbf{Our approach.} Suppose that $x^{\star}$ is a local minimizer of $f$ and that the Hessian $A := Hf(x^{\star})$ is strictly positive definite, with smallest eigenvalue $\mu$. Assuming that the functions $f_i$ are twice-differentiable, by Taylor's theorem we can write $\nabla f_i(x) = \nabla f_i(x^{\star}) + Hf_i(x^{\star})(x-x^{\star}) + R_i(x)$, where the residual $R_i$ fulfills $R_i(x) = o(\|x-x^{\star}\|)$. Defining 
\begin{align}
    \widehat{\nabla f_i}(x) = \nabla f_i(x^{\star}) + A(x-x^{\star}),
\end{align}
we can write $\nabla f_i(x) = \widehat{\nabla f_i}(x) + (Hf_i(x^{\star})-A)(x-x^{\star}) + R_i(x)$, and since $R_i(x) + (Hf_i(x^{\star})-A)(x-x^{\star}) = O(\|x-x^{\star}\|)$, we obtain that $\nabla f_i(x) = \widehat{\nabla f_i}(x) + O(\|x-x^{\star}\|)$. That is, $\widehat{\nabla f_i}(x)$ is an approximation of $\nabla f_i(x)$ to zero-th order; we refer to this setting as the \textit{zero-th order noise model}.
Since the stationary variance of iterates decreases proportionally to the stepsize $\gamma$ (equation \eqref{eq:iterate_error}), in the regime $\gamma \ll 1$ we obtain that $\widehat{\nabla f_{i_k}}(x_k)$ is a good proxy for $\nabla f_{i_k}(x_k)$ when $k$ is large enough that the bias term is negligible. We discuss further the consistency of the zero-th order noise model in \autoref{app:validity}. 
Our setting is a particular case of one studied by \cite{gitman2019understanding}, which assume access to gradient estimates $g_k = \nabla f(x_k) + \chi_k$ with $\chi_k$ a random zero-mean vector independent of $x_k$. Our model amounts to using $\chi_k = \nabla f_{i_k}(x^{\star}) - \nabla f(x_k) + A(x-x^{\star})$. Being more tailored to the finite-sum minimization problem, we are able to obtain more fine-grained results.

\textbf{Contributions.} Under the approximation $\nabla f_i(x) \approx \widehat{\nabla f_i}(x)$, we perform an analysis of the stationary covariance $\mathbb{E}[(x_k - x^{\star})(x_k - x^{\star})^{\top}]$ for SGD, SGDM and SNAG, under different shuffling schemes: vanilla (with replacement), RR and SO. Our main contributions are as follows:
\begin{itemize}[leftmargin=6mm,itemsep=-2pt,topsep=-2pt]
    \item We obtain exact analytic expressions for the stationary variance of the iterates of SGD, SGDM and SNAG with replacement, and an approximation that goes as $\Theta(\gamma/(1-\alpha))$ (\autoref{sec:with_replacement}).
    \item We also derive exact analytic expressions and approximations for the stationary variance of the iterates of SGD, SGDM and SNAG with SO (\autoref{sec:SO}) and RR (\autoref{sec:RR}). For SO, the approximation goes as $\Theta(\gamma^2 n/(1-\alpha)^2)$ and for RR it is similar but with a worse dependency on $n$. We reproduce this finding experimentally, and it runs counter to the convergence bounds of \cite{safran2020howgood} showing how RR outperforms SO.
    \item We perform experiments (\autoref{sec:experiments} and \autoref{sec:more_experiments}) which show that our analytic expressions match the empirical variances under the zero-th order noise model, and also that iterate variances under our noise model are very close to variances under the standard stochastic gradient noise, except in some cases in \autoref{sec:more_experiments}.
\end{itemize}
Compared to the standard analyses that provide upper and lower bounds on $\mathbb{E}[f(x_k) - f(x^{\star})]$, our approach has pros: we get direct theory-based numerical comparisons which prescribe the best method to use when the algorithm is close to convergence, and cons: we do not model the transient time of each algorithm, e.g. arguably SGD-SO has lower stationary variance than SGD-RR but higher transient time, which is why its convergence rate is worse.  
To obtain our results, we model the noise term as a wide-sense stationary process, and the iterates as a linear time-independent (LTI) transformation of the noise. We compute the autocorrelation function for the noise sequence and its Fourier function: the power spectral density of the noise. Using the transfer function for the LTI transformation, we obtain the power spectral density of the iterates, which then yields the covariance.

\textbf{Further related work.}
Analogous bounds on $\mathbb{E}[f(x_k) - f(x^{\star})]$ have been devised for SGDM \citep{liu2020animproved} 
and SNAG \citep{aybat2019auniversally}. Another work that provides expressions for the stationary variance of stochastic algorithms with momentum is \cite{gitman2019understanding}, but their assumptions are not tailored to finite-sum problems and the resulting expressions are not simple. On a different topic, alternative permutation-based SGD with better convergence bounds have been derived \citep{rajput2022permutationbased} There has also been work studying stochastic gradient algorithms from the continuous time perspective \citep{su2016adifferential,yang2018thephysical}.


\section{Preliminaries} \label{sec:preliminaries}

\textbf{Notation.}
We denote by $\mathcal{S}(\R^d)$ the space of Schwartz functions, which contains the functions in $\mathcal{C}^{\infty}(\R^{d})$ whose derivatives of any order decay faster than polynomials of all orders, i.e. for all $k, p \in (\mathbb{N}_0)^d$, $\sup_{x \in \R^d} |x^{k} \partial^{(p)} \phi(x)| < +\infty$. We denote by $\mathcal{S}'(\R^{d})$ the dual space of $\mathcal{S}(\R^d)$, which is known as the space of tempered distributions on $\R^{d}$.

\textbf{Wide-sense stationary processes and the autocorrelation function.} A $\mathbb{R}^d$-valued random process $y = (y_k)_{k \in \mathbb{Z}}$ is wide-sense stationary when its first moment $m_{y}(k) := \mathbb{E}[y_k]$ and autocovariance $K_{y}(k,k+\kappa) = \mathbb{E}[(y_k - m_y(k))(y_{k+\kappa} - m_y(k+\kappa))^{\top}]$ do not depend on the time $k$, and its second moment is finite at all times. That is, for all $k, \kappa \in \mathbb{Z}^{+}$,
\begin{align}
    m_y(k+\kappa) = m_y(\kappa), \quad K_{y}(k,k+\kappa) = K_{y}(0,\kappa), \quad \mathbb{E}[\|y_k\|^2] < +\infty.
\end{align}
This also implies that the \textit{autocorrelation} $R_{y}(k,k+\kappa) := \mathbb{E}[y_k y_{k+\kappa}^{\top}] = K_{y}(k,k+\kappa) + m_y(k) m_y(k+\kappa)^{\top}$ does not depend on $k$ but only on the time difference $\kappa$. Hence, we define the autocorrelation function $R_{y}(\kappa) = R_{y}(k,k+\kappa)$.

\textbf{The power spectral density.} For wide-sense stationary processes, one can use the autocorrelation function to compute the power spectral density, which provides information on the frequency content of the process. When $R_y$ decays fast enough (absolute summability, i.e. $\sum_{k=-\infty}^{\infty} \|R_y\| < +\infty$), the power spectral density is a function $S_y$ on $\R$ with period 1 defined as 
\begin{align} \label{eq:S_fourier_R}
    S_y(f) = \sum_{k=-\infty}^{+\infty} e^{-2\pi i k f} R_y(k). 
\end{align}
Note that $(R_y(-k))_k$ is the Fourier series of $S_y$. When $R_y$ is absolutely summable, the Weierstrass M-test readily shows that the convergence of \eqref{eq:S_fourier_R} is uniform. When $R_y$ is not absolutely summable, the power spectral density cannot be defined as a function in general, but it can still be defined as a tempered distribution; we defer the construction to \autoref{app:general_def}. 

\textbf{Linear time invariant (LTI) systems.} Discrete-time $\mathbb{R}^{d}$-valued LTI systems are maps of the form $b = (b_k)_{k \in \mathbb{Z}} \mapsto x = (x_k)_{k \in \mathbb{Z}}$, where
$x_k = (h * b)_k := \sum_{k'=-\infty}^{\infty} H_{k-k'} b_{k'}$.
Here, $h = (H_k)_{k \in \mathbb{Z}}$ is a sequence of matrices in $\mathbb{R}^{d \times d}$ which is known as the \textit{impulse response} of the LTI system, and characterizes it completely. As the name suggests, LTI systems are exactly the maps for which the output $b$ is a linear function of the input, and any time shift on the input yields the same shift on the output, i.e. if $(b_k)_{k \in \mathbb{Z}} \mapsto (x_k)_{k \in \mathbb{Z}}$, then $(b_{k-\kappa})_{k \in \mathbb{Z}} \mapsto (x_{k-\kappa})_{k \in \mathbb{Z}}$ for any $\kappa \in \mathbb{Z}$.

To analyze discrete time LTI systems, it is useful and customary to introduce the \textit{Z transform}, which for a sequence $(y_k)_{k\in \mathbb{Z}}$ is the function $Y : \mathbb{C} \to \mathbb{C}$ defined as $Y(z) := \mathcal{Z}\{y\}(z) = \sum_{k=-\infty}^{\infty} y_k z^{-k}$. Note, for example, that the relation \eqref{eq:S_fourier_R} between the power spectral density and the autocorrelation function may be written succintly as $S_y(f) = \mathcal{Z}\{R_y\}(e^{2\pi i f})$. Importantly, the Z transforms of the input, output and impulse response of the LTI system $x = h * b$ fulfill $X(z) = H(z) B(z)$. The Z transform $H(z) = \sum_{k=-\infty}^{\infty} H_k z^{-k}$ of the impulse response, which is matrix-valued in our context, is known as the \textit{transfer function} of the LTI system. 

The following theorem relates the power spectral densities of the output and input of an LTI system. This kind of result is very well-known, but we provide a proof in \autoref{app:proof_thm_1} because it is usually stated in continuous time and dimension 1. In fact, we prove a more general result (\autoref{thm:LTI_PSD_2}) in which the autocorrelation $R_b$ of the input is not assumed absolutely summable, which means that the power spectral densities $S_x, S_b$ are defined as tempered distributions. 

\begin{theorem} \label{thm:LTI_PSD}
Suppose that $b = (b_k)_{k \in \mathbb{Z}}$ is a $\mathbb{R}^d$-valued wide-sense stationary process and $x = h*b = (x_k)_{k \in \mathbb{Z}}$ is the output of $b$ by an LTI system with impulse response $h = (H_k)_{k \in \mathbb{Z}}$, where $H_k \in \R^{d \times d}$. 
Suppose that the impulse response is absolutely summable: $\sum_{k=-\infty}^{\infty} \|H_k\| < +\infty$, and that the autocorrelation function $R_b$ of the input is also absolutely summable. Then $x$ is a wide-sense stationary process with power spectral density: $S_x(f) = H(e^{-2\pi i f}) S_b(f) H(e^{2\pi i f})^{\top}$.
\end{theorem}

When $(R_b(k))_k$ is absolutely summable, \autoref{thm:LTI_PSD} implies that
\begin{align} \label{eq:variance_expression}
    \mathbb{E}[\|x_{k}\|^2] = \mathrm{Tr}[R_x(0)] = \int_{0}^{1} \mathrm{Tr}[S_x(f)] \, df = \int_{0}^{1} \mathrm{Tr}[H(e^{-2\pi i f}) S_b(f) H(e^{2\pi i f})^{\top}] \, df. 
\end{align}
A similar result holds true even when $(R_b(k))_k$ is not absolutely integrable (\autoref{lem:variance_S_x} in \autoref{app:proof_thm_1}); the general idea is that we can compute the second moment the output $x$ from the transfer function of the LTI system and the power spectral density or the autocorrelation of the input $b$.

\section{Variance computation for stochastic gradient algorithms with replacement} \label{sec:with_replacement}

In this section we obtain analytic expressions for the variance of the iterates of SGD, SGDM and SNAG when the stochastic gradients are sampled with replacement (i.e. the standard versions of these algorithms) under the zero-th order noise model described in \autoref{sec:intro}. Without loss of generality, we consider functions $(f_i)_{i=1}^n$ and $f$ of the form
\begin{align}
\begin{split} \label{eq:f_form}
    f_i(x) &= \frac{1}{2} x^{\top} A_i x + b_i^\top x, \text{where $A_i \in \R^{d \times d}$ symmetric}, \\
    f(x) &= \frac{1}{2} x^{\top} \bigg(\frac{1}{n} \sum_{i=1}^n A_i\bigg) x + \bigg(\frac{1}{n} \sum_{i=1}^n b_i \bigg)^{\top} x := \frac{1}{2} x^{\top} A x + b x,
\end{split}
\end{align}
where $A := \frac{1}{n} \sum_{i=1}^n A_i$, $b := \frac{1}{n} \sum_{i=1}^n b_i$, and $A$ is assumed to be strictly positive definite. Clearly the minimizer of $f$ is $x^{\star} = - A^{-1} b$. Hence, $\nabla f_i(x) = A_i x + b_i$, $\nabla f(x) = A x + b$ and the zero-th order approximations are $\widehat{\nabla f_i}(x) = \nabla f_i(x^{\star}) + A(x - x^{\star}) = A x - A_i A^{-1} b + b_i + b$. We define $y_i := - A_i A^{-1} b + b_i$, and note that $\widehat{\nabla f_i}(x) = A x + y_i + b$, $\frac{1}{n} \sum_{i=1}^{n} y_i = 0$. 

\textbf{Stochastic gradient algorithms as LTI systems.} The expression for SGDM under the zero-th order noise model is:
\begin{align}
    x_{k+1} &= x_{k} + \alpha(x_{k} - x_{k-1}) - \gamma \widehat{\nabla f_{i_k}}(x_{k}) = x_{k} + \alpha(x_{k} - x_{k-1}) - \gamma (A x_{k} + y_{i_k} + b) 
\end{align}
Defining the recentered iterate sequence $\Delta x_{k} = x_k - x^{\star}$, we obtain that
\begin{align} \label{eq:recentered_SGDM}
    \Delta x_{k+1} = \Delta x_{k} + \alpha(\Delta x_{k} - \Delta x_{k-1}) - \gamma A \Delta x_{k} - \gamma y_{i_k} = ((1+\alpha)\mathrm{Id} - \gamma A) \Delta x_{k} - \alpha \Delta x_{k-1} - \gamma y_{i_k}.
\end{align}
Note that this recurrence can be viewed as an LTI system that maps the noise sequence $(y_{i_k})_k$ to the recentered iterate sequence $\Delta x_{k}$: linear combinations of noise sequences yield linear combinations of recentered iterates, while time shifts on $(y_{i_k})$ yield shifts on $(\Delta x_{k})_k$ since the non-constant recurrence coefficients are time-independent. While the impulse response $(H_k)_{k \in \mathbb{Z}}$ is tedious to compute, we can directly obtain the transfer function by taking the Z-transform on both sides of \eqref{eq:recentered_SGDM}:
\begin{align}
    X(z) &:= \sum_{k=-\infty}^{\infty} z^{-k} \Delta x_{k} = \sum_{k=-\infty}^{\infty} z^{-k} (((1+\alpha)\mathrm{Id} - \gamma A) \Delta x_{k-1} - \alpha x_{k-2} - \gamma y_{i_{k-1}}) \\ &= ((1+\alpha)\mathrm{Id} - \gamma A) z^{-1} \sum_{k=-\infty}^{\infty} z^{-k} \Delta x_{k} - \alpha z^{-2} \sum_{k=-\infty}^{\infty} z^{-k} \Delta x_{k} - \gamma z^{-1} \sum_{k=-\infty}^{\infty} z^{-k} y_{i_{k}} \\ &= (((1+\alpha)\mathrm{Id} - \gamma A) z^{-1} - \alpha z^{-2} \mathrm{Id}) X(z) - \gamma z^{-1} Y(z).
\end{align}
Rearranging and multiplying denominator and numerator by $z$, we obtain that $X(z) = - (z^2 \mathrm{Id} - ((1+\alpha)\mathrm{Id} - \gamma A) z + \alpha \mathrm{Id})^{-1} \gamma z Y(z)$, which by comparison with $X(z) = H(z) Y(z)$ shows that $H(z) = - (z^2 \mathrm{Id} - z((1+\alpha)\mathrm{Id} - \gamma A) + \alpha \mathrm{Id})^{-1} \gamma z \mathrm{Id}$. 

Setting $\alpha = 0$ above, we obtain that for SGD, $\Delta x_{k+1} = (\mathrm{Id} - \gamma A) \Delta x_{k} - \gamma y_{i_k}$, and $H(z) = - (z \mathrm{Id} - \mathrm{Id} + \gamma A))^{-1} \gamma  \mathrm{Id}$.
A similar reasoning shows that for SNAG, $\Delta x_{k+1} = (1+\alpha)(\mathrm{Id} - \gamma A) \Delta x_{k} - \alpha (\mathrm{Id} - \gamma A)\Delta x_{k-1} - \gamma y_{i_k}$, and $H(z) = -(z^2 \mathrm{Id} - z (1 + \alpha) (\mathrm{Id} - \gamma A) + \alpha (\mathrm{Id} - \gamma A))^{-1} z \gamma \mathrm{Id}$.

\textbf{Autocorrelation and power spectral density of the noise sequence.} When there is replacement, note that for $k \neq 0$, $y_{i_{\kappa}}$ and $y_{i_{\kappa+k}}$ are independent uniform random variables over $\{1, \dots, n\}$, and thus $R_y(k) = R_y(\kappa, \kappa + k) = \mathbb{E}[y_{i_{\kappa}} y_{i_{\kappa + k}}^{\top}] = ( \frac{1}{n} \sum_{i=1}^{n} y_{i} ) (\frac{1}{n} \sum_{i=1}^{n} y_{i})^{\top} = 0$,
where we used that $\frac{1}{n} \sum_{i=1}^n y_i = 0$. When $k=0$,
\begin{align} \label{eq:autocorr_y_SGD}
    R_y(0) &= R_y(\kappa, \kappa) = \mathbb{E}[y_{i_{\kappa}} y_{i_{\kappa}}^{\top}] = \frac{1}{n} \sum_{i=1}^{n} \left( - A_i A^{-1} b + b_i \right) \left( - A_i A^{-1} b + b_i \right)^{\top} := \Sigma.
\end{align}
Hence, the autocorrelation of the noise reads $R_y(k) = \mathds{1}_{k = 0} \Sigma$, which is clearly absolutely summable. From definition \eqref{eq:S_fourier_R}, we obtain that $S_y(k) = \Sigma$.

\textbf{Variance of the iterates.} We want to apply \autoref{thm:LTI_PSD} and then use equation \eqref{eq:variance_expression} to compute $\mathbb{E}[\|\Delta x_{k}\|^2]$. The only assumption left to check is that the impulse response $(H_k)_k$ is absolutely summable. Although we have not computed $H_k$ explicitly, we make use of \autoref{lem:continuously_differentiable} in \autoref{app:proofs_replacement}, which proves that $(H_k)_k$ is absolutely summable as long as $f \to H(e^{2\pi i f})$ is continuously differentiable. For SGDM (analogous for the other algorithms), this boils down to ensuring that for complex $z$ of modulus 1, the matrix $z^2 \mathrm{Id} - ((1+\alpha)\mathrm{Id} - \gamma A) z + \alpha \mathrm{Id}$ has full rank, which is the case (see \autoref{lem:eigenvalues_not_zero} in \autoref{app:proofs_replacement}). Thus, \autoref{thm:LTI_PSD} is applicable; we obtain that the recentered iterate sequence $(\Delta x_k)_k$ is a wide-sense stationary process, and that $S_{\Delta x} = H(e^{-2\pi i f}) \Sigma H(e^{2\pi i f})^{\top}$. By \eqref{eq:variance_expression},
\begin{align}
    \mathbb{E}[\|\Delta x_{k}\|^2] = \int_{0}^{1} \mathrm{Tr}[H(e^{-2\pi i f}) \Sigma H(e^{2\pi i f})^{\top}] \, df.  
\end{align}
We can easily see that $\mathbb{E}[\Delta x_k] = 0$: since the expectation is constant over time and $\mathbb{E}[y_{i_k}] = 0$, equation \eqref{eq:recentered_SGDM} implies that $\mathbb{E}[\Delta x_{k}] = ((1+\alpha)\mathrm{Id} - \gamma A) \mathbb{E}[\Delta x_{k}] - \alpha \mathbb{E}[\Delta x_{k}] = 0 \implies \mathbb{E}[\Delta x_{k}] = 0$.
Finally, since $(x_k)_k$ is a translation of $(\Delta x_k)_k$, we get that the trace of the variance of $x_k$, i.e. the sum of the variances of each component of the $x_k$ is equal to:
\begin{align} 
\begin{split} \label{eq:trace_integral}
    \mathrm{Tr}[\mathrm{Var}(x_k)] &= \mathrm{Tr}[\mathrm{Var}(\Delta x_k)] = \mathrm{Tr}[\mathbb{E}[\Delta x_{k} \Delta x_{k}^{\top}] - \mathbb{E}[\Delta x_{k}] \mathbb{E}[\Delta x_{k}]^{\top}] = \mathbb{E}[\|\Delta x_{k}\|^2] \\ &= \int_{0}^{1} \mathrm{Tr}[H(e^{-2\pi i f}) \Sigma H(e^{2\pi i f})^{\top}] \, df
\end{split}
\end{align}

\textbf{Analytic expression of the variance.} Next, we show that the integral in the right-hand side of \eqref{eq:trace_integral} admits an analytic expression in terms of the spectrum of $A$, which diagonalizes in an orthonormal basis and has positive eigenvalues since it is symmetric and positive definite. Let $A = U \Lambda U^{\top}$ be its eigendecomposition. For SGDM (SNAG is analogous), $H(z) = -(z^2 \mathrm{Id} - ((1+\alpha)\mathrm{Id} - \gamma A) z + \alpha \mathrm{Id})^{-1} z \gamma \mathrm{Id}$ also diagonalizes in the basis $U$, and in particular we have that $U^{\top} H(z) = - (z^2 \mathrm{Id} - ((1+\alpha)\mathrm{Id} - \gamma \Lambda) z + \alpha \mathrm{Id})^{-1} z \gamma U^{\top}$. Consequently, the invariance of the trace to changes of basis implies that
\begin{align}
\begin{split} \label{eq:tr_decomposition}
    \mathrm{Tr}[H(e^{-2\pi i f}) &\Sigma H(e^{2\pi i f})^{\top}] = \mathrm{Tr}[U^{\top} H(e^{-2\pi i f}) \Sigma H(e^{2\pi i f})^{\top} U] 
    \\ = \gamma^2 \mathrm{Tr}[&(e^{-4\pi i f} \mathrm{Id} - e^{-2\pi i f} ((1 + \alpha)\mathrm{Id} - \gamma \Lambda) + \alpha \mathrm{Id})^{-1} U^{\top} \Sigma U \cdot \\  &(e^{4\pi i f} \mathrm{Id} - e^{2\pi i f} ((1 + \alpha)\mathrm{Id} - \gamma \Lambda) + \alpha \mathrm{Id})^{-1}] \\ = \gamma^2 \sum_{i=1}^{d} &(u_i^{\top} \Sigma u_i) |e^{-4\pi i f} - e^{-2\pi i f} (1 + \alpha - \gamma \lambda_i) + \alpha |^{-2}, 
\end{split}
\end{align}
which means that $\mathrm{Tr}[\mathrm{Var}(x_k)] = \gamma^2 \sum_{i=1}^{d} (u_i^{\top} \Sigma u_i) \int_{0}^{1} |e^{-4\pi i f} - e^{-2\pi i f} (1 + \alpha - \gamma \lambda_i) + \alpha|^{-2} \, df$. The following lemma, proved in \autoref{app:proofs_replacement}, shows that these integrals admit an exact expression:
\begin{proposition} \label{prop:integral}
    Suppose that $(1-\alpha)^2 + \gamma^2 \lambda^2 - 2 (1 + \alpha) \gamma \lambda \geq 0$. Define 
    \begin{align} \label{eq:rho_sgdm_def}
        \rho_{\pm} = \frac{2(1-\alpha)^2+ \gamma^2 \lambda^2 - 2(1 + \alpha)\gamma \lambda \pm 2 (1-\alpha) \sqrt{(1-\alpha)^2+ \gamma^2 \lambda^2 - 2 (1 + \alpha) \gamma \lambda}}{\gamma^2 \lambda^2},
    \end{align}
    which are both non-negative.
    Then, for all 
    $0 < \epsilon_0 \leq \epsilon_1 < 1$, $\int_{\epsilon_0}^{\epsilon_1} |e^{-4\pi i f} - e^{-2\pi i f} (1 + \alpha - \gamma \lambda) + \alpha|^{-2} \, df$ is equal to
    \begin{align} \label{eq:integral_expression}
        \frac{\left[ \frac{\rho_{+} - 1}{\sqrt{\rho_{+}}} \arctan \left(\frac{\tan(\frac{x}{2})}{\sqrt{\rho_{+}}} \right) + \frac{1 - \rho_{-}}{\sqrt{\rho_{-}}} \arctan \left(\frac{\tan(\frac{x}{2})}{\sqrt{\rho_{-}}} \right) \right]_{-\pi+2\pi \epsilon_0}^{-\pi + 2\pi \epsilon_1}
        }{4 \pi (1-\alpha) \sqrt{(1-\alpha)^2+ \gamma^2 \lambda^2 - 2 (1 + \alpha) \gamma \lambda}}
    \end{align}
\end{proposition}
Since $\lim_{x \to \pm\pi} \arctan \left(\frac{\tan(\frac{x}{2})}{\sqrt{\rho_{+}}} \right) = \lim_{x \to +\infty} \arctan \left( \pm x \right) = \pm \frac{\pi}{2}$, Proposition \ref{prop:integral} implies that if $(1-\alpha)^2 + \gamma^2 \lambda_i^2 - 2 (1 + \alpha) \gamma \lambda_i \geq 0$ for all $i \in \{1,\dots,d\}$, the variance for the SGDM iterates with replacement is
\begin{align} \label{eq:variance_sgdm_replace}
    \mathrm{Tr}[\mathrm{Var}(x_k)] = \gamma^2 \sum_{i=1}^{d} (u_i^{\top} \Sigma u_i) \frac{\sqrt{\rho_{+}^{i}} - \sqrt{\rho_{-}^{i}} - \frac{1}{\sqrt{\rho_{+}^{i}}} + \frac{1}{\sqrt{\rho_{-}^{i}}}}{4 (1-\alpha) \sqrt{(1-\alpha)^2+ \gamma^2 \lambda_i^2 - 2 (1 + \alpha) \gamma \lambda_i}},
\end{align}
where we defined $\rho_{\pm}^{i}$ as $\rho_{\pm}$ in \eqref{eq:rho_sgdm_def} with the choice $\lambda = \lambda_i$. 

For SGD, which corresponds to $\alpha = 0$, we have that $1 + \gamma^2 \lambda^2 - 2 \gamma \lambda = (1-\gamma \lambda)^2 \geq 0$, which means that Proposition \ref{prop:integral} holds for any values of $\gamma, \lambda > 0$. As shown in \autoref{lem:sgd_integral} (\autoref{app:proofs_replacement}), the expressions become simpler: 
\begin{align} \label{eq:variance_SGD_it}
    \rho_{+}^{i} = \frac{4(1-\gamma \lambda_i) +\gamma^2 \lambda_i^2}{\gamma^2 \lambda_i^2}, \quad \rho_{-}^{i} = 1, \quad \mathrm{Tr}[\mathrm{Var}(x_k)] = \gamma \sum_{i=1}^{d} (u_i^{\top} \Sigma u_i) \frac{1}{\lambda_i (2-\gamma \lambda_i)}.
\end{align}
To get an idea of the dependency of the variance \eqref{eq:variance_sgdm_replace} on the momentum parameter $\alpha$, it is convenient to study the regime in which $\gamma \lambda_i \ll 1-\alpha$ for all $i \in \{1,\dots,d\}$. As shown in \autoref{lem:approx_sgdm} (\autoref{app:proofs_replacement}), we obtain: 
\begin{align} 
\begin{split} \label{eq:replacement_limit}
    \rho_{+}^{i} &= \frac{4(1-\alpha)^2}{\gamma^2 \lambda_i^2} + O\left(\frac{1}{\gamma \lambda_i} \right), \quad \rho_{-} = \left(\frac{1+\alpha}{1-\alpha} \right)^{2} + O(\gamma \lambda_{i}), \\ &\text{and} \quad \mathrm{Tr}[\mathrm{Var}(x_k)] = \sum_{i=1}^{d} \frac{\gamma(u_i^{\top} \Sigma u_i)}{2(1-\alpha)\lambda_i} + O\bigg(\sqrt{\frac{\gamma^3}{\lambda_i}} \bigg).
\end{split}
\end{align}
In \autoref{prop:integral2} and \autoref{cor:variance_SNAG} we carry out an analogous program for SNAG with replacement. While the expressions of $\rho_{+}, \rho_{-}$ and the variance for SNAG differs slightly from \eqref{eq:variance_sgdm_replace}, in the limit $\gamma \lambda_i \ll 1-\alpha$ they are equal to \eqref{eq:replacement_limit} as well.

\section{Variance computation for Shuffle-Once stochastic gradient algorithms} \label{sec:SO}

\textbf{Autocorrelation of the noise sequence.} We place ourselves in the setting of \autoref{sec:with_replacement}, except that now we consider the shuffle-once (SO) scheme to sample the gradient estimates (see \autoref{sec:intro}). Under SO, note that for $k \not\equiv 0 \, (\mathrm{mod} \, n)$, $y_{i_{\kappa}}$ and $y_{i_{\kappa+k}}$ are not independent random variables. We have
\begin{align}
\begin{split} \label{eq:SO_correlation_neq_0}
    R_y(\kappa, \kappa + k) &= \mathbb{E}[y_{\kappa} y_{\kappa + k}^{\top}] = 
    \frac{1}{n} \sum_{i=1}^{n} y_i \frac{1}{n-1} \sum_{i' \neq i} y_{i'}^{\top}
    = \frac{1}{n(n-1)} \sum_{i=1}^{n} y_i \sum_{i'=1}^{n} y_i^{\top} - \frac{1}{n(n-1)} \sum_{i=1}^{n} y_i y_i^{\top} 
    \\ &= - \frac{1}{n(n-1)} \sum_{i=1}^{n} \left( - A_i A^{-1} b + b_i \right) \left( - A_{i'} A^{-1} b + b_{i'} \right)^{\top} = -\frac{1}{n-1} \Sigma.
\end{split}    
\end{align}
Here, we used that $\sum_{i=1}^{n} y_i = 0$, and the definition of the matrix $\Sigma$ (see equation \eqref{eq:autocorr_y_SGD}). For $k \equiv 0 \, (\mathrm{mod} \, n)$, we have that $y_{\kappa} = y_{\kappa + k}$ and hence $R_y(k) = \mathbb{E}[y_{\kappa} y_{\kappa}^{\top}]= \Sigma$ by the argument in equation \eqref{eq:autocorr_y_SGD}. Hence, the autocorrelation function takes the form $R_y(k) = ((1+\frac{1}{n-1}) \mathds{1}_{k \equiv 0 \, (\mathrm{mod} \, n)} - \frac{1}{n-1}) \Sigma$. 

\textbf{Analytic expression of the variance.} Note that $R_y$ is not absolutely summable, which means that in this setting, \autoref{thm:LTI_PSD} is not suitable to compute the power spectral density of the sequence $(\Delta x_k)_{k}$. Instead we use \autoref{prop:variance_S_x_SO} in \autoref{app:proofs_sec_SO}, which is a consequence of \autoref{thm:LTI_PSD_2} and \autoref{lem:variance_S_x}, and yields:
\begin{align} \label{eq:mse_SO}
    \mathrm{Tr}[\mathrm{Var}(x_k)] = \mathbb{E}[\|\Delta x_{k}\|^2] = 
    \frac{1}{n-1} \sum_{k'=1}^{n-1} \mathrm{Tr}[\Sigma H(e^{\frac{2\pi i k'}{n}})^{\top} H(e^{-\frac{2\pi i k'}{n}})].
\end{align}
Using \eqref{eq:tr_decomposition}, we obtain that for SGDM,
\begin{align} \label{eq:tr_sgdm_SO}
    \mathrm{Tr}[\mathrm{Var}(x_k)] = \frac{\gamma^2 }{n-1} \sum_{i=1}^{d} (u_i^{\top} \Sigma u_i) \sum_{k'=1}^{n-1} |e^{-\frac{4\pi i k'}{n}} - e^{-\frac{2\pi i k'}{n}} (1 + \alpha - \gamma \lambda_i) + \alpha |^{-2}.
\end{align}
\vspace{-5pt}
\begin{remark}
    Using that the Fourier transforms of the sequence $w_k = \mathds{1}_{k \equiv 0 \, (\mathrm{mod} \, n)}$ and $z_k$ are $W(f) = \frac{1}{n}\sum_{k=-\infty}^{+\infty} \delta(f - \frac{k}{n} )$ and $Z(f) = \delta(f)$, we obtain that the power spectral density of the noise is $S_y(f) = \frac{1}{n-1} \sum_{k=-\infty}^{+\infty} \mathds{1}_{k \neq 0} \delta ( f - \frac{k}{n} )$, and then we can informally derive \eqref{eq:mse_SO} by using \eqref{eq:variance_expression}. This is not formal because $R_y$ is not absolutely summable. 
\end{remark}
\vspace{-2pt}
Note the resemblance of equation \eqref{eq:mse_SO} to equation \eqref{eq:trace_integral}; instead of an integral between 0 and 1, we have an average of equispaced points between $1/n$ and $(n-1)/n$. As expected, in the regime $n \to \infty$ the variance for the SO iterates converges to the variance with replacement. However, when $\gamma \lambda_i n \ll 1-\alpha$ the two expressions have different behaviors, as described in the following proposition. 
\begin{proposition} \label{prop:variance_sgdm_so}
In the regime $n \gg 1$ and $\gamma \lambda_i n \ll 1-\alpha$ for all $i \in \{1,\dots,d\}$, we have that for SGDM-SO and SNAG-SO,
\begin{align} \label{eq:variance_sgdm_so_prop}
    \mathrm{Tr}[\mathrm{Var}(x_k)] = 
    \sum_{i=1}^{d} (u_i^{\top} \Sigma u_i) \frac{2 \gamma^2 n^2 (1 + O(\frac{\gamma \lambda_i}{1-\alpha}) + O(1/n^2))}{\pi^2 (n-1) (1-\alpha)^2}.
\end{align}
\end{proposition}
The proof is in \autoref{app:proofs_sec_SO}. Interestingly, \autoref{prop:variance_sgdm_so} shows that the variance for SO depends quadratically on $\gamma$ and on $1/(1-\alpha)$, and linearly on $n$. In contrast, equation \eqref{eq:replacement_limit} showed a linear dependence on $\gamma$ and on $1/(1-\alpha)$, and no explicit dependence on $n$.

\section{Variance computation for Random Reshuffle stochastic gradient algorithms} \label{sec:RR}


\textbf{Autocorrelation and power spectral density of the noise sequence.} Note that a priori, the noise sequence for RR is not a wide-sense stationary process; e.g. if $\kappa$ and $\kappa'$ belong to the same reshuffle we get that $\mathbb{E}[y_{\kappa} y_{\kappa'}^{\top}] = - \frac{1}{n-1} \Sigma$ by the argument in \eqref{eq:SO_correlation_neq_0}, while if they belong to a different reshuffle, $\mathbb{E}[y_{\kappa} y_{\kappa'}^{\top}] = 0$ because $y_{\kappa}$ and $y_{\kappa'}$ are independent.
To obtain a wide-sense stationary process, we randomize the first iteration of the reshuffle, which of course does not alter the average iterate variance in time. Given a fixed $\kappa$, let $R$ be the random variable that denotes the first iteration of the first reshuffle that takes place after iteration $\kappa$, which takes values uniformly between $\kappa + 1$ and $\kappa + n$. For any $k > 0$, we get
\begin{align}
    R_y(\kappa, \kappa + k) = \mathbb{E}[y_{\kappa} y_{\kappa + k}^{\top}] = \sum_{i=1}^{n} \text{Pr}(R = \kappa + i) \left(\mathds{1}_{k \geq i} \cdot 0 - \mathds{1}_{k < i} \frac{\Sigma}{n-1} \right) = -\frac{1}{n} \sum_{i=1}^{n} \frac{\mathds{1}_{k < i} \Sigma}{n-1}. 
\end{align}
Since $R_y(0) = \Sigma$ and the autocorrelation function is even, we obtain $R_y(k) = ((1+\frac{1}{n-1}) \delta_{k = 0} - \frac{n-|k|}{n(n-1)} \delta_{|k| \leq n} ) \Sigma$.
$R_y$ is absolutely summable, and $S_y$ is given by (see \autoref{lem:PSD_S_y} in \autoref{app:proofs_sec_RR}):
\begin{align} 
\begin{split} \label{eq:PSD_y_RR}
    \forall f \in [0,1], \quad S_y(f) = r_n(2\pi f) \Sigma, \quad
    \text{where} \ r_n(x) = 
    \frac{n}{n-1} - \frac{\frac{\sin((n-\frac{1}{2}) x)}{\sin(x/2)} + \frac{\sin^2(\frac{(n-1)x}{2})}{\sin^2(\frac{x}{2})}}{n(n-1)}
\end{split}
\end{align}
\textbf{Analytic expression of the variance.} Applying equations \eqref{eq:variance_expression} and \eqref{eq:tr_decomposition}, we obtain that for SGDM,
\begin{align} \label{eq:var_RR}
    \mathrm{Tr}[\mathrm{Var}(x_k)] &= \int_{0}^{1} \mathrm{Tr}[H(e^{-2\pi i f}) \Sigma H(e^{2\pi i f})^{\top}] r_n(2\pi f) \, df \\ &= \gamma^2 \sum_{i=1}^{d} (u_i^{\top} \Sigma u_i) \int_{0}^{1} |e^{-4\pi i f} - e^{-2\pi i f} (1 + \alpha - \gamma \lambda_i) + \alpha |^{-2} r_n(2\pi f) \, df.
\end{align}
By taking the Taylor approximation of $r_n(x)$ around zero, we obtain a more informative expression for the variance of the SGDM (and SNAG) iterates:
\begin{proposition} \label{prop:variance_sgdm_RR}
Let $\delta \in (0,1)$ arbitrary. In the regime $n^{1+\delta} \gg 1$ and $\gamma \lambda_i n^{1+\delta} \ll 1-\alpha$ for all $i \in \{1,\dots,d\}$, we have that for SGDM-RR and SNAG-RR,
    \begin{align} \label{eq:var_sgdm_RR}
        \mathrm{Tr}[\mathrm{Var}(x_k)] = 
        \gamma^2 \sum_{i=1}^{d} (u_i^{\top} \Sigma u_i) \frac{  n^{2+\delta}+O\bigg( \frac{\gamma^2 \lambda_i^2 n^{2+\delta}}{(1-\alpha)^2} + n^{2-\delta} + n + \frac{\gamma \lambda_i}{1-\alpha} \bigg) }{\pi^2 (n-1) (1-\alpha)^2}.
    \end{align}
\end{proposition}
Note that the leading term of the expression \eqref{eq:var_sgdm_RR} for RR is formally similar to the leading term of the expression \eqref{eq:variance_sgdm_so_prop} for SO, but it has a worse dependency on $n$; $n^{2+\delta}$ instead of $n^2$. We corroborate this finding experimentally in \autoref{sec:experiments}.


\section{Experiments} \label{sec:experiments}
We consider a simple regression problem with loss $f(x) = \frac{1}{2n} \sum_{i=1}^{n} (\langle z_i, x \rangle - y_i)^2$, where $x, z_i \in \R^d$, $y_i \in \R$. We sample $z_i$ independently from the $d$-dimensional standard Gaussian and we choose $y_i = \langle e_1, z_i \rangle + \epsilon_i$, where $e_1$ is the first vector of the canonical basis, and $\epsilon_i$ are i.i.d. Gaussians with mean 0 and standard deviation 0.1. This loss can written as \eqref{eq:f_form} with $A_i = z_i z_i^{\top}$ and $b_i = z_i y_i$. The code for the experiments can be found at \url{https://github.com/CDEnrich/sgd_shuffling}. 

In \autoref{table:1} we show values of the mean squared errors for SGD, SGDM and SNAG in each of the three shuffling schemes, with $n = 1000$, $d=5$, $\gamma = 0.0005$, $\alpha = 0.8$. We run all the algorithms with the standard stochastic gradient noise and with the zero-th order noise model that we introduce, and we also show the theoretical values as predicted by equation \eqref{eq:trace_integral}. Note that the eigenvalues $\lambda_i$ of $A$ are all of the same order, and so are the values $u_i^{\top} \Sigma u_i$. We observe that in this regime, the errors for the standard noise and the zero-th order noise algorithms are very close; within small multiples of the standard deviations for all algorithms except for SGDM-SO and SGDM-RR. This means that the zero-th order noise model is a great proxy for the standard noise. The theoretical errors fall within small multiples of the standard deviations, which confirms the correctness of our framework. Importantly, remark that the errors under SO are about half those of RR for the three algorithms, which is consistent with the leading term of \eqref{eq:variance_sgdm_so_prop} being smaller than \eqref{eq:var_sgdm_RR}. \autoref{fig:table1_1} plots the squared distance between the iterate and the optimum over runs of SGD and SGDM for the three shuffling schemes, and shows that the sequences obtained with both noise models are qualitatively similar.

In \autoref{sec:more_experiments} we show tables and figures comparing the mean squared errors for different values of $n$, $\gamma$ and $\alpha$. To sum up, when $u_i^{\top} \Sigma u_i$ are of different orders, the zero-th order noise model is still valid for algorithms with replacement, but worse for RR and SO in this order.

\begin{table}[]
\small
\centering
\begin{tabular}{|l|l|l|l|}
\hline
Algorithm & Full noise & 0th order noise & Theory  \\ \hline
SGD & $\num{1.2786e-05} \pm \num{1.0661e-07}$ & $\num{1.2805e-05} \pm \num{1.0986e-07}$ & $\num{1.2787e-05}$  \\ \hline
SGD-RR & $\num{4.1589e-07} \pm \num{7.5799e-10}$ & $\num{4.1461e-07} \pm \num{5.2744e-10}$ & $\num{4.1518e-07}$ \\ \hline
SGD-SO & $\num{2.1188e-07} \pm \num{3.1052e-09}$ & $\num{2.1344e-07} \pm \num{3.0547e-09}$ & $\num{2.1278e-07}$ \\ \hline
SGDM & $\num{6.4356e-05} \pm \num{2.2884e-07}$ & $\num{6.4173e-05} \pm \num{3.6560e-07}$ & $\num{6.3935e-05}$
\\ \hline
SGDM-RR & $\num{9.3752e-06} \pm \num{1.6149e-08}$ &  $\num{9.3175e-06} \pm \num{1.0781e-08}$ & $\num{9.3300e-06}$
\\ \hline
SGDM-SO & $\num{5.1690e-06} \pm \num{7.7684e-08}$ &  $\num{5.1967e-06} \pm \num{7.6184e-08}$ & $\num{5.1800e-06}$ 
\\ \hline
SNAG & $\num{6.4346e-05} \pm \num{2.1749e-07}$ & $\num{6.4123e-05} \pm \num{3.6823e-07}$ & $\num{6.3909e-05}$
\\ \hline
SNAG-RR & $\num{9.3684e-06} \pm \num{1.5707e-08}$ &  $\num{9.3101e-06} \pm \num{1.0771e-08}$ & $\num{9.3230e-06}$
\\ \hline
SNAG-SO & $\num{5.1651e-06} \pm \num{7.7623e-08}$ & $\num{5.1926e-06} \pm \num{7.6124e-08}$ & $\num{5.1760e-06
}$
\\ \hline
\end{tabular}
\normalsize
\vspace{3pt}
\caption{Mean squared errors $\mathbb{E}[\|x_k - x^{\star}\|^2]$ for $n = 1000$, $d = 5$;
$\gamma = 0.0005$, $\alpha = 0.8$ (seed 38). The estimates and and their standard deviations are computed over 10 runs of \num{6e6} iterations each. The eigenvalues $\lambda_i$ are $0.1807$, $0.1951$, $0.1998$, $0.2033$, $0.2194$. The values $u_i^{\top} \Sigma u_i$ are, in order, $0.0019$, $0.0019$, $0.0022$, $0.0020$, $0.0022$. The theoretical errors for SGDM and SNAG given by 
\eqref{eq:replacement_limit} and \eqref{eq:replacement_limit_snag} are $\num{6.3933e-05}$.}
\label{table:1}
\end{table}

\begin{figure}[H]
    \centering
    \includegraphics[width=0.95\textwidth]{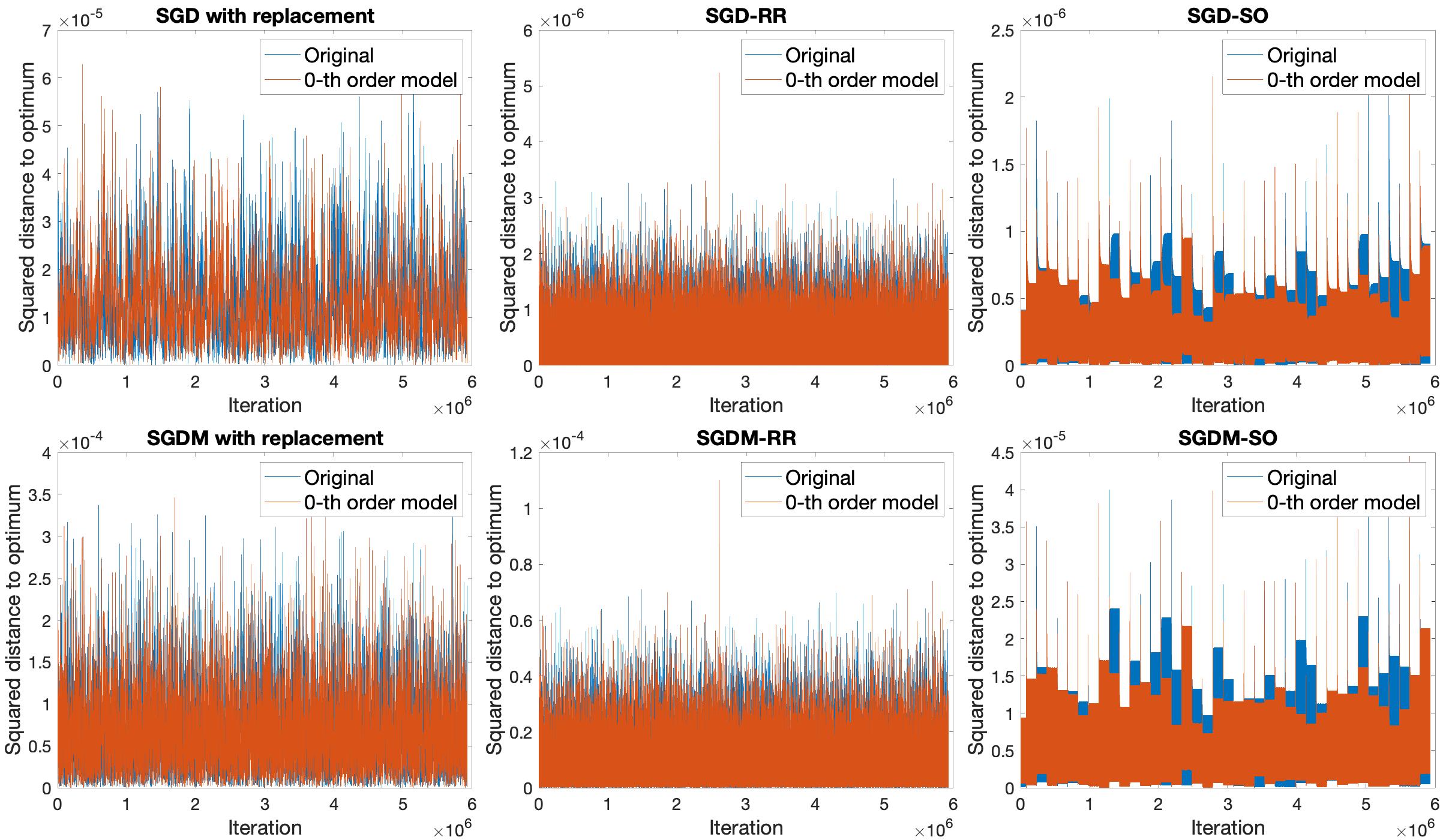}
    \caption{Plots for SGD and SGDM with the three shuffling schemes, in the setting of \autoref{table:1}.}
    \label{fig:table1_1}
\end{figure}




\bibliography{biblio}

\newpage

\appendix

\tableofcontents

\addtocontents{toc}{\protect\setcounter{tocdepth}{2}}

\section{Definition of the power spectral density for general autocorrelations} \label{app:general_def}

We treat the case $d=1$ first for simplicity. Given an autocorrelation function $R_y(\kappa) := \mathbb{E}[y_{k} y_{k+\kappa}]$, the construction is as follows: we consider the tempered distribution $\mathcal{R}_y \in \mathcal{S}'(\R)$ defined as
\begin{align}
    \mathcal{R}_y = \sum_{k=-\infty}^{\infty} R_y(k) \delta_{k}.
\end{align}
The Fourier transform of $\mathcal{R}_y$ is, by definition, the only element $\hat{\mathcal{R}}_y$ of $\mathcal{S}'(\R)$ such that for any Schwartz function $\phi \in \mathcal{S}(\R)$, 
\begin{align}
    \langle \hat{\mathcal{R}}_y, \phi \rangle = \langle \mathcal{R}_y, \hat{\phi} \rangle = \sum_{k=-\infty}^{\infty} R_y(k) \hat{\phi}(k) = \sum_{k=-\infty}^{\infty} R_y(k) \int_{\R} \phi(f) e^{-2\pi i k f} \, df.
\end{align}
We define the power spectral density as the tempered distribution $S_y = \hat{\mathcal{R}}_y$.
In the case $d > 1$, we define $\mathcal{S}_{(d \times d)}(\R) = \{ \phi = (\phi_{i,j})_{i,j=1}^{d} \, | \, \phi_{i,j} \in \mathcal{S}(\R) \}$, and analogously, we define $S_y$ as the only tempered distribution in $\mathcal{S}_{(d \times d)}'(\R)$ such that for any $\phi \in \mathcal{S}_{(d \times d)}(\R)$,
\begin{align}
    \langle S_y, \phi \rangle = \sum_{i,j=1}^{d} \sum_{k=-\infty}^{\infty} [R_y(k)]_{i,j} \int_{\R} \phi_{i,j}(f) e^{-2\pi i k f} \, df = \sum_{k=-\infty}^{\infty} \left\langle R_y(k), \int_{\R} \phi(f) e^{-2\pi i k f} \, df \right\rangle_{F},
\end{align}
where $\langle \cdot, \cdot \rangle_{F}$ denotes the Frobenius inner product of matrices. When $(R_y(k))_k$ is absolutely summable, we can apply the dominated convergence theorem to show that $\sum_{k=-\infty}^{\infty} \langle R_y(k), \int_{\R} \phi(f) e^{2\pi i k f} \, df\rangle_F  = \int_{\R} \langle (\sum_{k=-\infty}^{\infty} e^{2\pi i k f} R_y(k)), \phi(f) \rangle_F \, df$, and if we compare with $\langle S_y, \phi \rangle = \int_{\R} \langle S_y(f), \phi(f) \rangle_F \, df$ we recover the definition \eqref{eq:S_fourier_R}.

Note that $S_y$ has period 1, in the sense that for any $k' \in \mathbb{Z}$, $\langle S_y(f), \phi(f + k') \rangle = \sum_{k=-\infty}^{\infty} \langle R_y(k), \int_{\R} \phi(f+k') e^{2\pi i k f} \, df \rangle_F = \sum_{k=-\infty}^{\infty} \langle R_y(k), \int_{\R} \phi(f) e^{2\pi i k (f-k')} \, df \rangle_F = \langle S_y(f), \phi(f) \rangle$. This allows us to view $S_y$ as an element of the dual of the space $C_{\mathrm{per},(d \times d)}^{\infty}([0,1]) = \{ \phi = (\phi_{i,j})_{i,j=1}^{d} \, | \, \phi_{i,j} \in C_{\mathrm{per}}^{\infty}([0,1]) \}$ of matrix-valued infinitely differentiable functions on $[0,1]$ with periodic boundary conditions. 

\section{Proofs of \autoref{sec:preliminaries}} \label{app:proof_thm_1}

We show a generalization of \autoref{thm:LTI_PSD} in which the autocorrelation function of $R_b$ is not assumed to be absolutely summable.

\begin{theorem} \label{thm:LTI_PSD_2}
Consider the setting of \autoref{thm:LTI_PSD}, with the exception of the absolute summability of the autocorrelation $R_b$ (we do not assume it holds). Then, the output $x$ is a wide-sense stationary process with power spectral density $S_x \in \mathcal{S}_{(d \times d)}'(\R)$ given by
\begin{align} \label{eq:S_x_tempered}
    \forall \phi \in \mathcal{S}_{(d \times d)}(\R), \quad \langle S_x, \phi \rangle = \sum_{k=-\infty}^{\infty} \int_{\R} \langle H(e^{-2\pi i f}) R_b(k) H(e^{2\pi i f})^{\top}, \phi(f) \rangle_{F} \, e^{-2\pi i k f} \, df.
\end{align}
When $(R_b(k))_{k}$ is absolutely summable, the conclusion of \autoref{thm:LTI_PSD} follows from \eqref{eq:S_x_tempered}.
\end{theorem}
\begin{proof}
First, note that
\begin{align} 
\begin{split} \label{eq:autocorrel_x_d}
    R_x(\kappa, \kappa + k) &= \mathbb{E}[b_{\kappa} b_{\kappa + k}^{\top}] = \mathbb{E} \bigg[\sum_{k'=-\infty}^{\infty} \sum_{k''=-\infty}^{\infty} H_{k'} b_{\kappa - k'} (H_{k''} b_{\kappa + k - k''})^{\top} \bigg] \\ &= \sum_{k'=-\infty}^{\infty} \sum_{k''=-\infty}^{\infty} H_{k'} \mathbb{E}[b_{\kappa - k'} b_{\kappa + k - k''}^{\top}] H_{k''}^{\top} = \sum_{k'=-\infty}^{\infty} \sum_{k''=-\infty}^{\infty} H_{k'} R_b(k + k' - k'') H_{k''}^{\top}.
\end{split}
\end{align}
In the third equality we used Fubini's theorem to exchange the expectation with the infinite summations; we are allowed to do it because 
\begin{align} 
\begin{split} \label{eq:autocorrel_x_d_2}
&\sum_{k'=-\infty}^{\infty} \sum_{k''=-\infty}^{\infty} \|H_{k'} \mathbb{E}[b_{\kappa - k'} b_{\kappa + k - k''}^{\top}] H_{k''}^{\top}\|_F \\ &\leq \sum_{k'=-\infty}^{\infty} \sum_{k''=-\infty}^{\infty} \|H_{k'}\|_F \|R_b(k + k' - k'')\|_F \|H_{k''}^{\top}\|_F \\ &\leq \sup_{k} \|R_b(k)\|_F \bigg(\sum_{k'=-\infty}^{\infty} \|H_{k'}\|_F \bigg)^2 < +\infty.
\end{split}
\end{align}
In the last equality of \eqref{eq:autocorrel_x_d} as well as in \eqref{eq:autocorrel_x_d_2} we used that since $b$ is wide-sense stationary, $\mathbb{E}[b_{\kappa - k'} b_{\kappa + k - k''}^{\top}] = \mathbb{E}[b_{\kappa} b_{\kappa + k + k' - k''}^{\top}] = R_b(k + k' - k'')$. In the first equality of \eqref{eq:autocorrel_x_d_2} we also used that the Frobenius norm is multiplicative, and in the last one we used that $(H_k)_{k \in \mathbb{Z}}$ is absolutely summable, and that $\sup_{k} \|R_b(k)\|_F \leq \sup_{k} \mathbb{E}[\|b_{\kappa} b_{\kappa + k}^{\top} \|_F] \sup_{k} \mathbb{E}[\|b_{\kappa}\| \|b_{\kappa + k}\|] \leq \sup_{k} (\mathbb{E}[\|b_{\kappa}\|^2])^{1/2} (\mathbb{E}[\|b_{\kappa + k}^{\top} \|^2])^{1/2} = \mathrm{Tr}[R_b(0)]$.

Note that $R_x(\kappa, \kappa + k)$ is finite because \eqref{eq:autocorrel_x_d_2} provides an upper bound on its Frobenius norm.
Also, \eqref{eq:autocorrel_x_d} implies that $R_x(\kappa, \kappa + k)$ depends only on the time difference $k$, which means that we can define $R_x(k) = R_x(\kappa, \kappa + k)$. To prove that $x$ is wide-sense stationary, we also check that its mean is constant: $\mathbb{E}[x_k] = \mathbb{E}[\sum_{\kappa=-\infty}^{\infty} H_{\kappa} b_{k-\kappa}] = \sum_{\kappa=-\infty}^{\infty} H_{\kappa} \mathbb{E}[b_{k-\kappa}] = (\sum_{\kappa=-\infty}^{\infty} H_{\kappa}) \mathbb{E}[b_{0}] = H(1) \mathbb{E}[b_{0}]$, and that its second-order moment is finite: by \eqref{eq:autocorrel_x_d},
\begin{align}
\begin{split}
&\mathbb{E}[\|x_k\|^2] = \mathbb{E}[\mathrm{Tr}(x_k x_k^{\top})] = \mathrm{Tr}[R_x(k,k)] = \sum_{k'=-\infty}^{\infty} \sum_{k''=-\infty}^{\infty} \mathrm{Tr}[H_{k'} R_b(k' - k'') H_{k''}^{\top}] \\ &\leq \sup_{\kappa \in \mathbb{Z}} \{ \|R_b(\kappa)\|_2 \} \mathrm{Tr}\bigg[ \bigg(\sum_{k'=-\infty}^{\infty} H_{k'}\bigg) \bigg(\sum_{k''=-\infty}^{\infty} H_{k''}\bigg)^{\top}\bigg] = \sup_{\kappa \in \mathbb{Z}} \{ \|R_b(\kappa)\|_2 \} \mathrm{Tr}[H(1) H(1)^{\top}].
\end{split}
\end{align}
For any $\phi \in \mathcal{S}_{(d \times d)}(\R)$,
\begin{align}
\begin{split}
    \langle S_x, \phi \rangle &= \sum_{k=-\infty}^{\infty} \left\langle R_x(k), \int_{\R} \phi(f) e^{-2\pi i k f} \, df \right\rangle_{F} \\ &= \sum_{k=-\infty}^{\infty} \left\langle \sum_{k'=-\infty}^{\infty} \sum_{k''=-\infty}^{\infty} H_{k'} R_b(k + k' - k'') H_{k''}^{\top}, \int_{\R} \phi(f) e^{-2\pi i k f} \, df \right\rangle_{F} \\ &= \sum_{k=-\infty}^{\infty} \sum_{k'=-\infty}^{\infty} \sum_{k''=-\infty}^{\infty} \mathrm{Tr}\bigg[H_{k'} R_b(k + k' - k'') H_{k''}^{\top} \int_{\R} \phi(f)^{\top} e^{-2\pi i k f} \, df \bigg] \\ &= \sum_{k=-\infty}^{\infty} \sum_{k'=-\infty}^{\infty} \sum_{k''=-\infty}^{\infty} \mathrm{Tr}\bigg[ \int_{\R} H_{k''}^{\top} \phi(f)^{\top} H_{k'} e^{-2\pi i k f} \, df R_b(k + k' - k'') \bigg]
    \\ &= \sum_{k=-\infty}^{\infty} \sum_{k'=-\infty}^{\infty} \sum_{k''=-\infty}^{\infty} \mathrm{Tr}\bigg[ \int_{\R} H_{k''}^{\top} \phi(f)^{\top} H_{k'} e^{-2\pi i (k - k' + k'') f} \, df R_b(k) \bigg]
\end{split}
\end{align}
At this point, we exchange the order of integral and the infinite summations w.r.t. $k'$ and $k''$, appealing to Fubini's theorem and the fact that $(H_k)_k$ is absolutely summable. We obtain:
\begin{align}
\begin{split}
    &\sum_{k=-\infty}^{\infty} \mathrm{Tr}\bigg[ \int_{\R} \bigg(\sum_{k''=-\infty}^{\infty} e^{-2\pi i k'' f} H_{k''}^{\top} \bigg) \phi(f)^{\top} \bigg(\sum_{k'=-\infty}^{\infty} H_{k'} e^{2\pi i k' f} \bigg) e^{-2\pi i k f} \, df R_b(k) \bigg] \\ &= \sum_{k=-\infty}^{\infty} \mathrm{Tr}\bigg[ \int_{\R} H(e^{2\pi i f})^{\top} \phi(f)^{\top} H(e^{-2\pi i f}) e^{-2\pi i k f} \, df R_b(k) \bigg] \\ &= \sum_{k=-\infty}^{\infty} \mathrm{Tr}\bigg[ \int_{\R} H(e^{-2\pi i f}) R_b(k) H(e^{2\pi i f})^{\top} \phi(f)^{\top} e^{-2\pi i k f} \, df \bigg] \\ &= \sum_{k=-\infty}^{\infty} \int_{\R} \langle H(e^{-2\pi i f}) R_b(k) H(e^{2\pi i f})^{\top}, \phi(f) \rangle_{F} \, e^{-2\pi i k f} \, df
\end{split}
\end{align}
In the third equality we used that $\langle A, B \rangle = \mathrm{Tr}[A^\top B]$, and in the fourth equality we used that $\mathrm{Tr}[A B] = \mathrm{Tr}[B A]$ for $A,B$ of the appropriate dimensions. Note that when $(R_b(k))_k$ is absolutely summable, the right-hand side is equal to 
\begin{align}
\begin{split}
    &\int_{\R} \left\langle H(e^{-2\pi i f}) \bigg( \sum_{k=-\infty}^{\infty} e^{-2\pi i k f} R_b(k) \bigg) H(e^{2\pi i f})^{\top}, \phi(f) \right\rangle_{F} \, df \\ &= \int_{\R} \langle H(e^{-2\pi i f}) S_b(f) H(e^{2\pi i f})^{\top}, \phi(f) \rangle_{F} \, df,
\end{split}
\end{align}
which means that the power spectral density $S_x$ is a function: $S_x(f) = H(e^{-2\pi i f}) S_b(f) H(e^{2\pi i f})^{\top}$.
\end{proof}

When $(R_b(k))_k$ is absolutely summable, we have
\begin{align}
    \mathbb{E}[\|x_{k}\|^2] = \mathrm{Tr}[R_x(0)] = \int_{0}^{1} \mathrm{Tr}[S_x(f)] \, df = \int_{0}^{1} \mathrm{Tr}[H(e^{-2\pi i f}) S_b(f) H(e^{2\pi i f})^{\top}] \, df. 
\end{align}
\begin{lemma} \label{lem:variance_S_x}
Consider the setting of \autoref{thm:LTI_PSD}, with the exception of the absolute summability of the autocorrelation $R_b$ (we do not assume it holds). We have that
\begin{align}
    \mathbb{E}[\|x_{k}\|^2] = \sum_{k=-\infty}^{\infty} \int_{0}^{1} \mathrm{Tr}[R_b(k) H(e^{2\pi i f})^{\top} H(e^{-2\pi i f})] \, e^{-2\pi i k f} \, df.
\end{align}
\end{lemma}
\begin{proof}
We view $S_x$ as an element of the dual space of $C_{\mathrm{per},(d \times d)}^{\infty}([0,1])$ (see \autoref{app:general_def}). If we set $\phi \equiv \mathrm{Id} \in C_{\mathrm{per},(d \times d)}^{\infty}([0,1])$, we obtain that by the definition of the power spectral density $S_x$,
\begin{align}
    \langle S_x, \mathrm{Id} \rangle &= \sum_{k=-\infty}^{\infty} \left\langle R_x(k), \mathrm{Id} \int_{\R} e^{-2\pi i k f} \, df \right\rangle_{F} = \sum_{k=-\infty}^{\infty} \mathrm{Tr}[R_x(k)] \int_{0}^1 e^{-2\pi i k f} \, df \\ &= \sum_{k=-\infty}^{\infty} \mathds{1}_{k = 0} \mathrm{Tr}[R_x(k)] = \mathrm{Tr}[R_x(0)] = \mathbb{E}[\mathrm{Tr}[x_k x_k^{\top}]] = \mathbb{E}[\|x_{k}\|^2]. 
\end{align}
On the other hand, by \autoref{thm:LTI_PSD_2},
\begin{align}
    \langle S_x, \mathrm{Id} \rangle &= \sum_{k=-\infty}^{\infty} \int_{0}^{1} \langle H(e^{-2\pi i f}) R_b(k) H(e^{2\pi i f})^{\top}, \mathrm{Id} \rangle_{F} \, e^{-2\pi i k f} \, df \\ &= \sum_{k=-\infty}^{\infty} \int_{0}^{1} \mathrm{Tr}[H(e^{-2\pi i f}) R_b(k) H(e^{2\pi i f})^{\top}] \, e^{-2\pi i k f} \, df \\ &= \sum_{k=-\infty}^{\infty} \int_{0}^{1} \mathrm{Tr}[R_b(k) H(e^{2\pi i f})^{\top} H(e^{-2\pi i f})] \, e^{-2\pi i k f} \, df
\end{align}
\end{proof}

\section{Consistency of the zero-th order noise model} \label{app:validity}

In \autoref{sec:intro} and \autoref{sec:with_replacement} we define the zero-th order approximation $\widehat{\nabla f}_i(x) = \nabla f_i(x^{\star}) + A(x - x^{\star})$, where $x^{\star} = - A^{-1} b$ is the minimizer of $f$, which gives rise to the approximate dynamics $x_{k+1} = x_{k} + \alpha(x_{k} - x_{k-1}) - \gamma (A x_{k} - A_i A^{-1} b + b_i + b)$ (for SGDM). A first question that arises is whether the approximation around $x^{\star}$ is consistent: is the expectation $\mathbb{E}[x_k]$ equal to $x^{\star}$? If that were not the case, our model would be ill-posed because we cannot expect $\widehat{\nabla f}_i(x) \approx \nabla f_i(x)$ to hold unless $x \approx x^{\star}$. A second question is whether the algorithmic choices that we make for $x^{\star}$ and the matrix $A$ are warranted, or if alternative choices are possible. 

We answer both questions at once. If we set $x_{k+1} = x_{k} + \alpha(x_{k} - x_{k-1}) - \gamma (B (x_{k} - x^{\star}) + A_{i_k} x^{\star} + b_{i_k})$ for arbitrary $x^{\star} \in \R^d$ and $B \in \R^{d \times d}$, we have that
\begin{align}
    &\mathbb{E}[x_{k+1}] = \mathbb{E}[x_{k}] + \alpha(\mathbb{E}[x_{k}] - \mathbb{E}[x_{k-1}]) - \gamma (B (\mathbb{E}[x_{k}] - x^{\star}) + \mathbb{E}[A_{i_k}] x^{\star} + \mathbb{E}[b_{i_k}]) \\
    &\implies 0 = B (\mathbb{E}[x_{k}] - x^{\star}) + A x^{\star} + b \implies \mathbb{E}[x_{k}] = x^{\star} -B^{-1} (A x^{\star} + b).
\end{align}
Imposing that $\mathbb{E}[x_{k}] = x^{\star}$ for consistency, we obtain that $B^{-1} (A x^{\star} + b) = 0$ and hence $x^{\star}$ is forced to take value $- A^{-1} b$. However, note that this argument does not constrain the matrix $B$. Hence, a choice of $B$ different from $A$ would also produce a consistent model in this sense.

Beyond imposing $\mathbb{E}[x_k] = x^{\star}$, for our model to be consistent we also need that the iterates $x_k$ are actually close to $x^{\star}$ when the stepsize $\gamma$ is small.
We can shed light onto this issue by looking at the analysis of the Robbins-Monro method. Consider the optimization problem $\min _{x \in \mathcal{X}} \mathbb{E} [Q(x,Y)]$,
where $g(x)=\mathbb{E}[Q(x,Y)]$ is differentiable and convex, and the expectation is with respect to the random variable $Y$. This problem is equivalent to finding the root $x^{\star}$ of 
$\nabla g(x)=0$. Let $(Y_k)_{k\geq 0}$ be a sequence of independent random variables with the same distribution as $y$. The Robbins-Monro method generates a sequence $(x_k)_{k \geq 0}$ as $x_{k+1} = x_{k} - \gamma_k H(x_k,Y_k)$, where $\mathbb{E}[H(x,Y)] = \nabla g(x)$ and $(\gamma_k)_{k \geq 0}$ is a non-negative sequence of stepsizes.  
\begin{lemma}[\cite{bouleau1994numerical}] \label{lem:robbins_monro}
Suppose that (i) $\sum_{k = 0}^{\infty} \gamma_k = +\infty$, (ii) $\sum_{k = 0}^{\infty} \gamma_k^2 < +\infty$, (iii) for any $x$, there exists $B > 0$ such that $\|H(x,Y)\| \leq B$ almost surely, and (iv) $g$ is strictly convex, i.e. for any $0 < \delta < 1$, $\inf_{\delta < \|x - x_{*}\| < 1/\delta} \langle x - x^{\star}, \nabla g(x) \rangle > 0$. Then, $x_k$ converges to $x^{\star}$ almost surely.
\end{lemma}
Let $B$ be an arbitrary symmetric positive definite matrix. We apply \autoref{lem:robbins_monro} to the algorithm $x_{k+1} = x_{k} - \gamma_k (B (x_{k} - x^{\star}) + A_{i_k} x^{\star} + b_{i_k})$, i.e. we set $Y_k = i_k$ and $H(x,Y) = B (x - x^{\star}) + A_{Y} x^{\star} + b_{Y}$. Then, $\mathbb{E}[H(x,Y)] = B (x - x^{\star}) + A x^{\star} + b$, which can be regarded as the gradient of $g(x) = \frac{1}{2} \langle x, B x \rangle + \langle - B x^{\star} + A x^{\star} + b, x \rangle$. Under the conditions (i)-(iv), we obtain that $x_{k}$ converges almost surely to the minimizer of $g$, which is $x^{\star} -B^{-1} (A x^{\star} + b)$, which is equal to $-A^{-1} b$ under the condition $x^{\star} = - A^{-1} b$ as explained before. Hence, our model is consistent in the sense that when stepsizes converge to zero, the iterates converge almost surely to the minimizer.

\subsection{Difference between the iterates under zero-th order noise and standard noise} \label{subsec:difference}

To further study the consistency of the zero-th order noise model, we take a look at the difference between the iterates under the original noise and the iterates under the zero-th order noise. 
We will consider the case of SGDM for simplicity; while the arguments for SNAG are analogous and the arguments for SGD ($\alpha = 0$) are simpler.
Suppose that $(x_k)_k$ is the sequence of iterates under the original noise, i.e. 
$x_{k+1} = (1+\alpha) x_k - \alpha x_{k-1} - \gamma \nabla f_i(x_k)$. 
Suppose that $(x_k^0)_k$ is the sequence of iterates under the zero-th order noise model described in \autoref{sec:with_replacement}, 
i.e. 
$x_{k+1}^0 = x_k^0 + \alpha(x_k^0 - x_{k-1}^0) - \gamma \widehat{\nabla f}_{i_k}(x_k^0)$,
where $\widehat{\nabla f}_{i}(x) = A x_k + b - A_{i} A^{-1} b + b_{i}$. Remark that
\begin{align} \label{eq:widehat_f_i_f_i}
    \widehat{\nabla f_i}(x) &= A x - A_i A^{-1} b + b_i + b = A_i x + b_i -(A_i - A)(x + A^{-1}b) \\ &= \nabla f_i(x) - (A_i - A)(x + A^{-1}b) 
\end{align}
Let us define the sequence $(x^1_k)_k$ by setting $x_k = x^0_k + x^1_k$. Using equation \eqref{eq:widehat_f_i_f_i}, we have that:
\begin{align}
    &x^0_{k+1} + x^1_{k+1} \\ &= (1+\alpha)(x^0_k + x^1_k) - \alpha(x^0_{k-1} + x^1_{k-1}) - \gamma(\widehat{\nabla f}_{i_k}(x^0_k + x^1_k) + (A_{i_k} - A)(x^0_k + x^1_k + A^{-1}b)) \\ &= (1+\alpha)(x^0_k + x^1_k) - \alpha(x^0_{k-1} + x^1_{k-1})  - \gamma(\widehat{\nabla f}_{i_k}(x^0_k) + A x^1_k + (A_{i_k} - A)(x^0_k + x^1_k + A^{-1}b))
\end{align}
Hence, 
$x_{k+1}^0 = (1+\alpha) x_k^0 - \alpha x_{k-1}^0 - \gamma \widehat{\nabla f}_{i_k}(x_k^0)$ implies that
\begin{align} \label{eq:discrepancy}
    x^1_{k+1} &= (1+\alpha) x^1_k - \alpha x^1_{k-1} - \gamma(A x^1_k + (A_{i_k} - A)(x^0_k + x^1_k + A^{-1}b)) \\ &= (1+\alpha) x^1_k - \alpha x^1_{k-1} - \gamma(A_{i_k} x^1_k + (A_{i_k} - A)(x^0_k + A^{-1}b)) 
\end{align}
We want to find an expression for the sequence $(x_k^0)_k$. We define the sequence $(\tilde{x}_k^0)_k \subseteq \R^{2d}$ as $\tilde{x}_0^0 = [(x_0^0)^{\top}, (x_0^0)^{\top}]^{\top}$ and for $k \geq 1$, $\tilde{x}_k^0 = [(x_k^0)^{\top}, (x_{k-1}^0)^{\top}]^{\top}$, and the matrix $B \in \R^{2d \times 2d}$ as 
\begin{align}
    B = 
    \begin{bmatrix}
    (1+\alpha) \mathrm{Id} -\gamma A & - \alpha \mathrm{Id} \\ \mathrm{Id} & 0
    \end{bmatrix}.
\end{align}
Note that the eigenvalues and eigenvectors of $B$ are of the form:
\begin{align}
    \begin{bmatrix}
    (1+\alpha) \mathrm{Id} -\gamma A & - \alpha \mathrm{Id} \\ \mathrm{Id} & 0
    \end{bmatrix} 
    \begin{bmatrix}
    v \\ w
    \end{bmatrix} = \lambda \begin{bmatrix}
    v \\ w
    \end{bmatrix} \iff 
    \begin{cases}
    (1+\alpha) v - \gamma A v - \alpha w = \lambda v, \\ v = \lambda w
    \end{cases}
\end{align}
Thus, if $v$ is an eigenvector of $A$ with eigenvalue $\lambda_0$, we have that $[v^{\top}, v^{\top}/\lambda]$ is an eigenvector of $B$ with eigenvalue $\lambda$ satisfying:
\begin{align}
    &1+\alpha - \gamma \lambda_0 - \frac{\alpha}{\lambda} = \lambda \implies \lambda^2 - (1+\alpha - \gamma \lambda_0) \lambda + \alpha = 0 \\ &\implies \lambda_{\pm} = \frac{1+\alpha - \gamma \lambda_0 \pm \sqrt{(1+\alpha - \gamma \lambda_0)^2 - 4 \alpha}}{2}.
\end{align}
Note that in the regime $\gamma \lambda_0 \ll 1-\alpha$ for all eigenvalues $\lambda_0$ of $A$, we have 
\begin{align}
&\sqrt{(1+\alpha - \gamma \lambda_0)^2 - 4 \alpha} = \sqrt{(1+\alpha)^2 - 2 \gamma \lambda_0 (1+\alpha) + \gamma^2 \lambda_0^2 - 4 \alpha} \\ &= \sqrt{(1-\alpha)^2 - 2 \gamma \lambda_0 (1+\alpha) + \gamma^2 \lambda_0^2} = (1-\alpha)\sqrt{1 + \frac{\gamma^2 \lambda_0^2 - 2 \gamma \lambda_0 (1+\alpha)}{(1-\alpha)^2}} \\ &= (1-\alpha)\bigg(1 + \frac{\gamma^2 \lambda_0^2 - 2 \gamma \lambda_0 (1+\alpha)}{2(1-\alpha)^2} + O\bigg( \bigg(\frac{\gamma^2 \lambda_0^2 - 2 \gamma \lambda_0 (1+\alpha)}{(1-\alpha)^2} \bigg)^2 \bigg) \bigg) \\ &= 1-\alpha - \frac{\gamma \lambda_0(1+\alpha)}{(1-\alpha)^2} + O\bigg( \left(\frac{\gamma \lambda_0}{1-\alpha} \right)^2 \bigg).
\end{align}
And this means that 
\begin{align}
    \lambda_{\pm} = 
    \begin{cases}
    1 - \frac{\gamma \lambda_0(1+\alpha)}{2(1-\alpha)^2} + O\left( \left(\frac{\gamma \lambda_0}{1-\alpha} \right)^2 \right) \\ 
    \alpha + \frac{\gamma \lambda_0(1+\alpha)}{2(1-\alpha)^2} + O\left( \left(\frac{\gamma \lambda_0}{1-\alpha} \right)^2 \right)
    \end{cases}
\end{align}
Thus, when $\gamma \lambda_0 \ll 1-\alpha$ and $\frac{\gamma \lambda_0(1+\alpha)}{2(1-\alpha)^2} \leq 1-\alpha$, which holds if $\gamma \lambda_0 \leq (1-\alpha)^3$, we have that $|\lambda_{\pm}| < 1$. Since we obtain a pair $\lambda_{\pm}$ for each of the $d$ eigenvalues $\lambda_0$ of $A$, the $2d$ eigenvalues of $B$ are of this form.

We also define the vector $\tilde{b} \in \R^{2d}$ and the sequence $(\tilde{y}_{i_k})_{k}$ as
\begin{align}
    \tilde{b} = 
    \begin{bmatrix}
    b \\ 0
    \end{bmatrix} \quad 
    \tilde{y}_{i_k} = \begin{bmatrix}
    y_{i_k} \\ 0
    \end{bmatrix}.
\end{align}
Let $[v^{\top},w^{\top}]^{\top} = (\mathrm{Id} - B)^{-1} \tilde{b}$. Then,
\begin{align}
    \begin{bmatrix}
    -\alpha \mathrm{Id} + \gamma A & \alpha \mathrm{Id} \\ -\mathrm{Id} & \mathrm{Id}
    \end{bmatrix} 
    \begin{bmatrix}
    v \\ w
    \end{bmatrix} = \begin{bmatrix}
    b \\ 0
    \end{bmatrix}
\end{align}
And this implies that $v = w$, which means that $(-\alpha \mathrm{Id} + \gamma A)v + \alpha v = b \implies v = \frac{1}{\gamma} A^{-1} b \implies (\mathrm{Id} - B)^{-1} b = \frac{1}{\gamma} [(A^{-1} b)^{\top}, (A^{-1} b)^{\top}]^{\top}$. Thus, we get
\begin{align}
    x^0_k = [\tilde{x}^0_k]_{1:d} = - \gamma \bigg[\sum_{\kappa=0}^{+\infty} B^{\kappa} y_{i_{k-\kappa}} \bigg]_{1:d} - A^{-1} b,
\end{align}
and equation \eqref{eq:discrepancy} becomes
\begin{align}
    x^1_{k+1} = (1+\alpha) x^1_k - \alpha x^1_{k-1} - \gamma A_{i_k} x^1_k + \gamma^2 (A_{i_k} - A) \bigg[\sum_{\kappa=0}^{+\infty} B^{\kappa} y_{i_{k-\kappa}} \bigg]_{1:d}.
\end{align}
\paragraph{Expectation of the difference for gradient estimates with replacement.} When the indices $i_k$ are chosen i.i.d. for each iteration $k$, we can write
\begin{align}
    \mathbb{E}[x^1_{k+1}] &= (1+\alpha) \mathbb{E}[x^1_k] - \alpha \mathbb{E}[x^1_{k-1}] - \gamma \mathbb{E}[A_{i_k}] \mathbb{E}[x^1_k] \\ &+ \gamma^2 (\mathbb{E}[A_{i_k} y_{i_{k}}] - A \mathbb{E}[y_{i_{k}}]) + \gamma^2 \mathbb{E}[A_{i_k} - A] \bigg[\sum_{\kappa=1}^{+\infty} B^{\kappa} \mathbb{E}[y_{i_{k-\kappa}}] \bigg]_{1:d},
\end{align}
which implies that
\begin{align}
    &\mathbb{E}[x^1_k] = (1+\alpha) \mathbb{E}[x^1_k] - \alpha \mathbb{E}[x^1_{k}] - \gamma A \mathbb{E}[x^1_k] + \gamma^2 \mathbb{E}[A_{i} y_{i}] \implies 0 = - A \mathbb{E}[x^1_k] + \gamma \mathbb{E}[A_{i} y_{i}] \\ &\implies \mathbb{E}[x^1_k] = \gamma A^{-1} \mathbb{E}[A_{i} y_{i}]. 
\end{align}
That is, when functions are chosen with replacement, the difference between the expectation $\mathbb{E}[x^0_k]$ of the iterates under the zero-th order noise model and the expectation $\mathbb{E}[x^0_k + x^1_k]$ of the iterates under the standard stochastic noise is of order $\Theta(\gamma)$. While $x^0_k$ is unbiased in the sense that $\mathbb{E}[x^0_k] = x^{\star}$, $x^1_k$ is biased. Still $\|\mathbb{E}[x^1_k]\|^2 = \Theta(\gamma^2)$ is of lower order than the variance of $x^0_k$ given by the approximation \eqref{eq:replacement_limit}, which is $\Theta(\gamma)$, which is why in \autoref{sec:experiments} and \autoref{sec:more_experiments} we observe the same values for both types of noise. A complete analysis would require a study of the variance of $x^0_k + x^1_k$ by looking at higher-order noise models, which is left for future work.  

We also leave for future work the computation of the difference of expectations $\mathbb{E}[x^1_k]$ for RR and SO, and the analysis of the variance of $x^0_k + x^1_k$ through higher-order noise models. This would help explain the discrepancies between errors found in \autoref{sec:more_experiments} when the values $u_i^{\top} \Sigma u_i$ are highly unequal.

\section{Proofs of \autoref{sec:with_replacement}} \label{app:proofs_replacement}

\begin{lemma} \label{lem:continuously_differentiable}
Suppose that $f : \R \to \R$ with period 1 is continuously differentiable. Then, the sequence of its Fourier coefficents ${(s_k)}_{k \in \mathbb{Z}}$ is absolutely summable: $\sum_{k=-\infty}^{\infty} |s_k| < +\infty$.
\end{lemma}
\begin{proof}
This is a standard result. The Fourier coefficients of the second derivative $f''$ are ${(2\pi i k s_k)}_{k \in \mathbb{Z}}$. Since $f''$ is continuous and 1-periodic, it belongs to $L^2([0,1])$, which by Parseval's theorem implies that $\sum_{k=-\infty}^{\infty} k^2 s_k^2 < +\infty$. By the Cauchy-Schwarz inequality,
\begin{align}
    \sum_{k=-\infty}^{\infty} |s_k| = \sum_{k=-\infty}^{\infty} k |s_k| \cdot \frac{1}{k} \leq \left(\sum_{k=-\infty}^{\infty} k^2 s_k^2 \right) \left(\sum_{k=-\infty}^{\infty} \frac{1}{k^2} \right) < +\infty.
\end{align}
\end{proof}

\begin{lemma} \label{lem:eigenvalues_not_zero}
If $(1+\alpha)\mathrm{Id} - \gamma A \succeq 0$, the matrix $z^2 \mathrm{Id} - z((1+\alpha)\mathrm{Id} - \gamma A) + \alpha \mathrm{Id}$ has full rank when $z$ has modulus 1, which shows that the function $f \to H(e^{2\pi i f})$ for SGDM is continuous (and continuously differentiable). If $\mathrm{Id} - \gamma A \succeq 0$, the same statement holds for $z^2 \mathrm{Id} - z (1 + \alpha) (\mathrm{Id} - \gamma A) + \alpha (\mathrm{Id} - \gamma A)$, and it shows that the function $f \to H(e^{2\pi i f})$ for SNAG is continuous (and continuously differentiable).
\end{lemma}
\begin{proof}
Since $A$ is symmetric and strictly positive definite, we know by the spectral theorem that it diagonalizes and that it has positive eigenvalues $(\lambda_j)_{j=1}^{d}$. The matrix $z^2 \mathrm{Id} - ((1+\alpha)\mathrm{Id} - \gamma A) z + \alpha \mathrm{Id}$ diagonalizes in the same basis and has eigenvalues $z^2 - (1+\alpha - \gamma \lambda_j) z + \alpha$. If it did not have full rank, we would have, for some $j \in \{1, \dots, d\}$,
\begin{align}
    z^2 - (1+\alpha - \gamma \lambda_j) z + \alpha = 0 \implies z_{\pm} = \frac{1+\alpha - \gamma \lambda_j \pm \sqrt{(1+\alpha - \gamma \lambda_j)^2 - 4 \alpha}}{2}.
\end{align}
The condition $(1+\alpha)\mathrm{Id} - \gamma A \succeq 0$ implies that $1+\alpha - \gamma \lambda_j > 0$ for any $j \in \{1, \dots, d\}$. If $(1+\alpha - \gamma \lambda_j)^2 - 4 \alpha > 0$,
\begin{align}
|z_{\pm}| &= \frac{1+\alpha - \gamma \lambda_j \pm \sqrt{(1+\alpha - \gamma \lambda_j)^2 - 4 \alpha}}{2} \\ &= \frac{1+\alpha - \gamma \lambda_j \pm \sqrt{1+2\alpha + \alpha^2 + \gamma^2 \lambda_j^2 -2\gamma \lambda_j (1+\alpha) - 4 \alpha}}{2} \\ &\leq \frac{1+\alpha - \gamma \lambda_j + \sqrt{(1-\alpha)^2 - \gamma \lambda_j (1+\alpha)}}{2} < \frac{1+\alpha + 1-\alpha}{2} = 1,
\end{align}
where we used that $\gamma^2 \lambda_j^2 < \gamma \lambda_j (1+\alpha)$ in the second equality. If $(1+\alpha - \gamma \lambda_j)^2 - 4 \alpha < 0$,
\begin{align}
    |z_{\pm}|^2 = \frac{(1+\alpha - \gamma \lambda_j)^2 + 4\alpha - (1+\alpha - \gamma \lambda_j)^2}{4} = \alpha < 1.
\end{align}
Hence, all the zeros must have modulus strictly less than 1, which concludes the proof for SGDM.
For SNAG, we have that if the matrix $z^2 \mathrm{Id} - z (1 + \alpha) (\mathrm{Id} - \gamma A) + \alpha (\mathrm{Id} - \gamma A)$ did not have full rank, for some $j \in \{1, \dots, d\}$,
\begin{align}
\begin{split}
    &z^2 - (1 + \alpha) (1 - \gamma \lambda_j) z + \alpha (1 - \gamma \lambda_j) = 0 \\ &\implies z_{\pm} = \frac{(1 + \alpha) (1 - \gamma \lambda_j) \pm \sqrt{(1 + \alpha)^2 (1 - \gamma \lambda_j)^2 - 4 \alpha (1 - \gamma \lambda_j)}}{2}.
\end{split}
\end{align} 
The condition $\mathrm{Id} - \gamma A \succeq 0$ implies that $1 - \gamma \lambda_j > 0$ for any $j \in \{1, \dots, d\}$. If $(1 + \alpha)^2 (1 - \gamma \lambda_j)^2 - 4 \alpha (1 - \gamma \lambda_j)> 0$, we have that
\begin{align}
    |z_{\pm}| &= \frac{(1 + \alpha) (1 - \gamma \lambda_j) \pm \sqrt{(1 + \alpha)^2 (1 - \gamma \lambda_j)^2 - 4 \alpha (1 - \gamma \lambda_j)}}{2} \\ &\leq \frac{(1 + \alpha) (1 - \gamma \lambda_j) + \sqrt{(1 + \alpha)^2 (1 - \gamma \lambda_j)^2 - 4 \alpha (1 - \gamma \lambda_j)^2}}{2} \\ &= \frac{(1 + \alpha) (1 - \gamma \lambda_j) + (1 - \alpha) (1 - \gamma \lambda_j)}{2} = 1 - \gamma \lambda_j < 1,
\end{align}
while if $(1 + \alpha)^2 (1 - \gamma \lambda_j)^2 - 4 \alpha (1 - \gamma \lambda_j) < 0$,
\begin{align}
    |z_{\pm}|^2 = \frac{(1 + \alpha)^2 (1 - \gamma \lambda_j)^2 + 4 \alpha (1 - \gamma \lambda_j) - (1 + \alpha)^2 (1 - \gamma \lambda_j)^2}{4} = \alpha (1 - \gamma \lambda_j) < 1.
\end{align}
\end{proof}

\textit{\textbf{Proof of Proposition \ref{prop:integral}.}}
    Define $\xi = 1+\alpha -\gamma \lambda$ and $\chi = \alpha$, and $p = \frac{2 \xi(1 + \chi)}{1+\xi^2+\chi^2}$, $q = \frac{2 \chi}{1+\xi^2+\chi^2}$. We can reexpress the integral as
\begin{align} \label{eq:exp_to_cos}
    &\gamma^2 \int_{\epsilon}^{1-\epsilon} (e^{-4\pi i f} - \xi e^{-2\pi i f} + \chi)^{-1} (e^{4\pi i f} - \xi e^{2\pi i f} + \chi)^{-1} \, df 
    \\ &= \int_{\epsilon}^{1-\epsilon} (1+\xi^2+\chi^2 - \xi(1 + \chi) (e^{-2\pi i f} + e^{2\pi i f}) + \chi(e^{-4\pi i f} + e^{4\pi i f}))^{-1} \, df \\ &= \int_{\epsilon}^{1-\epsilon} (1+\xi^2+\chi^2 - 2 \xi(1 + \chi) \cos(2\pi f) + 2 \chi \cos(4\pi f))^{-1} \, df \\ &= \frac{1}{2\pi(1+\xi^2+\chi^2)} \int_{2\pi \epsilon}^{2\pi(1-\epsilon)} \left(1 - \frac{2 \xi(1 + \chi)}{1+\xi^2+\chi^2} \cos(f) + \frac{2 \chi}{1+\xi^2+\chi^2} \cos(2f) \right)^{-1} \, df
    \\ &= \frac{1}{2\pi(1+\xi^2+\chi^2)} \int_{-\pi(1-2\epsilon)}^{\pi(1-2\epsilon)} \left(1 - \frac{2 \xi(1 + \chi)}{1+\xi^2+\chi^2} \cos(f + \pi) + \frac{2 \chi}{1+\xi^2+\chi^2} \cos(2(f+\pi)) \right)^{-1} \, df
    \\ &= \frac{1}{2\pi(1+\xi^2+\chi^2)} \int_{-\pi(1-2\epsilon)}^{\pi(1-2\epsilon)} \left(1 + p \cos(f) + q \cos(2f) \right)^{-1} \, df
\end{align}
We are interested in computing $\int \left(1 + p \cos(f) + q \cos(2f) \right)^{-1} \, dx$.
We make the change of variables $u = \tan(\frac{x}{2})$, which means that $x = 2 \arctan(u)$ and $dx = \frac{2}{1+u^2} \, du$. Note that
\begin{align}
    \frac{1 - \tan^2(\frac{x}{2})}{1 + \tan^2(\frac{x}{2})} = \frac{\frac{\cos^2(\frac{x}{2}) - \sin^2(\frac{x}{2})}{\cos^2(\frac{x}{2})}}{\frac{\cos^2(\frac{x}{2}) + \sin^2(\frac{x}{2})}{\cos^2(\frac{x}{2})}} = \cos^2\left(\frac{x}{2} \right) - \sin^2 \left(\frac{x}{2} \right) = \cos(x),
\end{align}
and this implies that $\cos(x) = \frac{1-u^2}{1+u^2}$. Also, the double angle formula for the tangent is $\tan(2x) = \frac{2 \tan(x)}{1-\tan^2(x)}$, which means that
\begin{align}
    \cos(2x) &= \frac{1 - \tan^2(x)}{1 + \tan^2(x)} = \frac{1 - \left(\frac{2 \tan(\frac{x}{2})}{1-\tan^2(\frac{x}{2})}\right)^2}{1 + \left(\frac{2 \tan(\frac{x}{2})}{1-\tan^2(\frac{x}{2})}\right)^2} = \frac{1 - \left(\frac{2 u}{1-u^2}\right)^2}{1 + \left(\frac{2 u}{1-u^2}\right)^2} 
    = \frac{1-6u^2 + u^4}{(1 + u^2)^2}.
\end{align}
Thus,
\begin{align} \label{eq:change_tan}
    \int \frac{1}{1+p \cos(x) + q\cos(2x)} \, dx &= \int \frac{2}{(1 + u^2)\left(1+ \frac{a (1-u^2)}{1+u^2} + \frac{b(1-6u^2 + u^4)}{(1 + u^2)^2} \right)} \, du \\ &= \int \frac{2}{1 + u^2 + p (1-u^2) + \frac{q(1-6u^2 + u^4)}{1 + u^2}} \, du \\ &= \int \frac{2(1 + u^2)}{(1 + u^2)^2 + p (1-u^2)(1 + u^2) + q(1-6u^2 + u^4)} \, du 
    \\ &= \int \frac{2(1 + u^2)}{1 + p + q + (2 -6q) u^2 + (1 - p + q) u^4} \, du.
\end{align}
We can solve this integral by partial fractions. First, we compute the roots of the denominator: setting $v = u^2$, the roots fulfill $(1 - p + q) v^2 + (2 -6q) v + 1 + p + q = 0$, which implies that 
\begin{align} \label{eq:v_pm}
    v_{\pm} 
    = \frac{-(1-3q) \pm \sqrt{(1 -3q)^2 - (1 - p + q)(1 + p + q)}}{1 - p + q}.
\end{align}
At this point it is convenient to compute $1-p+q$, $1+p+q$ and $1-3q$: \autoref{lem:1_minus_p_plus_q}(i) shows that they are equal to
\begin{align} 
\begin{split}\label{eq:1_minus_p_plus_q_SGDM}
    1-p+q &= \frac{\gamma^2 \lambda^2}{1+(1 + \alpha - \gamma \lambda)^2+\alpha^2} > 0, \\
    1+p+q &= \frac{2(1 + \alpha - \gamma \lambda)^2 + 2(1+\alpha)^2 - \gamma^2 \lambda^2}{1+(1 + \alpha - \gamma \lambda)^2+\alpha^2}, \\
    1-3q &= 
    \frac{2(1-\alpha)^2+ \gamma^2 \lambda^2 - 2(1 + \alpha)\gamma \lambda}{1+(1 + \alpha - \gamma \lambda)^2+\alpha^2}.
\end{split}
\end{align}
Using these expressions, we get
\begin{align}
\begin{split} \label{eq:sqrt_positive}
    &(1+(1 + \alpha - \gamma \lambda)^2+\alpha^2)^2((1 -3q)^2 - (1 - p + q)(1 + p + q)) \\ &= (2(1-\alpha)^2+ \gamma^2 \lambda^2 - 2(1 + \alpha)\gamma \lambda)^2 - \gamma^2 \lambda^2(2(1 + \alpha - \gamma \lambda)^2 + 2(1+\alpha)^2 - \gamma^2 \lambda^2) \\ &= 
    4(1-\alpha)^4 + \gamma^4 \lambda^4 + 4(1 + \alpha)^2 \gamma^2 \lambda^2 + 4 (1-\alpha)^2 \gamma^2 \lambda^2 - 8 (1-\alpha)^2 (1 + \alpha) \gamma \lambda - 4 \gamma^3 \lambda^3 (1 + \alpha) \\ &- 2\gamma^2 \lambda^2 ((1 + \alpha)^2 + \gamma^2 \lambda^2 - 2 \gamma \lambda (1+\alpha)) - 2\gamma^2 \lambda^2 (1+\alpha)^2 + \gamma^4 \lambda^4 \\ &= 4(1-\alpha)^4 + 4 (1-\alpha)^2 \gamma^2 \lambda^2 - 8 (1-\alpha)^2 (1 + \alpha) \gamma \lambda \\ &= 4 (1-\alpha)^2 ((1-\alpha)^2 + \gamma^2 \lambda^2 - 2 (1 + \alpha) \gamma \lambda) > 0,
\end{split}
\end{align}
where the last equality holds by the assumption that $(1-\alpha)^2 + \gamma^2 \lambda^2 - 2 (1 + \alpha) \gamma \lambda > 0$.
Using equations \eqref{eq:1_minus_p_plus_q_SGDM} and \eqref{eq:sqrt_positive}, if we define $\rho_{\pm} = -v_{\mp}$, we get
\begin{align}
    \rho_{\pm} = \frac{2(1-\alpha)^2+ \gamma^2 \lambda^2 - 2(1 + \alpha)\gamma \lambda \pm 2 (1-\alpha) \sqrt{(1-\alpha)^2+ \gamma^2 \lambda^2 - 2 (1 + \alpha) \gamma \lambda}}{\gamma^2 \lambda^2}.
\end{align}
Note that $\rho_{\pm}$ are both positive because $1-3q \geq \sqrt{(1-3q)^2 - (1-p+q)(1+p+q)}$, since 
\begin{align}
    (1-p+q)(1+p+q) &= \frac{\gamma^2 \lambda^2(2(1 + \alpha - \gamma \lambda)^2 + 2(1+\alpha)^2 - \gamma^2 \lambda^2)}{(1+(1 + \alpha - \gamma \lambda)^2+\alpha^2)^2} \\ &= \frac{\gamma^2 \lambda^2(2(1 + \alpha)^2 + 2\gamma^2 \lambda^2 - 4 (1+\alpha) \gamma \lambda + 2(1+\alpha)^2 - \gamma^2 \lambda^2)}{(1+(1 + \alpha - \gamma \lambda)^2+\alpha^2)^2} \\ &\geq \frac{\gamma^2 \lambda^2(2(1+\alpha)^2 - \gamma^2 \lambda^2)}{(1+(1 + \alpha - \gamma \lambda)^2+\alpha^2)^2} \geq 0.
\end{align}
Here the first inequality holds because $(1+\alpha)^2 \geq (1-\alpha)^2$ since $\alpha \in [0,1)$, and by the assumption that $(1-\alpha)^2 + \gamma^2 \lambda^2 - 2 (1 + \alpha) \gamma \lambda \geq 0$, and the last inequality holds because $2(1+\alpha)^2 - \gamma^2 \lambda^2 \geq (1-\alpha)^2 - \gamma^2 \lambda^2 \geq 0$, by the same assumption.

Hence, the denominator of \eqref{eq:change_tan} has four imaginary roots $u_1 = i\sqrt{\rho_{+}}, u_2=-i\sqrt{\rho_{+}}, u_3 = i\sqrt{\rho_{+}}, u_4 = -i\sqrt{\rho_{+}}$. Since $(u-u_1)(u-u_2)(u-u_1)(u-u_2) = (u-i\sqrt{\rho_{+}})(u+i\sqrt{\rho_{+}}) (u-i\sqrt{\rho_{-}})(u+i\sqrt{\rho_{-}}) = (u^2 + \rho_{+})(u^2 + \rho_{-})$, we obtain
\begin{align}
   \frac{2(1 + u^2)}{1 + p + q + (2 -6p) u^2 + (1 - p + q) u^4} = \frac{2}{(1-p+q)} \frac{1+u^2}{(u^2 + \rho_{+})(u^2 + \rho_{-})}
\end{align}
Assume that $A,B$ are such that
\begin{align}
    \frac{1+u^2}{(u^2 + \rho_{+})(u^2 + \rho_{-})} = \frac{A}{u^2 + \rho_{+}} + \frac{B}{u^2 + \rho_{-}} = \frac{A(u^2 + \rho_{-}) + B(u^2 + \rho_{+})}{(u^2 + \rho_{+})(u^2 + \rho_{-})} = \frac{(A+B)u^2 + A\rho_{-} + B \rho_{+}}{(u^2 + \rho_{+})(u^2 + \rho_{-})}.
\end{align}
That is, $A + B = 1$ and $A\rho_{-} + B \rho_{+} = 1$, which means that $A = \frac{\rho_{+} - 1}{\rho_{+} - \rho_{-}}$, $B = \frac{1 - \rho_{-}}{\rho_{+} - \rho_{-}}$. Thus, the right-hand side of \eqref{eq:change_tan} is equal to:
\begin{align}
    &\frac{2}{(1-p+q)} \int \left( \frac{\rho_{+} - 1}{\rho_{+} - \rho_{-}} \cdot \frac{1}{u^2 + \rho_{+}} + \frac{1 - \rho_{-}}{\rho_{+} - \rho_{-}} \cdot \frac{1}{u^2 + \rho_{-}} \right) \, du \\ &= \frac{2}{(1-p+q)(\rho_{+} - \rho_{-})} \left( \frac{\rho_{+} - 1}{\sqrt{\rho_{+}}} \arctan \left(\frac{u}{\sqrt{\rho_{+}}} \right) + \frac{1 - \rho_{-}}{\sqrt{\rho_{-}}} \arctan \left(\frac{u}{\sqrt{\rho_{-}}} \right) \right) 
\end{align}
In the equality we used that $\int \frac{1}{t^2 + m^2} \, dt = \frac{1}{m} \arctan(t/m)$. We undo the change of variables $u = \tan(\frac{x}{2})$, and we obtain that the right-hand side of 
\eqref{eq:exp_to_cos} reads
\begin{align}
    &\frac{\left[ \frac{\rho_{+} - 1}{\sqrt{\rho_{+}}} \arctan \left(\frac{\tan(\frac{x}{2})}{\sqrt{\rho_{+}}} \right) + \frac{1 - \rho_{-}}{\sqrt{\rho_{-}}} \arctan \left(\frac{\tan(\frac{x}{2})}{\sqrt{\rho_{-}}} \right) \right]_{-\pi(1-2\epsilon)}^{\pi(1-2\epsilon)}}{\pi(1+\xi^2+\chi^2)(1-p+q)(\rho_{+} - \rho_{-})}  
    \\ &= \frac{\left[ \frac{\rho_{+} - 1}{\sqrt{\rho_{+}}} \arctan \left(\frac{\tan(\frac{x}{2})}{\sqrt{\rho_{+}}} \right) + \frac{1 - \rho_{-}}{\sqrt{\rho_{-}}} \arctan \left(\frac{\tan(\frac{x}{2})}{\sqrt{\rho_{-}}} \right) \right]_{-\pi(1-2\epsilon)}^{\pi(1-2\epsilon)}}{4 \pi (1-\alpha) \sqrt{(1-\alpha)^2+ \gamma^2 \lambda^2 - 2 (1 + \alpha) \gamma \lambda}} 
\end{align}
The equality holds because $(1+\xi^2+\chi^2)(1-p+q) = \gamma^2 \lambda^2$, and $\rho_{+} - \rho_{-} = \frac{4 (1-\alpha)}{\gamma^2 \lambda^2} \sqrt{(1-\alpha)^2+ \gamma^2 \lambda^2 - 2 (1 + \alpha) \gamma \lambda}$.
\qed

\begin{lemma} \label{lem:sgd_integral}
    For SGD (the setting $\alpha = 0$), the integral in Proposition \ref{prop:integral} simplifies to:
\begin{align} \label{eq:int_sgd_simple}
    \int_{\epsilon}^{1-\epsilon} |e^{-4\pi i f} - e^{-2\pi i f} (1 - \gamma \lambda)|^{-2} \, df = \frac{\left[ \arctan \left(\frac{\tan(\frac{x}{2})}{\sqrt{\rho_{+}}} \right) \right]_{-\pi(1-2\epsilon)}^{\pi(1-2\epsilon)}}{\pi \gamma \lambda (2-\gamma \lambda)}
\end{align}
Equation \eqref{eq:variance_SGD_it} follows from this.
\end{lemma}
\begin{proof}
If we subsitute $\alpha = 0$ into \eqref{eq:rho_sgdm_def}, we get:
\begin{align}
    &\rho_{\pm} = \frac{2 + \gamma^2 \lambda^2 - 2\gamma \lambda \pm 2(1-\gamma \lambda)}{\gamma^2 \lambda^2} = 
    \begin{cases} \frac{4(1-\gamma \lambda) +\gamma^2 \lambda^2}{\gamma^2 \lambda^2} \\ 1 
    \end{cases}
    \\ &\implies \sqrt{\rho_{+}} - \frac{1}{\sqrt{\rho_{+}}} = \frac{\rho_{+} - 1}{\sqrt{\rho_{+}}} 
    = \frac{\frac{4(1-\gamma \lambda)}{\gamma^2 \lambda^2}}{\sqrt{1 +\frac{4(1-\gamma \lambda)}{\gamma^2 \lambda^2}}} = \frac{4(1-\gamma \lambda)}{\gamma^2 \lambda^2} \frac{\gamma \lambda}{2 - \gamma \lambda} = \frac{4(1-\gamma \lambda)}{\gamma \lambda (2-\gamma \lambda)},
\end{align}
Here, we used that $1 + \frac{4(1-\gamma \lambda)}{\gamma^2 \lambda^2} = \frac{(2-\gamma \lambda)^2}{\gamma^2 \lambda^2}$. Also, $\sqrt{\rho_{-}} - \frac{1}{\sqrt{\rho_{-}}} = 0$. Equation \eqref{eq:int_sgd_simple} follows.
\end{proof}

\begin{lemma} \label{lem:1_minus_p_plus_q}
(i) When $\xi = 1+\alpha -\gamma \lambda$ and $\chi = \alpha$, we have that $1-p+q$, $1+p+q$ and $1-3b$ are given by equation \eqref{eq:1_minus_p_plus_q_SGDM}. 
(ii) When $\xi = (1+\alpha)(1 -\gamma \lambda)$ and $\chi = \alpha(1 -\gamma \lambda)$, $1-p+q$, $1+p+q$ and $1-3b$ are given by \eqref{eq:1_minus_p_plus_q_SNAG}.
\end{lemma}
\begin{proof} (i)
\begin{align}
    p &= \frac{2 (1 + \alpha - \gamma \lambda) (1 + \alpha)}{1+(1 + \alpha - \gamma \lambda)^2+\alpha^2} = \frac{2 + 2\alpha - 2\gamma \lambda + 2\alpha + 2 \alpha^2 - 2\alpha \gamma \lambda}{2 + \alpha - \gamma \lambda + \alpha + \alpha^2 - \alpha \gamma \lambda -\gamma \lambda -\gamma \lambda \alpha + \gamma^2 \lambda^2 +\alpha^2} \\ &= \frac{2 + 4\alpha - 2\gamma \lambda + 2 \alpha^2 - 2\alpha \gamma \lambda}{2 + 2\alpha - 2\gamma \lambda + 2\alpha^2 - 2\alpha \gamma \lambda  + \gamma^2 \lambda^2} = 1 + \frac{2\alpha - \gamma^2 \lambda^2}{1+(1 + \alpha - \gamma \lambda)^2+\alpha^2}, \\
    q &= \frac{2\alpha}{1+(1 + \alpha - \gamma \lambda)^2+\alpha^2}.
\end{align}
Thus,
\begin{align}
    1 - p + q = \frac{- 2\alpha + \gamma^2 \lambda^2 + 2\alpha}{1+(1 + \alpha - \gamma \lambda)^2+\alpha^2} = \frac{\gamma^2 \lambda^2}{1+(1 + \alpha - \gamma \lambda)^2+\alpha^2} > 0.
\end{align}
Also,
\begin{align}
    1+p+q &= \frac{2(1+(1 + \alpha - \gamma \lambda)^2+\alpha^2) + 4\alpha - \gamma^2 \lambda^2}{1+(1 + \alpha - \gamma \lambda)^2+\alpha^2} = \frac{2(1 + \alpha - \gamma \lambda)^2 + 2(1+\alpha)^2 - \gamma^2 \lambda^2}{1+(1 + \alpha - \gamma \lambda)^2+\alpha^2}, \\ 
    1-3q &= \frac{1+(1 + \alpha - \gamma \lambda)^2+\alpha^2 - 6\alpha}{1+(1 + \alpha - \gamma \lambda)^2+\alpha^2} = \frac{1+(1 + \alpha)^2 + \gamma^2 \lambda^2 - 2(1 + \alpha)\gamma \lambda +\alpha^2 - 6\alpha}{1+(1 + \alpha - \gamma \lambda)^2+\alpha^2} \\ &= \frac{2(1-\alpha)^2+ \gamma^2 \lambda^2 - 2(1 + \alpha)\gamma \lambda}{1+(1 + \alpha - \gamma \lambda)^2+\alpha^2} 
\end{align}
(ii)
\begin{align}
    p &= \frac{2 (1 + \alpha) (1 - \gamma \lambda) (1 + \alpha (1 - \gamma \lambda))}{1+(1 + \alpha)^2 (1 - \gamma \lambda)^2+\alpha^2 (1 - \gamma \lambda)^2} 
    \\ &= \frac{2 + 4\alpha - 2\gamma \lambda - 6\alpha \gamma \lambda + 2\alpha^2 - 4 \alpha^2 \gamma \lambda + 2\alpha \gamma^2 \lambda^2 + 2\alpha^2 \gamma^2 \lambda^2}{2 + 2\alpha + 2\alpha^2 - 2 \gamma \lambda - 4 \gamma \lambda \alpha - 4 \gamma \lambda \alpha^2 + \gamma^2 \lambda^2 + 2 \gamma^2 \lambda^2 \alpha + 2 \gamma^2 \lambda^2 \alpha^2} 
    \\ &= 1 + \frac{4\alpha - 6\alpha \gamma \lambda - (2\alpha - 4 \gamma \lambda \alpha + \gamma^2 \lambda^2)}{2 + 2\alpha + 2\alpha^2 - 2 \gamma \lambda - 4 \gamma \lambda \alpha - 4 \gamma \lambda \alpha^2 + \gamma^2 \lambda^2 + 2 \gamma^2 \lambda^2 \alpha + 2 \gamma^2 \lambda^2 \alpha^2}
    \\ &= 1 + \frac{2\alpha - 2\alpha \gamma \lambda - \gamma^2 \lambda^2}{1+(1 + \alpha)^2 (1 - \gamma \lambda)^2+\alpha^2 (1 - \gamma \lambda)^2}.
\end{align}
\begin{align}
    q = \frac{2 \alpha (1 - \gamma \lambda)}{1+(1 + \alpha)^2 (1 - \gamma \lambda)^2+\alpha^2 (1 - \gamma \lambda)^2}.
\end{align}
Thus,
\begin{align}
    1-p+q = \frac{-2\alpha + 2\alpha \gamma \lambda + \gamma^2 \lambda^2 + 2 \alpha (1 - \gamma \lambda)}{1+(1 + \alpha)^2 (1 - \gamma \lambda)^2+\alpha^2 (1 - \gamma \lambda)^2} = \frac{\gamma^2 \lambda^2}{1+(1 + \alpha)^2 (1 - \gamma \lambda)^2+\alpha^2 (1 - \gamma \lambda)^2}
\end{align}
\begin{align}
    1+p+q &= \frac{2(1+(1 + \alpha)^2 (1 - \gamma \lambda)^2+\alpha^2 (1 - \gamma \lambda)^2) + 4\alpha(1 - \gamma \lambda) - \gamma^2 \lambda^2}{1+(1 + \alpha)^2 (1 - \gamma \lambda)^2+\alpha^2 (1 - \gamma \lambda)^2} \\ &= \frac{2(1 + \alpha)^2 (1 - \gamma \lambda)^2 
    + 2 (\alpha(1-\gamma \lambda) + 1)^2
    - \gamma^2 \lambda^2}{1+(1 + \alpha)^2 (1 - \gamma \lambda)^2+\alpha^2 (1 - \gamma \lambda)^2} 
\end{align}
\begin{align}
    1-3q &= \frac{1+(1 + \alpha)^2 (1 - \gamma \lambda)^2+\alpha^2 (1 - \gamma \lambda)^2 - 6 \alpha (1-\gamma \lambda)}{1+(1 + \alpha)^2 (1 - \gamma \lambda)^2+\alpha^2 (1 - \gamma \lambda)^2}
\end{align}
Hence, equation \eqref{eq:1_minus_p_plus_q_SNAG} follows.
\end{proof}

\begin{lemma} \label{lem:approx_sgdm}
When $\gamma \lambda_i \ll 1-\alpha$ for all $i \in \{1,\dots,d\}$, we have $\rho_{+} = \frac{4(1-\alpha)^2}{\gamma^2 \lambda^2} + O\left(\frac{1}{\gamma \lambda} \right)$, $\rho_{-} = (\frac{1+\alpha}{1-\alpha})^{2} + O(\gamma \lambda)$, and equation \eqref{eq:replacement_limit} holds.
\end{lemma}
\begin{proof}
We have that 
\begin{align}
    &\sqrt{(1-\alpha)^2+ \gamma^2 \lambda^2 - 2 (1 + \alpha) \gamma \lambda} = (1-\alpha) \sqrt{1 + \frac{\gamma^2 \lambda^2 - 2 (1 + \alpha) \gamma \lambda}{(1-\alpha)^2}} \\ &= (1-\alpha) \left( 1 + \frac{\gamma^2 \lambda^2 - 2 (1 + \alpha) \gamma \lambda}{2(1-\alpha)^2} - \frac{1}{8}\bigg(\frac{\gamma^2 \lambda^2 - 2 (1 + \alpha) \gamma \lambda}{(1-\alpha)^2}\bigg)^2 + O(\gamma^3 \lambda^3) \right) \\ &= 1-\alpha + \frac{\gamma^2 \lambda^2 - 2 (1 + \alpha) \gamma \lambda}{2(1-\alpha)} - \frac{1}{8}\frac{(\gamma^2 \lambda^2 - 2 (1 + \alpha) \gamma \lambda)^2}{(1-\alpha)^3} + O(\gamma^3 \lambda^3)
\end{align}
If we plug this into equation \eqref{eq:rho_sgdm_def}, we get
\begin{align}
    \rho_{-} 
    &= \frac{ \bigg( \frac{1}{4}\frac{(\gamma^2 \lambda^2 - 2 (1 + \alpha) \gamma \lambda)^2}{(1-\alpha)^2} + O(\gamma^3 \lambda^3) \bigg)}{\gamma^2 \lambda^2} = \frac{(\gamma \lambda - 2 (1 + \alpha))^2}{4(1-\alpha)^2} + O(\gamma \lambda) = \bigg(\frac{1+\alpha}{1-\alpha}\bigg)^{2} + O(\gamma \lambda).
\end{align}
\begin{align}
    \rho_{+} 
    &= \frac{4(1-\alpha)^2+ 2\gamma^2 \lambda^2 - 4(1 + \alpha)\gamma \lambda - \frac{1}{4}\frac{(\gamma^2 \lambda^2 - 2 (1 + \alpha) \gamma \lambda)^2}{(1-\alpha)^2} + O(\gamma^3 \lambda^3)}{\gamma^2 \lambda^2} = \frac{4(1-\alpha)^2}{\gamma^2 \lambda^2} + O\left(\frac{1}{\gamma \lambda} \right).
\end{align}
Then, 
\begin{align}
    \sqrt{\rho_{+}} &= \sqrt{\frac{4(1-\alpha)^2}{\gamma^2 \lambda^2} + O\left(\frac{1}{\gamma \lambda} \right)} = \sqrt{\frac{4(1-\alpha)^2}{\gamma^2 \lambda^2}} \sqrt{1+ \frac{O\left(\gamma \lambda \right)}{4(1-\alpha)^2}} = \frac{2(1-\alpha)}{\gamma \lambda} + O\left(\frac{1}{\sqrt{\gamma \lambda}} \right), \\
    \sqrt{\rho_{-}} &= \sqrt{\bigg(\frac{1+\alpha}{1-\alpha}\bigg)^{2} + O(\gamma \lambda)} =\sqrt{\bigg(\frac{1+\alpha}{1-\alpha}\bigg)^{2}} \sqrt{1 + O(\gamma \lambda)\bigg(\frac{1-\alpha}{1+\alpha}\bigg)^{2}} =  \frac{1+\alpha}{1-\alpha} + O(\gamma \lambda), \\
    \frac{1}{\sqrt{\rho_{-}}} &- \sqrt{\rho_{-}} = \frac{1-\alpha}{1+\alpha} - \frac{1+\alpha}{1-\alpha} + O(\gamma \lambda) = -\frac{4\alpha}{1-\alpha^2} + O(\gamma \lambda).
\end{align}
Hence, equation \eqref{eq:integral_expression} simplifies to:
\begin{align}
    \frac{\left[ \left( \frac{2(1-\alpha)}{\gamma \lambda} + O\left(\frac{1}{\sqrt{\gamma \lambda}} \right) \right) \arctan \left(\frac{\tan(\frac{x}{2})}{\frac{2(1-\alpha)}{\gamma \lambda} + O\left(\frac{1}{\sqrt{\gamma \lambda}} \right)} \right) + \left( \frac{-4\alpha}{1-\alpha^2} + O(\gamma \lambda) \right) \arctan \left(\frac{\tan(\frac{x}{2})}{\frac{-4\alpha}{1-\alpha^2} + O(\gamma \lambda)} \right) \right]_{-\pi(1-2\epsilon)}^{\pi(1-2\epsilon)}}{4 \pi (1-\alpha)^2 + O(\gamma \lambda)}.
\end{align}
And equation \eqref{eq:variance_sgdm_replace} simplifies to:
\begin{align}
    \mathrm{Tr}[\mathrm{Var}(x_k)] = \gamma^2 \sum_{i=1}^{d} (u_i^{\top} \Sigma u_i) \frac{ \frac{2(1-\alpha)}{\gamma \lambda_i} + O\left(\frac{1}{\sqrt{\gamma \lambda_i}} \right) }{4 (1-\alpha)^2 + O(\gamma \lambda_i)} = \sum_{i=1}^{d} \frac{\gamma(u_i^{\top} \Sigma u_i)}{2(1-\alpha)\lambda_i} + O\bigg(\sqrt{\frac{\gamma^3}{\lambda_i}} \bigg).
\end{align}
\end{proof}

\begin{proposition} \label{prop:integral2}
    Suppose that $(\frac{1-\alpha}{1+\alpha})^2 \geq \gamma \lambda$. Define 
    \begin{align} \label{eq:rho_snag_def}
        \rho_{\pm} &= \frac{1+(1 + \alpha)^2 (1 - \gamma \lambda)^2+\alpha^2 (1 - \gamma \lambda)^2 - 6 \alpha (1-\gamma \lambda)}{\gamma^2 \lambda^2} \\ &\pm \frac{2 (1 - \alpha(1 - \gamma \lambda)) \sqrt{(1-\gamma \lambda)((1-\alpha)^2 - \gamma \lambda (1+\alpha)^2)}}{\gamma^2 \lambda^2},
    \end{align}
    which are both non-negative.
    Then, for all $0 < \epsilon \leq 1/2$, $\int_{\epsilon}^{1-\epsilon} |e^{-4\pi i f} - e^{-2\pi i f} (1 + \alpha) (1 - \gamma \lambda) + \alpha (1 - \gamma \lambda)|^{-2} \, df$ is equal to
    \begin{align} \label{eq:integral_expression2}
        \frac{\left[ \frac{\rho_{+} - 1}{\sqrt{\rho_{+}}} \arctan \left(\frac{\tan(\frac{x}{2})}{\sqrt{\rho_{+}}} \right) + \frac{1 - \rho_{-}}{\sqrt{\rho_{-}}} \arctan \left(\frac{\tan(\frac{x}{2})}{\sqrt{\rho_{-}}} \right) \right]_{-\pi(1-2\epsilon)}^{\pi(1-2\epsilon)}}{4 \pi (1 - \alpha(1 - \gamma \lambda)) \sqrt{(1-\gamma \lambda)((1-\alpha)^2 - \gamma \lambda (1+\alpha)^2)}}
    \end{align}
\end{proposition}
\begin{proof}
The proof is analogous to the one of Proposition \ref{prop:integral}. In this case we define $\xi = (1 + \alpha) (1 - \gamma \lambda)$ and $\chi = \alpha (1 - \gamma \lambda)$, and $p,q$ are defined in terms of $\xi$, $\chi$ in the same way: $p = \frac{2 \xi(1 + \chi)}{1+\xi^2+\chi^2}$, $q = \frac{2 \chi}{1+\xi^2+\chi^2}$. The proof is exactly the same until equation \eqref{eq:1_minus_p_plus_q_SGDM}. In this case, Lemma \ref{lem:1_minus_p_plus_q}(ii) shows that we have instead:
\begin{align} 
\begin{split} \label{eq:1_minus_p_plus_q_SNAG}
    1 - p + q &= \frac{\gamma^2 \lambda^2}{1+(1 + \alpha)^2 (1 - \gamma \lambda)^2+\alpha^2 (1 - \gamma \lambda)^2} \\
     1+p+q &= \frac{2(1 + \alpha)^2 (1 - \gamma \lambda)^2 + 2 (\alpha(1-\gamma \lambda) + 1)^2
    - \gamma^2 \lambda^2}{1+(1 + \alpha)^2 (1 - \gamma \lambda)^2+\alpha^2 (1 - \gamma \lambda)^2} \\
    1-3q &= \frac{
    1+(1 + \alpha)^2 (1 - \gamma \lambda)^2+\alpha^2 (1 - \gamma \lambda)^2 - 6 \alpha (1-\gamma \lambda)}{1+(1 + \alpha)^2 (1 - \gamma \lambda)^2+\alpha^2 (1 - \gamma \lambda)^2}.
\end{split}
\end{align}
Using these expressions, we get
\begin{align}
\begin{split} \label{eq:sqrt_positive2}
    &(1+(1 + \alpha)^2 (1 - \gamma \lambda)^2+\alpha^2 (1 - \gamma \lambda)^2)^2 ((1 -3q)^2 - (1 - p + q)(1 + p + q)) \\ &= (1+(1 + \alpha)^2 (1 - \gamma \lambda)^2+\alpha^2 (1 - \gamma \lambda)^2 - 6 \alpha (1-\gamma \lambda))^2 \\ &- \gamma^2 \lambda^2 (2(1 + \alpha)^2 (1 - \gamma \lambda)^2 + 2 (\alpha(1-\gamma \lambda) + 1)^2
    - \gamma^2 \lambda^2) \\ &= 
    4(\gamma \lambda - 1)(\alpha \gamma \lambda - \alpha + 1)^2 (2 \alpha + \gamma \lambda - \alpha^2 + 2 \alpha \gamma \lambda + \alpha^2 \gamma \lambda - 1) 
    \\ &= 4 (1-\gamma \lambda) (1 - \alpha(1 - \gamma \lambda))^2 ((1-\alpha)^2 - \gamma \lambda (1+\alpha)^2) \geq 0
\end{split}
\end{align}
For the inequality, we used the assumption that $(\frac{1-\alpha}{1+\alpha})^2 \geq \gamma \lambda$. Using equations \eqref{eq:1_minus_p_plus_q_SNAG} and \eqref{eq:sqrt_positive2}, if we define $\rho_{\pm} = -v_{\mp}$, we obtain \eqref{eq:rho_snag_def}.
Note that $\rho_{\pm}$ are both positive because $1-3q \geq \sqrt{(1-3q)^2 - (1-p+q)(1+p+q)}$, since 
\begin{align}
    (1-p+q)(1+p+q) &= \frac{\gamma^2 \lambda^2(2(1 + \alpha)^2 (1 - \gamma \lambda)^2 + 2 (\alpha(1-\gamma \lambda) + 1)^2
    - \gamma^2 \lambda^2)}{(1+(1 + \alpha)^2 (1 - \gamma \lambda)^2+\alpha^2 (1 - \gamma \lambda)^2)^2} \geq 0
\end{align}
The inequality holds because $(\frac{1-\alpha}{1+\alpha})^2 \geq \gamma \lambda$ implies that $1-\gamma \lambda \geq 0$, which implies that $2 (\alpha(1-\gamma \lambda) + 1)^2 - \gamma^2 \lambda^2 \geq 0$. Using the same argument as in Proposition \ref{prop:integral}, we obtain
\begin{align}
    &\frac{\left[ \frac{\rho_{+} - 1}{\sqrt{\rho_{+}}} \arctan \left(\frac{\tan(\frac{x}{2})}{\sqrt{\rho_{+}}} \right) + \frac{1 - \rho_{-}}{\sqrt{\rho_{-}}} \arctan \left(\frac{\tan(\frac{x}{2})}{\sqrt{\rho_{-}}} \right) \right]_{-\pi(1-2\epsilon)}^{\pi(1-2\epsilon)}}{\pi(1+\xi^2+\chi^2)(1-p+q)(\rho_{+} - \rho_{-})}  
    \\ &= \frac{\left[ \frac{\rho_{+} - 1}{\sqrt{\rho_{+}}} \arctan \left(\frac{\tan(\frac{x}{2})}{\sqrt{\rho_{+}}} \right) + \frac{1 - \rho_{-}}{\sqrt{\rho_{-}}} \arctan \left(\frac{\tan(\frac{x}{2})}{\sqrt{\rho_{-}}} \right) \right]_{-\pi(1-2\epsilon)}^{\pi(1-2\epsilon)}}{4 \pi (1 - \alpha(1 - \gamma \lambda)) \sqrt{(1-\gamma \lambda)((1-\alpha)^2 - \gamma \lambda (1+\alpha)^2)}}
\end{align}
The equality holds because $(1+\xi^2+\chi^2)(1-p+q) = \gamma^2 \lambda^2$, and $\rho_{+} - \rho_{-} = \frac{4 (1 - \alpha(1 - \gamma \lambda))}{\gamma^2 \lambda^2} \sqrt{(1-\gamma \lambda)((1-\alpha)^2 - \gamma \lambda (1+\alpha)^2)}$.
\end{proof}

\begin{corollary} \label{cor:variance_SNAG}
    If $(\frac{1-\alpha}{1+\alpha})^2 \geq \gamma \lambda_i$ for all $i \in \{1,\dots,d\}$, the variance for the SNAG iterates with replacement under the zero-th order noise model is
\begin{align} \label{eq:variance_snag_replace}
    \mathrm{Tr}[\mathrm{Var}(x_k)] = \gamma^2 \sum_{i=1}^{d} (u_i^{\top} \Sigma u_i) \frac{\sqrt{\rho_{+}^{i}} - \sqrt{\rho_{-}^{i}} - \frac{1}{\sqrt{\rho_{+}^{i}}} + \frac{1}{\sqrt{\rho_{-}^{i}}}}{4 (1 - \alpha(1 - \gamma \lambda_i)) \sqrt{(1-\gamma \lambda_i)((1-\alpha)^2 - \gamma \lambda_i (1+\alpha)^2)}},
\end{align}
where we defined $\rho_{\pm}^{i}$ as $\rho_{\pm}$ in \eqref{eq:rho_sgdm_def} with the choice $\lambda = \lambda_i$. In the limit $\gamma \lambda_i \ll 1-\alpha$ for all $i \in \{1,\dots,d\}$, we obtain the same variance as in the limit for SGDM:
\begin{align} \label{eq:replacement_limit_snag}
    \mathrm{Tr}[\mathrm{Var}(x_k)] = \sum_{i=1}^{d} \frac{\gamma(u_i^{\top} \Sigma u_i)}{2(1-\alpha)\lambda_i} + O\bigg(\sqrt{\frac{\gamma^3}{\lambda_i}} \bigg).
\end{align}
\end{corollary}
\begin{proof}
    Equation \eqref{eq:variance_snag_replace} is a direct consequence of \eqref{eq:integral_expression2}, following the same reasoning as for SGDM. To study the limit $\gamma \lambda_i \ll 1-\alpha$, we use an analogous argument to Lemma \ref{lem:approx_sgdm}. We have that $\xi = (1 + \alpha) (1 - \gamma \lambda) = 1 + \alpha + O(\gamma \lambda)$, $\chi = \alpha (1 - \gamma \lambda) = \alpha + O(\gamma \lambda)$, and
\begin{align}
    p &= \frac{2 \xi(1+ \chi)}{1+\xi^2+\chi^2} =  \frac{2 (1+\alpha)^2 + O(\gamma \lambda)}{1+(1+\alpha)^2+\alpha^2 + O(\gamma \lambda)} 
    = \frac{1 + 2\alpha + \alpha^2}{1 + \alpha + \alpha^2} + O(\gamma \lambda), \\ q &= \frac{2 \chi}{1+\xi^2+\chi^2} = \frac{2 \alpha + O(\gamma \lambda)}{1+(1+\alpha)^2+\alpha^2 + O(\gamma \lambda)} = \frac{\alpha}{1 + \alpha + \alpha^2} + O(\gamma \lambda),
\end{align}
which means that
\begin{align}
    1-3q &= \frac{1 + \alpha + \alpha^2 -3\alpha}{1 + \alpha + \alpha^2} + O(\gamma \lambda) 
    = \frac{(1-\alpha)^2}{1 + \alpha + \alpha^2} + O(\gamma \lambda). \\
    1-p+q &= \frac{\gamma^2 \lambda^2}{1+(1 + \alpha)^2 +\alpha^2 +  O(\gamma \lambda)} = \frac{\gamma^2 \lambda^2}{2(1 + \alpha + \alpha^2)} +  O(\gamma^3 \lambda^3), \\
    1+p+q &= 1+\frac{1 + 2\alpha + \alpha^2}{1 + \alpha + \alpha^2} + \frac{\alpha}{1 + \alpha + \alpha^2} +  O(\gamma \lambda) = \frac{2 + 4\alpha + 2\alpha^2}{1 + \alpha + \alpha^2} +  O(\gamma \lambda).
\end{align}
\begin{align}
    &(1 + \alpha + \alpha^2)^2 ((1-3q)^2 - (1-p+q)(1+p+q)) \\ &= (1-\alpha)^4 - \frac{\gamma^2 \lambda^2}{2} (2 + 4\alpha + 2\alpha^2) +  O(\gamma \lambda) = (1-\alpha)^4 +  O(\gamma \lambda). 
\end{align}
\begin{align}
    \rho_{\pm} &= \frac{1-3q \pm \sqrt{(1 -3q)^2 - (1 - p + q)(1 + p + q)}}{1 - p + q} = \frac{(1-\alpha)^2 \pm \sqrt{(1-\alpha)^4 +  O(\gamma \lambda)}}{\frac{1}{2}\gamma^2 \lambda^2 +  O(\gamma^3 \lambda^3)} \\ &= 
    \begin{cases}
    \frac{4(1-\alpha)^2}{\gamma^2 \lambda^2} + O\left(\frac{1}{\gamma \lambda}\right) \\
    O\left(\frac{1}{\gamma \lambda}\right).
    \end{cases}
\end{align}
This does not yield the largest-order term of $\rho_{-}$. To obtain the largest-order term of $\rho_{-}$, we write
\begin{align}
    &\sqrt{(1 -3q)^2 - (1 - p + q)(1 + p + q)} = (1 -3q) \sqrt{1 - \frac{(1 - p + q)(1 + p + q)}{(1 -3q)^2}} \\ &= (1 -3q) \bigg(1 - \frac{(1 - p + q)(1 + p + q)}{2(1 -3q)^2} + O(\gamma^4 \lambda^4) \bigg) = 1 -3q - \frac{(1 - p + q)(1 + p + q)}{2(1 -3q)} + O(\gamma^4 \lambda^4) \\ &= 1 -3q - \frac{(\frac{1}{2}\gamma^2 \lambda^2 + O(\gamma^3 \lambda^3))(2(1+\alpha)^2+O(\gamma \lambda))}{2(1-\alpha)^2 + O(\gamma \lambda)} + O(\gamma^4 \lambda^4) \\ &= 1 -3q - \frac{\gamma^2 \lambda^2 (1+\alpha)^2}{2(1-\alpha)^2} + O(\gamma^3 \lambda^3).
\end{align}
Hence,
\begin{align}
    \rho_{-} = \frac{1-3q - \left( 1 -3q - \frac{\gamma^2 \lambda^2 (1+\alpha)^2}{2(1-\alpha)^2} + O(\gamma^3 \lambda^3) \right)}{\gamma^2 \lambda^2} = \frac{1}{2}\left( \frac{1+\alpha}{1-\alpha} \right)^2 + O(\gamma \lambda).
\end{align}
Taking square roots, we get 
\begin{align}
    \sqrt{\rho_{+}} = \frac{2(1-\alpha)}{\gamma \lambda} + O\left( \frac{1}{\sqrt{\gamma \lambda}} \right), \qquad \sqrt{\rho_{-}} = \frac{1+\alpha}{\sqrt{2}(1-\alpha)} + O(\gamma \lambda)
\end{align}
We conclude the proof:
\begin{align}
    4 (1 - \alpha(1 - \gamma \lambda_i)) \sqrt{(1-\gamma \lambda_i)((1-\alpha)^2 - \gamma \lambda_i (1+\alpha)^2)} = 4 (1 - \alpha)^2
\end{align}
\begin{align} \label{eq:variance_snag_replace_limit}
    \mathrm{Tr}[\mathrm{Var}(x_k)] = \sum_{i=1}^{d} \gamma^2  (u_i^{\top} \Sigma u_i) \frac{\frac{2(1-\alpha)}{\gamma \lambda}}{4 (1 - \alpha)^2} = \sum_{i=1}^{d} \frac{\gamma(u_i^{\top} \Sigma u_i)}{2(1-\alpha)\lambda_i} + O\bigg(\sqrt{\frac{\gamma^3}{\lambda_i}} \bigg).
\end{align}
\end{proof}

\section{Proofs of \autoref{sec:SO}} \label{app:proofs_sec_SO}
\begin{proposition} \label{prop:variance_S_x_SO}
When $R_b(k) = ((1+\frac{1}{n-1}) \mathds{1}_{k \equiv 0 \, (\mathrm{mod} \, n)} - \frac{1}{n-1}) \Sigma$, the process $x = h * b$ satisfies
\begin{align}
    \mathbb{E}[\|x_{k}\|^2] = \frac{1}{n-1} \sum_{k'=1}^{n-1} \mathrm{Tr}[\Sigma H(e^{\frac{2\pi i k'}{n}})^{\top} H(e^{-\frac{2\pi i k'}{n}})] 
\end{align}
\end{proposition}
\begin{proof}
We use \autoref{lem:variance_S_x}. We rewrite $R_b(k) = \frac{n}{n-1} \mathds{1}_{k \equiv 0 \, (\mathrm{mod} \, n)} \Sigma - \frac{1}{n-1} \Sigma$, and treat the two terms independently. On the one hand, we have that 
\begin{align}
&\sum_{k=-\infty}^{\infty} \int_{0}^{1} \mathrm{Tr}[\Sigma H(e^{2\pi i f})^{\top} H(e^{-2\pi i f})] \, e^{-2\pi i k f} \, df \\ &= \sum_{k=-\infty}^{\infty} \langle \mathrm{Tr}[\Sigma H(e^{2\pi i f})^{\top} H(e^{-2\pi i f})], e^{-2\pi i k f} \rangle_{L^2([0,1])} \\ &= \bigg(\sum_{k=-\infty}^{\infty} \langle \mathrm{Tr}[\Sigma H(e^{2\pi i f})^{\top} H(e^{-2\pi i f})], e^{-2\pi i k f} \rangle_{L^2([0,1])} e^{2\pi i k f'} \bigg) \bigg|_{f' = 0} \\ &= (\mathrm{Tr}[\Sigma H(e^{2\pi i f})^{\top} H(e^{-2\pi i f})])|_{f' = 0} = \mathrm{Tr}[\Sigma H(0)^{\top} H(0)],
\end{align}
and on the other hand
\begin{align} \label{eq:mathds_transform}
&\sum_{k=-\infty}^{\infty} \int_{0}^{1} \mathrm{Tr}[\mathds{1}_{k \equiv 0 \, (\mathrm{mod} \, n)} \Sigma H(e^{2\pi i f})^{\top} H(e^{-2\pi i f})] \, e^{-2\pi i k f} \, df \\ &= \sum_{k=-\infty}^{\infty} \int_{0}^{1} \mathrm{Tr}[\Sigma H(e^{2\pi i f})^{\top} H(e^{-2\pi i f})] \, e^{-2\pi i n k f} \, df
\\ &= \frac{1}{n} \sum_{k=-\infty}^{\infty} \int_{0}^{n} \mathrm{Tr}[\Sigma H(e^{\frac{2\pi i f'}{n}})^{\top} H(e^{-\frac{2\pi i f'}{n}})] \, e^{-2\pi i k f'} \, df'
\\ &= \frac{1}{n} \sum_{k=-\infty}^{\infty} \sum_{j=0}^{n-1} \int_{j}^{j+1} \mathrm{Tr}[\Sigma H(e^{\frac{2\pi i f'}{n}})^{\top} H(e^{-\frac{2\pi i f'}{n}})] \, e^{-2\pi i k f'} \, df' \\ &= \frac{1}{n} \sum_{k=-\infty}^{\infty} \sum_{k'=0}^{n-1} \int_{0}^{1} \mathrm{Tr}[\Sigma H(e^{\frac{2\pi i (k'+f)}{n}})^{\top} H(e^{-\frac{2\pi i (k'+f)}{n}})] \, e^{-2\pi i k (k'+f)} \, df
\\ &= \frac{1}{n} \sum_{k'=0}^{n-1} \sum_{k=-\infty}^{\infty} \int_{0}^{1} \mathrm{Tr}[\Sigma H(e^{\frac{2\pi i (k'+f)}{n}})^{\top} H(e^{-\frac{2\pi i (k'+f)}{n}})] \, e^{-2\pi i k f} \, df
\\ &= \frac{1}{n} \sum_{k'=0}^{n-1} \sum_{k=-\infty}^{\infty} \langle \mathrm{Tr}[\Sigma H(e^{\frac{2\pi i (k'+f)}{n}})^{\top} H(e^{-\frac{2\pi i (k'+f)}{n}})], e^{-2\pi i k f} \rangle_{L^2([0,1])} \\ &= \frac{1}{n} \sum_{k'=0}^{n-1} \mathrm{Tr}[\Sigma H(e^{\frac{2\pi i k'}{n}})^{\top} H(e^{-\frac{2\pi i k'}{n}})].
\end{align}
Hence,
\begin{align}
\begin{split}
    \mathbb{E}[\|x_{k}\|^2] &= \frac{n}{n-1} \bigg( \frac{1}{n} \sum_{k'=0}^{n-1} \mathrm{Tr}[\Sigma H(e^{\frac{2\pi i k'}{n}})^{\top} H(e^{-\frac{2\pi i k'}{n}})] \bigg) - \frac{1}{n-1} \mathrm{Tr}[\Sigma H(0)^{\top} H(0)] \\ &= \frac{1}{n-1} \sum_{k'=1}^{n-1} \mathrm{Tr}[\Sigma H(e^{\frac{2\pi i k'}{n}})^{\top} H(e^{-\frac{2\pi i k'}{n}})].
\end{split}    
\end{align}
\end{proof}

\textbf{\textit{Proof of \autoref{prop:variance_sgdm_so}.}}
For $\rho_{+} \geq \rho_{-} > 0$, define the function
\begin{align} \label{eq:g_def}
g(x) = \frac{\rho_{+} - 1}{\sqrt{\rho_{+}}} \arctan\bigg(\frac{\tan\left(\frac{x}{2} \right)}{\sqrt{\rho_{+}}} \bigg) - \frac{\rho_{-} - 1}{\sqrt{\rho_{-}}} \arctan\bigg(\frac{\tan\left(\frac{x}{2} \right)}{\sqrt{\rho_{-}}} \bigg)
\end{align}
has derivative 
\begin{align} \label{eq:g'_def}
    g'(x) &= \frac{\rho_{+} - 1}{\sqrt{\rho_{+}}} \frac{1}{1 + \frac{\tan^2(\frac{x}{2})}{\rho_{+}}} \cdot \frac{1}{\sqrt{\rho_{+}} \cos^2(\frac{x}{2})} \cdot \frac{1}{2} - \frac{\rho_{-} - 1}{\sqrt{\rho_{-}}} \frac{1}{1 + \frac{\tan^2(\frac{x}{2})}{\rho_{-}}} \cdot \frac{1}{\sqrt{\rho_{-}} \cos^2(\frac{x}{2})} \cdot \frac{1}{2} \\ &= \frac{\rho_{+}-1}{2 \rho_{+} \cos^2(\frac{x}{2}) + 2 \sin^2(\frac{x}{2})} - \frac{\rho_{-} - 1}{2 \rho_{-} \cos^2(\frac{x}{2}) + 2 \sin^2(\frac{x}{2})} \\ &= \frac{\rho_{+}-1}{2 + 2 (\rho_{+} - 1) \cos^2(\frac{x}{2})} - \frac{\rho_{-}-1}{2 + 2 (\rho_{-} - 1) \cos^2(\frac{x}{2})}.
\end{align}
Since $z \mapsto \frac{z}{2 + 2 z \cos^2(\frac{x}{2})}$ is increasing for $z \geq 0$, the inequality $\rho_{+} - 1 \geq \rho_{-} - 1$ implies that the right-hand side is non-negative, i.e. $g'(x) \geq 0$ for all $x \in (-\pi,\pi)$.
And the second derivative is:
\begin{align} \label{eq:g_second}
    &g''(x) = \frac{2 (\rho_{+}-1)^2
    \cos(\frac{x}{2}) \sin(\frac{x}{2})}{(2 + 2 (\rho_{+} - 1) \cos^2(\frac{x}{2}))^2} - \frac{2 (\rho_{-}-1)^2
    \cos(\frac{x}{2}) \sin(\frac{x}{2})}{(2 + 2 (\rho_{-} - 1) \cos^2(\frac{x}{2}))^2} \\ &= \frac{(\rho_{+}-1)^2
    \sin(x)}{(2 + 2 (\rho_{+} - 1) \cos^2(\frac{x}{2}))^2} - \frac{(\rho_{-}-1)^2
    \sin(x)}{(2 + 2 (\rho_{-} - 1) \cos^2(\frac{x}{2}))^2}
\end{align}
Note that for any $x \in [0,\pi)$, $z \mapsto \frac{z \sqrt{\sin(x)}}{2 + 2 z \cos^2(\frac{x}{2})}$ is increasing for $z \geq 0$. 
Hence, $z \mapsto \frac{z^2 \sin(x)}{(2 + 2 z \cos^2(\frac{x}{2}))^2}$ is increasing for $z \geq 0$, which 
through the inequality $\rho_{+} - 1 \geq \rho_{-} - 1$ implies that the right-hand side of \eqref{eq:g_second} is greater or equal than zero. Hence, $g''$ is non-negative on $[0,\pi]$, which means that $g'$ is non-decreasing on $[0,\pi)$, and by symmetry, non-increasing on $(-\pi,0]$. 

Note that
\begin{align}
\begin{split} \label{eq:sum_g'}
    &\left[ \frac{\rho_{+} - 1}{\sqrt{\rho_{+}}} \arctan \left(\frac{\tan(\frac{x}{2})}{\sqrt{\rho_{+}}} \right) + \frac{1 - \rho_{-}}{\sqrt{\rho_{-}}} \arctan \left(\frac{\tan(\frac{x}{2})}{\sqrt{\rho_{-}}} \right) \right]_{-\pi(1-1/n)}^{\pi(1-1/n)} 
    \\ &= 
    g(\pi(1-1/n)) - g(-\pi(1-1/n)) \\ &= 
    \int_{\pi(1-2/n)}^{\pi(1-1/n)} g'(x) \, dx + \int_{-\pi(1-1/n)}^{-\pi(1-2/n)} g'(x) \, dx + \sum_{k=1}^{n-2} \int_{\pi(1-2(k+1)/n)}^{\pi(1-2k/n)} g'(x) \, dx \\ &\geq \frac{\pi}{n} g'(\pi(1-2/n)) + \frac{\pi}{n} g'(-\pi(1-2/n)) + \frac{2\pi}{n} g'(0) + \frac{2\pi}{n} \sum_{k=2}^{n-2} g'(\pi(1-2k/n)) \\ &\geq \frac{\pi}{n} \sum_{k=1}^{n-1} g'(\pi(1-2k/n)).
\end{split}
\end{align}
Let $\rho_{\pm}$ be as in \eqref{eq:rho_sgdm_def}. Applying \autoref{lem:g'_equality} and using the short-hand $Z = 4 \pi (1-\alpha) \sqrt{(1-\alpha)^2+ \gamma^2 \lambda^2 - 2 (1 + \alpha) \gamma \lambda}$, we obtain that the right-hand side of \eqref{eq:sum_g'} is equal to
\begin{align}
    \frac{Z}{2n} \sum_{k=1}^{n-1} |e^{-\frac{4\pi i k}{n}} - e^{-\frac{2\pi i k}{n}} (1 + \alpha - \gamma \lambda) + \alpha|^{-2}.
\end{align}
Thus, using that $g$ is odd, \eqref{eq:sum_g'} implies that
\begin{align} \label{eq:average_bound}
    \frac{1}{n-1} \sum_{k=1}^{n-1} |e^{-\frac{4\pi i k}{n}} - e^{-\frac{2\pi i k}{n}} (1 + \alpha - \gamma \lambda) + \alpha|^{-2} \leq \frac{4n}{(n-1)Z} g(\pi(1-1/n)). 
\end{align}
In the regime $\sqrt{\rho_{+}} \gg n \gg 1$ and $\rho_{-} \geq 1$, \autoref{lem:sqrt_rho_n} implies that
\begin{align}
    g(\pi(1-1/n)) = \frac{2n}{\pi} + O\bigg(\frac{2n}{\pi \rho_{+}} + \frac{\pi}{2n} \bigg),
\end{align}
which upon substitution into the right-hand side of \eqref{eq:average_bound} yields:
\begin{align}
    &\frac{8n^2}{\pi(n-1)Z} (1 + O(1/\rho_{+}) + O(1/n^2)) 
    \\ &= \frac{2n^2}{ \pi^2 (n-1) (1-\alpha) \sqrt{(1-\alpha)^2+ \gamma^2 \lambda^2 - 2 (1 + \alpha) \gamma \lambda}} (1 + O(1/\rho_{+}) + O(1/n^2))
\end{align}
Plugging this development into the right-hand side of \eqref{eq:tr_sgdm_SO} with the choices $\lambda = \lambda_i$ (and thus, $\rho_{+} = \rho_{+}^i$), we obtain:
\begin{align} \label{eq:variance_sgdm_so_prop2}
    \mathrm{Tr}[\mathrm{Var}(x_k)] = \sum_{i=1}^{d} (u_i^{\top} \Sigma u_i) \frac{8 \gamma^2 n^2 (1 + O(1/\rho_{+}^{i})+ O(1/n^2))}{4 \pi^2 (n-1) (1-\alpha) \sqrt{(1-\alpha)^2+ \gamma^2 \lambda_i^2 - 2 (1 + \alpha) \gamma \lambda_i}}.
\end{align}
As stated in equation \eqref{eq:replacement_limit} (see \autoref{lem:approx_sgdm}), in the regime $\gamma \lambda_i \ll 1-\alpha$ for all $i \in \{1,\dots,d\}$, we have that $\rho_{+} = \frac{4(1-\alpha)^2}{\gamma^2 \lambda_i^2} + O(\frac{1}{\gamma \lambda_i})$. Hence, $\sqrt{\rho_{+}} \gg n$ holds iff $1-\alpha \gg \gamma \lambda_i n$ for all $i$, which is strictly stronger than the condition $\gamma \lambda_i \ll 1-\alpha$. Also by equation \eqref{eq:replacement_limit}, we get that $\rho_{-} = (\frac{1+\alpha}{1-\alpha})^{2} + O(\gamma \lambda_{i}) \geq 1$, which means that \autoref{lem:sqrt_rho_n} can be applied in this regime. Thus, under the condition $1-\alpha \gg \gamma \lambda_i n$ for all $i$ we can further simplify \eqref{eq:variance_sgdm_so_prop2} to:
\begin{align}
    \mathrm{Tr}[\mathrm{Var}(x_k)] &= \sum_{i=1}^{d} (u_i^{\top} \Sigma u_i) \frac{8 \gamma^2 n^2 (1 + O(\frac{\gamma^2 \lambda_i^2}{(1-\alpha)^2}) + O(1/n^2))(1+O(\frac{\gamma \lambda_i}{1-\alpha}))}{4 \pi^2 (n-1) (1-\alpha)^2} \\ &= \sum_{i=1}^{d} (u_i^{\top} \Sigma u_i) \frac{2 \gamma^2 n^2 (1 + O(\frac{\gamma \lambda_i}{1-\alpha}) + O(1/n^2))}{\pi^2 (n-1) (1-\alpha)^2}
\end{align}
We proceed analogously for SNAG and get the same result; the expressions for $\rho_{+}$ and $Z$ differ but they are the same up to first order.
\qed

\begin{lemma} \label{lem:g'_equality}
Define $g$ as in equation \eqref{eq:g_def}. If we set $\rho_{\pm}$ as in \eqref{eq:rho_sgdm_def} and we make use of the short-hand $Z = 4 \pi (1-\alpha) \sqrt{(1-\alpha)^2+ \gamma^2 \lambda^2 - 2 (1 + \alpha) \gamma \lambda}$, we have that for all $f \in [0,1]$,
\begin{align}
     |e^{-4\pi i f} - e^{-2\pi i f} (1 + \alpha - \gamma \lambda) + \alpha|^{-2} = \frac{2\pi}{Z} g'(\pi(1-2f)).
\end{align}
\end{lemma}
\begin{proof}
Note that \autoref{prop:integral} can be rewritten as
\begin{align}
    \int_{\epsilon_0}^{\epsilon_1} |e^{-4\pi i f} - e^{-2\pi i f} (1 + \alpha - \gamma \lambda) + \alpha|^{-2} \, df = \frac{1}{Z}(g(-\pi + 2\pi\epsilon_1) - g(-\pi + 2\pi\epsilon_0)),
\end{align}
Since
\begin{align}
    g(-\pi + 2\pi\epsilon_1) - g(-\pi + 2\pi\epsilon_0) = \int_{-\pi + 2\pi\epsilon_0}^{-\pi + 2\pi\epsilon_1} g'(x) \, dx = 2\pi \int_{\epsilon_0}^{\epsilon_1} g'(-\pi + 2\pi f) \, df,
\end{align}
and $\epsilon_0$, $\epsilon_1$ are arbitrary, we deduce that the integrands are equal almost everywhere, and then everywhere by continuity: for all $f \in [0,1]$, $|e^{-4\pi i f} - e^{-2\pi i f} (1 + \alpha - \gamma \lambda) + \alpha|^{-2} = \frac{2\pi}{Z} g'(-\pi(1-2f))$. The statement of the lemma follows from the fact that $g'$ is an even function.
\end{proof}
\begin{lemma} \label{lem:sqrt_rho_n}
When $\sqrt{\rho_{+}} \gg n \gg 1$, we have that 
\begin{align}
    \frac{\rho_{+} - 1}{\sqrt{\rho_{+}}} \arctan \left(\frac{\tan(\frac{\pi(1-1/n)}{2})}{\sqrt{\rho_{+}}} \right) 
    = \frac{2n}{\pi} + O\bigg(\frac{2n}{\pi \rho_{+}} + \frac{\pi}{2n} \bigg).
\end{align}
\end{lemma}
\begin{proof}
We reexpress 
\begin{align} \label{eq:arctan_rewrite}
    \arctan \left(\frac{\tan(\frac{\pi(1-1/n)}{2})}{\sqrt{\rho_{+}}} \right) = \arctan \bigg(\frac{1}{\sqrt{\rho_{+}} \cot(\frac{\pi(1-1/n)}{2})} \bigg) 
\end{align}
The Taylor approximation of the arctangent around zero is: $\arctan(x) = x - \frac{x^3}{3} + O(x^5)$. 
The Taylor approximation of the cotangent around $\frac{\pi}{2}$ is: $\cot\left(\frac{\pi}{2} + x \right) 
= -x - \frac{x^3}{3} + O(x^4)$,
where we used that $(\frac{d}{dt}\cot)(t) = \frac{1}{\sin^2(t)}$, that $(\frac{d^2}{dt^2}\cot)(t) = - \frac{\sin(2t)}{\sin^4(t)}$ and that $(\frac{d^3}{dt^3}\cot)(t) = (-2 \cos(2t) \sin^4(t) + 4 \sin(2t) \sin^3(t) \cos(t))/\sin^8(t)$. Hence, if we set $x = -\frac{\pi}{2n}$, the right-hand side of \eqref{eq:arctan_rewrite} is equal to
\begin{align}
    \frac{1}{\sqrt{\rho_{+}}(-x - \frac{x^3}{3} + O(x^4))} + O\bigg(\frac{1}{(\sqrt{\rho_{+}} x)^3} \bigg) = -\frac{1}{\sqrt{\rho_{+}}x} + O\bigg( \frac{x}{\sqrt{\rho_{+}}} \bigg) = \frac{2n}{\pi \sqrt{\rho_{+}}} + O\bigg( \frac{\pi}{2n\sqrt{\rho_{+}}} \bigg).
\end{align}
Thus,
\begin{align}
    \frac{\rho_{+} - 1}{\sqrt{\rho_{+}}} \arctan \left(\frac{\tan(\frac{\pi(1-1/n)}{2})}{\sqrt{\rho_{+}}} \right) 
    = \frac{2n}{\pi} \left(1 - \frac{1}{\rho_{+}} \right) + O\bigg( \frac{\pi}{2n} \left(1 - \frac{1}{\rho_{+}} \right) \bigg) = \frac{2n}{\pi} + O\bigg(\frac{2n}{\pi \rho_{+}} + \frac{\pi}{2n} \bigg).
\end{align}
\end{proof}

\section{Proofs of \autoref{sec:RR}} \label{app:proofs_sec_RR}

\begin{lemma} \label{lem:PSD_S_y}
The power spectral density $S_y$ for the RR noise sequence is given by equation \eqref{eq:PSD_y_RR}.
\end{lemma}
\begin{proof}
If we let $x = 2\pi f$,
\begin{align}
    &S_y(k) = \sum_{k=-\infty}^{+\infty} e^{-2\pi i k f} R_y(k) 
    \\ &= \left(1+\frac{1}{n-1} \right) \Sigma - \frac{\Sigma}{n(n-1)} \sum_{k=-(n-1)}^{n-1} (n-|k|) e^{-2\pi i k f} \\ &= \Sigma - \frac{\Sigma}{n(n-1)} \sum_{k=1}^{n-1} (n-k) (e^{-2\pi i k f} + e^{2\pi i k f}) = \Sigma - \frac{2 \Sigma}{n(n-1)} \sum_{k=1}^{n-1} (n-k) \cos(2\pi k f) \\ &= \Sigma - \frac{2 \Sigma}{n-1} \left(\frac{\sin((n-\frac{1}{2}) x)}{2 \sin(x/2)} - \frac{1}{2} \right) + \frac{2 \Sigma}{n(n-1)} \left( 
    \frac{(n-1) \sin((n-\frac{1}{2})x)}{2\sin(\frac{x}{2})} - \frac{\sin^2(\frac{(n-1)x}{2})}{2\sin^2(\frac{x}{2})}
    \right) \\ &= \Sigma \left( \frac{n}{n-1} - \frac{ 
    \frac{\sin((n-\frac{1}{2}) x)}{\sin(x/2)} + \frac{\sin^2(\frac{(n-1)x}{2})}{\sin^2(\frac{x}{2})}}{n(n-1)} \right).
\end{align}
where we used equations \eqref{eq:sum_cos} and \eqref{eq:sum_k_cos} from \autoref{lem:sum_cos} in the fifth equality, and that $\frac{1}{n-1} - \frac{1}{n} = \frac{1}{n(n-1)}$ in the last equality.
\end{proof}
\begin{lemma} \label{lem:sum_cos}
The following inequalities hold:
\begin{align} \label{eq:sum_cos}
    \sum_{k=1}^n \cos(kx) &= 
    \frac{\sin((n+\frac{1}{2}) x)}{2 \sin(x/2)} - \frac{1}{2}, \\
    \sum_{k=1}^n \sin(kx) &= 
    \frac{\sin(\frac{nx}{2}) \sin(\frac{(n+1)x}{2})}{\sin(\frac{x}{2})},
    \\ \sum_{k=1}^n k\cos(kx) &= 
    \frac{n \sin((n+\frac{1}{2})x)}{2\sin(\frac{x}{2})} - \frac{\sin^2(\frac{nx}{2})}{2\sin^2(\frac{x}{2})}
    \label{eq:sum_k_cos}
\end{align}
\end{lemma}
\begin{proof}
Let $D_n(x) = \frac{\sin((n+1/2) x)}{2\pi \sin(x/2)}$ be the Dirichlet kernel. Since $\sum_{k=1}^n \cos(kx) = \pi D_n(x) - \frac{1}{2}$, the first equation follows. For the second equation, note that since 
\begin{align}
\sum_{i=0}^{n} e^{ikx} = \frac{e^{i(n+1)x} - 1}{e^{ix} - 1} = e^{\frac{i(n+1)x}{2}} e^{-\frac{ix}{2}} \frac{e^{\frac{i(n+1)x}{2}} - e^{\frac{-i(n+1)x}{2}}}{e^{\frac{ix}{2}} - e^{-\frac{ix}{2}}} = e^{\frac{inx}{2}} \frac{\sin(\frac{(n+1)x}{2})}{\sin(\frac{x}{2})},
\end{align}
we have that
\begin{align}
    \sum_{k=1}^n \sin(kx) &= \frac{1}{2i}\sum_{k=1}^n (e^{i kx} - e^{-ikx}) = \frac{1}{2i} \left( e^{\frac{inx}{2}} \frac{\sin(\frac{(n+1)x}{2})}{\sin(\frac{x}{2})} - e^{-\frac{inx}{2}} \frac{\sin(-\frac{(n+1)x}{2})}{\sin(-\frac{x}{2})} \right) \\ &= \frac{\sin(\frac{nx}{2}) \sin(\frac{(n+1)x}{2})}{\sin(\frac{x}{2})}
\end{align}
Finally, the third equation follows from:
\begin{align}
    \sum_{k=1}^n k\cos(kx) &= \frac{d}{dx} \left( \sum_{k=1}^n \sin(kx) \right) =
    \frac{d}{dx} \left( \frac{\sin(\frac{nx}{2}) \sin(\frac{(n+1)x}{2})}{\sin(\frac{x}{2})} \right) \\ &= \frac{\frac{n}{2}\cos(\frac{nx}{2}) \sin(\frac{(n+1)x}{2}) + \frac{n+1}{2}\sin(\frac{nx}{2}) \cos(\frac{(n+1)x}{2})}{\sin(\frac{x}{2})} - \frac{\sin(\frac{nx}{2}) \sin(\frac{(n+1)x}{2}) \cos(\frac{x}{2})}{2\sin^2(\frac{x}{2})} \\ &= 
    \frac{\frac{n}{2}\sin((n+\frac{1}{2})x) + \frac{1}{2}\sin(\frac{nx}{2}) \cos(\frac{(n+1)x}{2})}{\sin(\frac{x}{2})} - \frac{\sin(\frac{nx}{2}) \sin(\frac{(n+1)x}{2}) \cos(\frac{x}{2})}{2\sin^2(\frac{x}{2})} \\ &= \frac{n \sin((n+\frac{1}{2})x)}{2\sin(\frac{x}{2})} - \frac{\sin^2(\frac{nx}{2})}{2\sin^2(\frac{x}{2})}.
\end{align}
The second-to-last equality holds because of the formula for the sine of a sum: $\sin((n+\frac{1}{2})x) = \cos(\frac{nx}{2}) \sin(\frac{(n+1)x}{2}) + \sin(\frac{nx}{2}) \cos(\frac{(n+1)x}{2})$. The last equality holds because $\sin(\frac{nx}{2}) = \sin(\frac{(n+1)x}{2}) \cos(-\frac{x}{2}) + \cos(\frac{(n+1)x}{2}) \sin(-\frac{x}{2}) = \sin(\frac{n x}{2}) \cos(\frac{x}{2}) - \cos(\frac{n x}{2}) \sin(\frac{x}{2})$.
\end{proof}

\begin{lemma} \label{lem:r_n_taylor}
Let $r_n$ be the function defined in equation \eqref{eq:PSD_y_RR}.
The Taylor expansion of $r_n$ around $0$ is of the form
\begin{align}
    r_n(x) = \frac{(n^4 + n^2 - 2n + 1)x^2}{12 n (n - 1)} + O(n^4 x^4).
\end{align}
Also, for all $x \in \R$, we have that $|r_n(x)| 
\leq \frac{2n}{n-1}$.
\end{lemma}
\begin{proof}
Note that the Taylor series around zero for the sine is of the form $\sin(x) = x - \frac{x^3}{6} + O(x^5)$. Hence,
\begin{align} \label{eq:taylor_sines}
    \frac{\sin((n-\frac{1}{2}) x)}{\sin(\frac{x}{2})} &= \frac{(n-\frac{1}{2}) x - \frac{((n-\frac{1}{2}) x)^3}{6} + O(((n-\frac{1}{2}) x)^5)}{\frac{x}{2} - \frac{(\frac{x}{2})^3}{6} + O((\frac{x}{2})^5)} \\ &= \frac{2n-1 - \frac{(n-\frac{1}{2})^3 x^2}{3} + O((n-\frac{1}{2})^5 x^4)}{1 - \frac{x^2}{24} + O(x^4)} \\ &= \bigg(2n-1 - \frac{(n-\frac{1}{2})^3 x^2}{3} + O\bigg(\left(n-\frac{1}{2} \right)^5 x^4\bigg) \bigg) \bigg(1 + \frac{x^2}{24} + O(x^4) \bigg) \\ &= 2n-1 - \bigg(\frac{(n-\frac{1}{2})^3}{3} - \frac{2n-1}{24} \bigg) x^2 + O(n^5 x^4).
\end{align}
Similarly,
\begin{align} \label{eq:taylor_sines_2}
    \frac{\sin(\frac{(n-1)x}{2})}{\sin(\frac{x}{2})} &= \frac{\frac{(n-1)x}{2} - \frac{(\frac{(n-1)x}{2})^3}{6} + O((\frac{(n-1)x}{2})^5)}{\frac{x}{2} - \frac{(\frac{x}{2})^3}{6} + O((\frac{x}{2})^5)} \\ &= \frac{n-1 - \frac{(\frac{n-1}{2})^3 x^2}{3} + O((\frac{n-1}{2})^5 x^4)}{1 - \frac{x^2}{24} + O(x^4)} \\ &= \bigg(n-1 - \frac{(n-1)^3 x^2}{24} + O\bigg(\left(\frac{n-1}{2} \right)^5 x^4 \bigg) \bigg) \bigg(1 + \frac{x^2}{24} + O(x^4) \bigg) \\ &= n-1 - \bigg(\frac{(n-1)^3}{24} - \frac{n-1}{24} \bigg) x^2 + O(n^5 x^4).
\end{align}
Hence,
\begin{align}
    \frac{\sin^2(\frac{(n-1)x}{2})}{\sin^2(\frac{x}{2})} = (n-1)^2 - 2(n-1) \bigg(\frac{(n-1)^3}{24} - \frac{n-1}{24} \bigg) x^2 + O(n^6 x^4).
\end{align}
Plugging equations \eqref{eq:taylor_sines} and \eqref{eq:taylor_sines_2} into the definition of $r_n$, we obtain
\begin{align}
    r_n(x) &= \frac{n}{n-1} - \frac{2n-1 - \bigg(\frac{(n-\frac{1}{2})^3}{3} - \frac{2n-1}{24} \bigg) x^2 + (n-1)^2 - 2(n-1) \bigg(\frac{(n-1)^3}{24} - \frac{n-1}{24} \bigg) x^2 + O(n^6 x^4)}{n(n-1)} \\ &= \frac{n}{n-1} - \frac{2n-1+(n-1)^2}{n(n-1)} + \frac{\left(\frac{(n-\frac{1}{2})^3}{3} - \frac{2n-1}{24} + \frac{(n-1)^4}{12} - \frac{(n-1)^2}{12} \right)x^2}{n(n-1)} + O(n^4 x^4) \\ &= 
    \frac{n(n + 1) x^2}{12} + O(n^4 x^4).
\end{align}
To show the second statement of the lemma, we use that both $x \mapsto \sin((n-\frac{1}{2}) x)/\sin(\frac{x}{2})$ and $x \mapsto \sin^2(\frac{(n-1)x}{2})/\sin^2(\frac{x}{2})$ are maximized in absolute value at $x = 0$, where they take values $2n-1$ and $(n-1)^2$ as shown before. Hence, for all $x \in \R$, we have that $|r_n(x)| \leq \frac{n}{n-1} + \frac{2n-1 + (n-1)^2}{n(n-1)} 
= \frac{2n}{n-1}$.
\end{proof}

\textit{\textbf{Proof of \autoref{prop:variance_sgdm_RR}.}}
As shown by equation \eqref{eq:var_RR}, we have
\begin{align} \label{eq:18_copy}
    \mathrm{Tr}[\mathrm{Var}(x_k)] = \gamma^2 \sum_{i=1}^{d} (u_i^{\top} \Sigma u_i) \int_{0}^{1} |e^{-4\pi i f} - e^{-2\pi i f} (1 + \alpha - \gamma \lambda_i) + \alpha |^{-2} r_n(2\pi f) \, df
\end{align}
Applying \autoref{lem:g'_equality}, we obtain that $|e^{-4\pi i f} - e^{-2\pi i f} (1 + \alpha - \gamma \lambda) + \alpha|^{-2} = \frac{2\pi}{Z} g'(\pi(1-2f))$, where $Z = 4 \pi (1-\alpha) \sqrt{(1-\alpha)^2+ \gamma^2 \lambda^2 - 2 (1 + \alpha) \gamma \lambda}$, and the expression of $g'$ is shown in equation \eqref{eq:g'_def}. For an arbitrary $0 < \epsilon \leq 1$, we can write 
\begin{align}
    \int_{0}^{1} g'(\pi(1-2f)) r_n(2\pi f) \, df = \int_{\frac{\epsilon}{2n}}^{1-\frac{\epsilon}{2n}} g'(\pi(1-2f)) r_n(2\pi f) \, df + \int_{-\frac{\epsilon}{2n}}^{\frac{\epsilon}{2n}} g'(\pi(1-2f)) r_n(2\pi f) \, df.
\end{align}
Note that since $g'(x) \geq 0$ for all $x \in \R$ as shown in the proof of \autoref{prop:variance_sgdm_so}, and $|r_n(x)| \leq \frac{2n}{n-1}$ by \autoref{lem:r_n_taylor}, we obtain the bound
\begin{align} \label{eq:g'_r_n_bound_1}
    &\int_{\frac{\epsilon}{2n}}^{1-\frac{\epsilon}{2n}} g'(\pi(1-2f)) r_n(2\pi f) \, df \leq \frac{2n}{n-1} \int_{\frac{\epsilon}{n}}^{1-\frac{\epsilon}{n}} g'(\pi(1-2f)) \, df \\ &= \frac{2n}{2\pi(n-1)} \int_{-\pi (1-\frac{\epsilon}{n})}^{\pi(1-\frac{\epsilon}{n})} g'(x) \, dx \\ &= \frac{n}{\pi(n-1)} \bigg[\frac{\rho_{+} - 1}{\sqrt{\rho_{+}}} \arctan \left(\frac{\tan(\frac{x}{2})}{\sqrt{\rho_{+}}} \right) + \frac{1 - \rho_{-}}{\sqrt{\rho_{-}}} \arctan \left(\frac{\tan(\frac{x}{2})}{\sqrt{\rho_{-}}} \right)\bigg]_{-\pi (1-\frac{\epsilon}{n})}^{\pi(1-\frac{\epsilon}{n})}.
\end{align}
In analogy with \autoref{lem:sqrt_rho_n}, we obtain that when $\sqrt{\rho_{+}} \gg n/\epsilon \gg 1$,
\begin{align}
    \frac{\rho_{+} - 1}{\sqrt{\rho_{+}}} \arctan \left(\frac{\tan(\frac{\pi(1-\epsilon/n)}{2})}{\sqrt{\rho_{+}}} \right) = \frac{2n}{\pi \epsilon} + O\bigg(\frac{2n}{\pi \epsilon \rho_{+}} + \frac{\pi \epsilon}{2n} \bigg),
\end{align}
which means that the right-hand side of equation \eqref{eq:g'_r_n_bound_1} is equal to
\begin{align}
    \frac{n}{\pi(n-1)} \bigg( \frac{2n}{\pi \epsilon} + O\bigg(\frac{2n}{\pi \epsilon \rho_{+}} + \frac{\pi \epsilon}{2n} \bigg) \bigg) = \frac{2n^2}{\pi^2 \epsilon (n-1)} + O\bigg( \frac{n}{\epsilon \rho_{+}} + \frac{\epsilon}{n} + 1 \bigg).
\end{align}
And
\begin{align}
    &\int_{-\frac{\epsilon}{2n}}^{\frac{\epsilon}{2n}} g'(\pi(1-2f)) r_n(2\pi f) \, df \\ &= \int_{-\frac{\epsilon}{2n}}^{\frac{\epsilon}{2n}} \bigg(\frac{\rho_{+}-1}{2 + 2 (\rho_{+} - 1) \cos^2(\frac{\pi(1-2f)}{2})} - \frac{\rho_{-}-1}{2 + 2 (\rho_{-} - 1) \cos^2(\frac{\pi(1-2f)}{2})}\bigg) r_n(2\pi f) \, df.
\end{align}
Using the Taylor approximation of $r_n$ provided by \autoref{lem:r_n_taylor}, we obtain that the first term in the right-hand side is equal to:
\begin{align}
    &\int_{-\frac{\epsilon}{2n}}^{\frac{\epsilon}{2n}} \frac{\rho_{+}-1}{2 + 2 (\rho_{+} - 1) \cos^2(\frac{\pi(1-2f)}{2})} \bigg( \frac{n(n + 1)(2\pi f)^2}{12} + O(n^4 (2\pi f)^4) \bigg) \, df \\ &= \int_{-\frac{\epsilon}{2n}}^{\frac{\epsilon}{2n}} \frac{\rho_{+}-1}{1 + (\rho_{+} - 1) (\pi^2 f^2 + O(f^4))} \bigg( \frac{n(n + 1)\pi^2 f^2}{6} + O(n^4 f^4) \bigg) \, df \\ &=  \bigg(\frac{n(n + 1)}{6} + O(n^2 \epsilon^2) \bigg) \int_{-\frac{\epsilon}{2n}}^{\frac{\epsilon}{2n}} \frac{(\rho_{+}-1) \pi^2 f^2}{1 + (\rho_{+} - 1) \pi^2 f^2} \, df 
     \\ &= \bigg( \frac{n(n + 1)}{6} + O(n^2 \epsilon^2) \bigg) \cdot \frac{1}{\sqrt{\rho_{+} - 1} \pi} \int_{-\frac{\sqrt{\rho_{+} - 1} \pi \epsilon}{2n}}^{\frac{\sqrt{\rho_{+} - 1} \pi \epsilon}{2n}} \frac{y^2}{1 + y^2} \, dy 
     \\ &= 
     \bigg( \frac{n(n + 1)}{6} + O(n^2 \epsilon^2) \bigg) \cdot \frac{1}{\sqrt{\rho_{+} - 1} \pi} \bigg[ y - \arctan(y) \bigg]^{\frac{\sqrt{\rho_{+} - 1} \pi \epsilon}{2n}}_{-\frac{\sqrt{\rho_{+} - 1} \pi \epsilon}{2n}} 
     \\ &= \frac{(n + 1) \epsilon}{6} + O(n \epsilon^3).
\end{align}
where in the first equality we used that $\cos^2(\frac{\pi}{2}-\pi f) = \sin^2(\pi f) = (\pi f - \frac{(\pi f)^3}{6} + O(\pi f)^5)^2 = \pi^2 f^2 - \frac{\pi^4 f^4}{3} + O(f^6)$, and in the third equality we used the change of variables $y = \sqrt{\rho_{+} - 1} \pi f$. Thus, we obtain that
\begin{align}
    \int_{0}^{1} g'(\pi(1-2f)) r_n(2\pi f) \, df = \frac{2n^2}{\pi^2 \epsilon (n-1)}  + \frac{(n + 1) \epsilon}{6} + O\bigg( \frac{n}{\epsilon \rho_{+}} + \frac{\epsilon}{n} + n \epsilon^3 + 1 \bigg).
\end{align}
Plugging this back into equation \eqref{eq:18_copy}, we get:
\begin{align}
    \mathrm{Tr}[\mathrm{Var}(x_k)] = \gamma^2 \sum_{i=1}^{d} (u_i^{\top} \Sigma u_i) \frac{\frac{2n^2}{\pi^2 \epsilon (n-1)}  + \frac{(n + 1) \epsilon}{6} + O\left( \frac{n}{\epsilon \rho_{+}^{i}} + \frac{\epsilon}{n} + n \epsilon^3 + 1 \right)}{2 (1-\alpha) \sqrt{(1-\alpha)^2+ \gamma^2 \lambda_i^2 - 2 (1 + \alpha) \gamma \lambda_i}}
\end{align}
Finally, if we set $\epsilon = n^{-\delta}$ for any $\delta \in (0,1)$, $\sqrt{\rho_{+}} \gg n/\epsilon \gg 1$ is fulfilled and we obtain that 
\begin{align}
    \mathrm{Tr}[\mathrm{Var}(x_k)] &= \gamma^2 \sum_{i=1}^{d} (u_i^{\top} \Sigma u_i) \frac{\frac{2n^{2+\delta}}{\pi^2 (n-1)}  + \frac{(n + 1) n^{-\delta}}{6} + O\left( \frac{n^{1+\delta}}{\rho_{+}^{i}} + n^{-1-\delta} + n^{1-3\delta} + 1 \right)}{2 (1-\alpha) \sqrt{(1-\alpha)^2+ \gamma^2 \lambda_i^2 - 2 (1 + \alpha) \gamma \lambda_i}} \\ &= \gamma^2 \sum_{i=1}^{d} (u_i^{\top} \Sigma u_i) \frac{\bigg(\frac{2n^{2+\delta}}{\pi^2 (n-1)}  
    + O\left( \frac{\gamma^2 \lambda_i^2 n^{1+\delta}}{(1-\alpha)^2} + n^{1-\delta} + 1 \right) \bigg) (1+O(\frac{\gamma \lambda_i}{1-\alpha})) }{2 (1-\alpha)^2} \\ &= \gamma^2 \sum_{i=1}^{d} (u_i^{\top} \Sigma u_i) \frac{  
    n^{2+\delta}+O\bigg( \frac{\gamma^2 \lambda_i^2 n^{2+\delta}}{(1-\alpha)^2} + n^{2-\delta} + n + \frac{\gamma \lambda_i}{1-\alpha} \bigg) }{\pi^2 (n-1) (1-\alpha)^2}.
\end{align}
As in the case of SO, the argument is analogous and yields the same result for SNAG.
\qed

\section{Further experiments and experimental details} \label{sec:more_experiments}

We provide mean squared error values and plots of the squared distance to the optimum for additional combinations of $n$, $\gamma$ and $\alpha$. In \autoref{fig:table1_2} we show the plots for SNAG with the three shuffling schemes in the setting of \autoref{table:1}, which are very similar to the SGDM plots in \autoref{fig:table1_1}.
\begin{figure}
    \centering
    \includegraphics[width=0.99\textwidth]{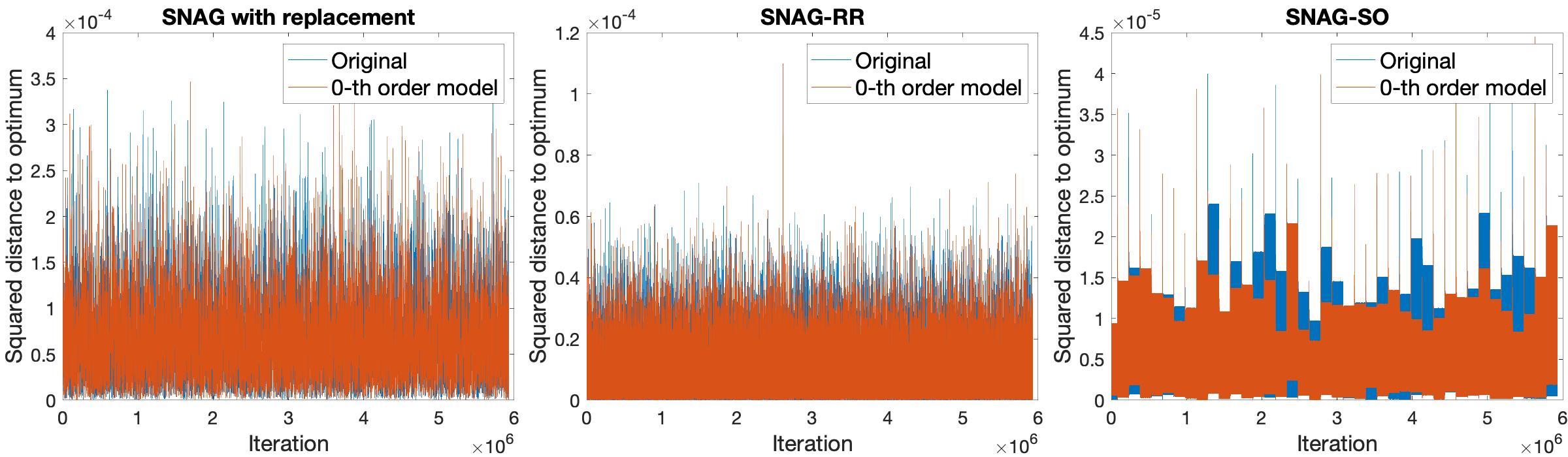}
    \caption{Plots for SNAG with the three shuffling schemes, in the setting of \autoref{table:1}.}
    \label{fig:table1_2}
\end{figure}

In \autoref{table:1_0.9} we show mean squared error values for SGDM and SNAG in the three shuffling schemes, taking the setting from \autoref{table:1} but changing $\alpha = 0.8$ by $\alpha = 0.9$ (there would be no use in showing values for SGD again because it does not depend on $\alpha$). Comparing with \autoref{table:1}, we see observe the dependency on $1/(1-\alpha)$ for schemes with replacement (see \eqref{eq:replacement_limit}), e.g. for SGDM the theoretical values are \num{1.2787e-04} and \num{6.3935e-05}, whose ratio is $2 = (1-0.8)/(1-0.9)$. This is expected, because we are in the regime $\gamma \lambda_i \ll 1-\alpha$ for which equation \eqref{eq:replacement_limit} holds, since for $\lambda = 0.2194$ we have $\gamma \lambda_i = \num{1.097e-04}$ and for $\alpha = 0.9$, $1-\alpha = 0.1$.
We also observe the approximate dependency on $1/(1-\alpha)^2$ for RR and SO (see \eqref{eq:variance_sgdm_so_prop} and \eqref{eq:var_sgdm_RR}), e.g. for SGDM-RR the theoretical values are \num{3.3099e-05} and \num{9.3300e-06}, whose ratio is $3.5476 \approx 4 = (1-0.8)^2/(1-0.9)^2$. Expectedly, this approximation is less exact, because equations \eqref{eq:variance_sgdm_so_prop} and \eqref{eq:var_sgdm_RR} rely on the assumption $\gamma \lambda_i n \ll 1-\alpha$, and for $\lambda = 0.2194$ we have $\gamma \lambda_i n = 0.1097$ while for $\alpha = 0.9$, $1-\alpha = 0.1$. That is, $\gamma \lambda_i n \ll 1-\alpha$ does not quite hold since the two quantities are of the same order.

\begin{table}[]
\small
\centering
\begin{tabular}{|l|l|l|l|}
\hline
Algorithm & Full noise & 0th order noise & Theory  \\ \hline
SGDM & $\num{1.2878e-04} \pm \num{3.5199e-07}$ & $\num{1.2832e-04} \pm \num{4.9145e-07}$ & $\num{1.2787e-04}$
\\ \hline
SGDM-RR & $\num{3.3371e-05} \pm \num{5.8192e-08}$ &  $\num{3.3050e-05} \pm \num{3.7389e-08}$ & $\num{3.3099e-05}$
\\ \hline
SGDM-SO & $\num{2.0123e-05} \pm \num{3.0961e-07}$ &  $\num{2.0179e-05} \pm \num{3.0246e-07}$ & $\num{2.0113e-05}$ 
\\ \hline
SNAG & $\num{1.2868e-04} \pm \num{3.4302e-07}$ & $\num{1.2820e-04} \pm \num{4.9187e-07}$ & $\num{1.2776e-04}$
\\ \hline
SNAG-RR & $\num{3.3319e-05} \pm \num{5.6173e-08}$ &  $\num{3.2997e-05} \pm \num{3.6883e-08}$ & $\num{3.3046e-05}$
\\ \hline
SNAG-SO & $\num{2.0089e-05} \pm \num{3.0906e-07}$ & $\num{2.0144e-05} \pm \num{3.0193e-07}$ & $\num{2.0077e-05}$
\\ \hline
\end{tabular}
\normalsize
\vspace{3pt}
\caption{Mean squared errors $\mathbb{E}[\|x_k - x^{\star}\|^2]$ for $n = 1000$, $d = 5$;
$\gamma = 0.0005$, $\alpha = 0.9$ (seed 38). The estimates and and their standard deviations are computed over 10 runs of \num{6e6} iterations each. The eigenvalues $\lambda_i$ are $0.1807$, $0.1951$, $0.1998$, $0.2033$, $0.2194$. The values $u_i^{\top} \Sigma u_i$ are, in order, $0.0019$, $0.0019$, $0.0022$, $0.0020$, $0.0022$. The theoretical errors for SGDM and SNAG given by the approximations \eqref{eq:replacement_limit} and \eqref{eq:replacement_limit_snag} are $\num{1.2787e-04}$.}
\label{table:1_0.9}
\end{table}

In \autoref{table:2} we show mean squared error values for SGD, SGDM and SNAG in the three shuffling schemes, taking the configuration from \autoref{table:1} but changing $\gamma = 0.0005$ by $\gamma = 0.01$, i.e. setting the stepsize 20 times larger. We observe the dependency on $\gamma$ for algorithms with replacement (see \eqref{eq:replacement_limit}), e.g. for SGDM the theoretical values are \num{1.2794e-03} and \num{6.3935e-05}, whose ratio is 20.01. This is expected, because we are in the regime $\gamma \lambda_i \ll 1-\alpha$ for which equation \eqref{eq:replacement_limit} holds, since for $\lambda = 0.2194$ we have $\gamma \lambda_i = 0.002194$ and $1-\alpha = 0.2$.
The approximate dependency on $\gamma^2$ for algorithms with RR and SO (see \eqref{eq:variance_sgdm_so_prop} and \eqref{eq:var_sgdm_RR}) is not so clearly observed. While we would expect a ratio around 400, for SGDM-SO the theoretical values are \num{1.0231e-03} and \num{5.1800e-6}, whose ratio is 197.50, and for SGDM-RR the theoretical values are \num{1.0479e-03} and \num{9.3300e-06}, whose ratio is 112.31. The reason for the discrepancy is that the assumption $\gamma \lambda_i n \ll 1-\alpha$ which underlies \eqref{eq:variance_sgdm_so_prop} and \eqref{eq:var_sgdm_RR} does not hold when $\gamma = 0.01$: for $\lambda_i = 0.2194$ and $\gamma = 0.01$ we have that $\gamma \lambda_i n = 2.194$, while $1-\alpha = 0.2$. Still, note that as in \autoref{table:1}, our theoretical values match the experimental values under the zero-th order noise, and are very close to the experimental values under the standard stochastic noise. 

\autoref{table:2_0.9} has the same configuration as \autoref{table:2}, changing $\alpha = 0.8$ by $\alpha = 0.9$. Comparing with \autoref{table:2}, we observe the dependency on $1/(1-\alpha)$ for schemes with replacement, e.g. for SGDM-RR the theoretical values are \num{2.5587e-03} and \num{1.2794e-03}, whose ratio is 1.9999. As commented in the previous paragraph, this is expected because \eqref{eq:replacement_limit} is valid as the assumptions $\gamma \lambda_i \ll 1-\alpha$ hold.
The dependency on $1/(1-\alpha)^2$ for RR and SO is not observed clearly, e.g. for SGDM-RR the theoretical values are \num{2.3142e-03} and \num{1.0479e-03}, whose ratio is $2.2084 \not\approx 4 = (1-0.8)^2/(1-0.9)^2$. As mentioned in the previous paragraph, this is because the assumption $\gamma \lambda_i n \ll 1-\alpha$ underlying \eqref{eq:variance_sgdm_so_prop} and \eqref{eq:var_sgdm_RR} does not hold.

\begin{table}[]
\small
\centering
\begin{tabular}{|l|l|l|l|}
\hline
Algorithm & Full noise & 0th order noise & Theory  \\ \hline
SGD & $\num{2.5817e-04} \pm \num{5.1169e-07}$ & $\num{2.5653e-04} \pm \num{6.0118e-07}$ & $\num{2.5599e-04}$  \\ \hline
SGD-RR & $\num{1.1214e-04} \pm \num{1.7867e-07}$ & $\num{1.1038e-04} \pm \num{1.1605e-07}$ & $\num{1.1057e-04}$ \\ \hline
SGD-SO & $\num{8.1243e-05} \pm \num{7.9598e-07}$ & $\num{7.9693e-05} \pm \num{1.0789e-06}$ & $\num{8.0060e-05}$ \\ \hline
SGDM & $\num{1.3190e-03} \pm \num{1.5184e-06}$ & $\num{1.2802e-03} \pm \num{1.4555e-06}$ & $\num{1.2794e-03}$
\\ \hline
SGDM-RR & $\num{1.0934e-03} \pm \num{1.0756e-06}$ &  $\num{1.0456e-03} \pm \num{7.5995e-07}$ & $\num{1.0479e-03}$
\\ \hline
SGDM-SO & $\num{1.0708e-03} \pm \num{7.1175e-06}$ &  $\num{1.0165e-03} \pm \num{9.8846e-06}$ & $\num{1.0231e-03}$ 
\\ \hline
SNAG & $\num{1.3084e-03} \pm \num{1.5119e-06}$ & $\num{1.2701e-03} \pm \num{1.4484e-06}$ & $\num{1.2692e-03}$
\\ \hline
SNAG-RR & $\num{1.0829e-03} \pm \num{1.0697e-06}$ &  $\num{1.0357e-03} \pm \num{7.5421e-07}$ & $\num{1.0380e-03}$
\\ \hline
SNAG-SO & $\num{1.0601e-03} \pm \num{7.0729e-06}$ & $\num{1.0065e-03} \pm \num{9.8162e-06}$ & $\num{1.0130e-03}$
\\ \hline
\end{tabular}
\normalsize
\vspace{3pt}
\caption{Mean squared errors $\mathbb{E}[\|x_k - x^{\star}\|^2]$ for $n = 1000$, $d = 5$;
$\gamma = 0.01$, $\alpha = 0.8$ (seed 38). The estimates and and their standard deviations are computed over 10 runs of \num{6e6} iterations each. The eigenvalues $\lambda_i$ are $0.1807$, $0.1951$, $0.1998$, $0.2033$, $0.2194$. The values $u_i^{\top} \Sigma u_i$ are, in order, $0.0019$, $0.0019$, $0.0022$, $0.0020$, $0.0022$. The theoretical errors for SGDM and SNAG given by the approximations \eqref{eq:replacement_limit} and \eqref{eq:replacement_limit_snag} are $\num{1.2787e-03}$.}
\label{table:2}
\end{table}

\begin{table}[]
\small
\centering
\begin{tabular}{|l|l|l|l|}
\hline
Algorithm & Full noise & 0th order noise & Theory  \\ \hline
SGDM & $\num{2.7214e-03} \pm \num{2.2397e-06}$ & $\num{2.5600e-03} \pm \num{2.3748e-06}$ & $\num{2.5587e-03}$
\\ \hline
SGDM-RR & $\num{2.4920e-03} \pm \num{1.7013e-06}$ &  $\num{2.3095e-03} \pm \num{1.3852e-06}$ & $\num{2.3142e-03}$
\\ \hline
SGDM-SO & $\num{2.4830e-03} \pm \num{1.9304e-05}$ &  $\num{2.3022e-03} \pm \num{1.3852e-06}$ & $\num{2.3036e-03}$ 
\\ \hline
SNAG & $\num{2.6709e-03} \pm \num{2.2728e-06}$ & $\num{2.5150e-03} \pm \num{2.3113e-06}$ & $\num{2.5135e-03}$
\\ \hline
SNAG-RR & $\num{2.4415e-03} \pm \num{1.6206e-06}$ &  $\num{2.2646e-03} \pm \num{1.3942e-06}$ & $\num{2.2693e-03}$
\\ \hline
SNAG-SO & $\num{2.4319e-03} \pm \num{1.8966e-05}$ & $\num{2.2570e-03} \pm \num{1.8489e-05}$ & $\num{2.2584e-03}$
\\ \hline
\end{tabular}
\normalsize
\vspace{3pt}
\caption{Mean squared errors $\mathbb{E}[\|x_k - x^{\star}\|^2]$ for $n = 1000$, $d = 5$;
$\gamma = 0.01$, $\alpha = 0.9$ (seed 38). The estimates and and their standard deviations are computed over 10 runs of \num{6e6} iterations each. The eigenvalues $\lambda_i$ are $0.1807$, $0.1951$, $0.1998$, $0.2033$, $0.2194$. The values $u_i^{\top} \Sigma u_i$ are, in order, $0.0019$, $0.0019$, $0.0022$, $0.0020$, $0.0022$. The theoretical errors for SGDM and SNAG given by the approximations \eqref{eq:replacement_limit} and \eqref{eq:replacement_limit_snag} are $\num{2.5573e-03}$.}
\label{table:2_0.9}
\end{table}

In \autoref{table:3} 
we change the number of functions from $n=1000$ to $n=10$. In this setting, the eigenvalues $\lambda_i$ of $A$ are no longer of the same order: the ratio of the largest one (0.3422) to the smallest one (0.0074) is 46.24. The values $u_i^{\top} \Sigma u_i$ are also of different orders. We make two observations regarding the validity of the zero-th order noise model in this setting:
\begin{itemize}[leftmargin=20pt,topsep=-2pt,itemsep=-2pt]
    \item For algorithms with replacement, the mean squared errors for the standard noise and the zero-th order noise are still very close, i.e. well within the respective confidence intervals. 
    Thus, the zero-th order noise model is a good proxy of the actual noise.
    \item For RR, and even more so for SO, the mean squared errors for the two kinds of noise are not close. For RR, the errors under the standard noise are larger, but both are of the same order. For SO, the errors under the standard noise are between 2 and 5 times larger. Hence, the zero-th order noise model is not a good proxy of the actual noise for algorithms with RR or SO when the $u_i^{\top} \Sigma u_i$ values are highly unequal. Apparently, the error introduced by the zero-th order noise model is of the same order as the estimate. As discussed in \autoref{subsec:difference}, it would be interesting to study a higher-order noise model that captures these behaviors; it is left for future work. 
\end{itemize}
Note that in \autoref{table:3}, for the algorithms with replacement and RR the difference between the theoretical errors and the experimental errors under the zero-th order noise model is within small multiples of the variance. This is not the case for SO algorithms, although the values are still very close. The reason behind the discrepancy is that the experimental error estimates are slightly biased by the transient regimes in between periods with different permutations (see the details on the SO experiments in \autoref{subsec:additional_exp_details}). This bias could be reduced or eliminated by increasing the number of iterations discarded at the beginning of each period, when averaging the errors.

\begin{table}[]
\small
\centering
\begin{tabular}{|l|l|l|l|}
\hline
Algorithm & Full noise & 0th order noise & Theory  \\ \hline
SGD & $\num{2.3563e-07} \pm \num{4.4271e-09}$ & $\num{2.3455e-07} \pm \num{8.0714e-09}$ & $\num{2.4547e-07}$  \\ \hline
SGD-RR & $\num{3.3438e-11} \pm \num{1.2872e-14}$ & $\num{2.5371e-11} \pm \num{3.5778e-15}$ & $\num{2.5374e-11}$ \\ \hline
SGD-SO & $\num{3.0928e-11} \pm \num{6.8468e-13}$ & $\num{1.2771e-11} \pm \num{1.0416e-13}$ & $\num{1.2688e-11}$ \\ \hline
SGDM & $\num{1.2101e-06} \pm \num{1.3509e-08}$ & $\num{1.2103e-06} \pm \num{2.2300e-08}$ & $\num{1.2272e-06}$
\\ \hline
SGDM-RR & $\num{3.0289e-10} \pm \num{1.0927e-13}$ &  $\num{2.4381e-10} \pm \num{5.3817e-14}$ & $\num{2.4386e-10}$
\\ \hline
SGDM-SO & $\num{8.2068e-11} \pm \num{2.3103e-12}$ &  $\num{2.6884e-11} \pm \num{3.4432e-13}$ & $\num{2.6210e-11}$ 
\\ \hline
SNAG & $\num{1.2087e-06} \pm \num{1.3329e-08}$ & $\num{1.2092e-06} \pm \num{2.1955e-08}$ & $\num{1.2272e-06}$
\\ \hline
SNAG-RR & $\num{3.2745e-10} \pm \num{1.2765e-13}$ &  $\num{2.4376e-10} \pm \num{5.3449e-14}$ & $\num{2.4380e-10}$
\\ \hline
SNAG-SO & $\num{1.3884e-10} \pm \num{4.7755e-12}$ & $\num{2.7338e-11} \pm \num{3.5023e-13}$ & $\num{2.6210e-11}$
\\ \hline
\end{tabular}
\normalsize
\vspace{3pt}
\caption{Mean squared errors $\mathbb{E}[\|x_k - x^{\star}\|^2]$ for $n = 10$, $d = 5$;
$\gamma = 0.0002$, $\alpha = 0.8$ (seed 38). The estimates and and their standard deviations are computed over 10 runs of \num{6e6} iterations each. The eigenvalues $\lambda_i$ are $0.0074$, $0.0947$, $0.1322$, $0.2763$, $0.3422$. The values $u_i^{\top} \Sigma u_i$ are, in the same order, $\num{4.7496e-06}$, $\num{1.7296e-05}$, $\num{1.4069e-04}$, $\num{4.8573e-05}$, $\num{1.3472e-04}$. The theoretical errors for SGDM and SNAG given by the approximations \eqref{eq:replacement_limit} and \eqref{eq:replacement_limit_snag} are $\num{1.2273e-06}$.}
\label{table:3}
\end{table}

\autoref{table:3_0.9} is in the same configuration as \autoref{table:3}, but changing $\alpha = 0.8$ by $\alpha = 0.9$. The dependency on $1/(1-\alpha)$ for algorithms with replacement holds with high precision, e.g. for SGDM the theoretical values are \num{2.4545e-06} and \num{1.2272e-06}, whose ratio is $2.00 \approx (1-0.2)/(1-0.1)$. This is because the assumption $\gamma \lambda_i \ll 1-\alpha$ holds: for $\lambda_i = 0.3422$, $\gamma \lambda_i = \num{6.844e-05}$, while $1-\alpha = 0.1$. The dependency on $1/(1-\alpha)^2$ for RR and SO prescribed by \eqref{eq:variance_sgdm_so_prop} and \eqref{eq:var_sgdm_RR} does not hold, e.g. for SGDM-RR the theoretical values are \num{5.5325e-10} and \num{2.4386e-10}, whose ratio is $2.2687 \not\approx 4 = (1-0.2)^2/(1-0.1)^2$, and for SGDM-SO the theoretical values are \num{2.5407e-11} and \num{2.6210e-11}, whose ratio is $0.9693 \not\approx 4 = (1-0.2)^2/(1-0.1)^2$. 
While the assumption $\gamma \lambda_i n \ll 1-\alpha$ holds, because for $\lambda_i = 0.3422$, $\gamma \lambda_i = \num{6.844e-04}$, while $1-\alpha = 0.1$, the assumption $n \gg 1$ does not hold because $n$ is just 10. This explains why equations \eqref{eq:variance_sgdm_so_prop} and \eqref{eq:var_sgdm_RR} do not work in this setting.

\begin{table}[]
\small
\centering
\begin{tabular}{|l|l|l|l|}
\hline
Algorithm & Full noise & 0th order noise & Theory  \\ \hline
SGDM & $\num{2.4318e-06} \pm \num{2.5268e-08}$ & $\num{2.4476e-06} \pm \num{3.2703e-08}$ & $\num{2.4545e-06}$
\\ \hline
SGDM-RR & $\num{6.3833e-10} \pm \num{2.3336e-13}$ &  $\num{5.5307e-10} \pm \num{1.7729e-13}$ & $\num{5.5325e-10}$
\\ \hline
SGDM-SO & $\num{8.3591e-11} \pm \num{2.5475e-12}$ &  $\num{2.5792e-11} \pm \num{3.3510e-13}$ & $\num{2.5407e-11}$ 
\\ \hline
SNAG & $\num{2.4293e-06} \pm \num{2.4789e-08}$ & $\num{2.4456e-06} \pm \num{3.1930e-08}$ & $\num{2.4538e-06}$
\\ \hline
SNAG-RR & $\num{6.8118e-10} \pm \num{2.6013e-13}$ &  $\num{5.5282e-10} \pm \num{1.7732e-13}$ & $\num{5.5299e-10}$
\\ \hline
SNAG-SO & $\num{1.5968e-10} \pm \num{5.9892e-12}$ & $\num{2.6679e-11} \pm \num{3.4739e-13}$ & $\num{2.5407e-11}$
\\ \hline
\end{tabular}
\normalsize
\vspace{3pt}
\caption{Mean squared errors $\mathbb{E}[\|x_k - x^{\star}\|^2]$ for $n = 10$, $d = 5$;
$\gamma = 0.0002$, $\alpha = 0.9$ (seed 38). The estimates and and their standard deviations are computed over 10 runs of \num{6e6} iterations each. The eigenvalues $\lambda_i$ are $0.0074$, $0.0947$, $0.1322$, $0.2763$, $0.3422$. The values $u_i^{\top} \Sigma u_i$ are, in the same order, $\num{4.7496e-06}$, $\num{1.7296e-05}$, $\num{1.4069e-04}$, $\num{4.8573e-05}$, $\num{1.3472e-04}$. The theoretical errors for SGDM and SNAG given by the approximations \eqref{eq:replacement_limit} and \eqref{eq:replacement_limit_snag} is $\num{2.4546e-06}$.}
\label{table:3_0.9}
\end{table}

In \autoref{table:4}, we take the configuration from \autoref{table:3} but change $\gamma = 0.0002$ to $\gamma = 0.002$, i.e. we set the stepsize 10 times larger. We observe the dependency on $\gamma$ for algorithms with replacement (see \eqref{eq:replacement_limit}), e.g. for SGDM the theoretical values are \num{1.2274e-05} and \num{1.2272e-06}, whose ratio is 10.001. This is expected, because we are in the regime $\gamma \lambda_i \ll 1-\alpha$ for which equation \eqref{eq:replacement_limit} holds, since for $\lambda = 0.3422$ and $\gamma = 0.002$ we have $\gamma \lambda_i = \num{6.844e-04}$ and $1-\alpha = 0.2$.
The approximate dependency on $\gamma^2$ for algorithms with RR and SO prescribed by \eqref{eq:variance_sgdm_so_prop} and \eqref{eq:var_sgdm_RR} is also observed. We expect a ratio of $10^2 = 100$. For SGDM-SO the theoretical values are \num{2.6262e-09} and \num{2.6210e-11}, whose ratio is 100.198, and for SGDM-RR the theoretical values are \num{2.4359e-08} and \num{2.4386e-10}, whose ratio is 99.889. Note that the assumption $\gamma \lambda_i n \ll 1-\alpha$ which underlies \eqref{eq:variance_sgdm_so_prop} and \eqref{eq:var_sgdm_RR} holds as explained in the previous paragraph, but $n \gg 1$ does not. Still, as we see the $\gamma^2$ dependency is preserved even when $n \not\gg 1$, as one can see by looking at the arguments. 

\autoref{fig:table4} shows runs for each algorithm in the setting of \autoref{table:4}. Qualitatively, the plots match the analysis we developed around \autoref{table:3}:
\begin{itemize}[leftmargin=20pt,topsep=-2pt,itemsep=-2pt]
    \item In the left column of \autoref{fig:table4}, we observe qualitatively similar behaviors for both kinds of noise. 
    \item In the middle column of \autoref{fig:table4} we see that the sequences of squared distances for RR are qualitatively similar, but higher for the standard noise than for the zero-th order noise. In the left column (see \autoref{subsec:additional_exp_details} for details on the SO experiments and plots), we see that for the standard noise, the squared error stabilizes around widely different values for different permutations, while for the zero-th order noise the differences are smaller and the errors consistently lower. 
\end{itemize}

\begin{table}[]
\small
\centering
\begin{tabular}{|l|l|l|l|}
\hline
Algorithm & Full noise & 0th order noise & Theory  \\ \hline
SGD & $\num{2.4322e-06} \pm \num{2.5264e-08}$ & $\num{2.4479e-06} \pm \num{3.2691e-08}$ & $\num{2.4550e-06}$  \\ \hline
SGD-RR & $\num{3.4625e-09} \pm \num{1.4967e-12}$ & $\num{2.5355e-09} \pm \num{3.5896e-13}$ & $\num{2.5358e-09}$ \\ \hline
SGD-SO & $\num{5.3819e-09} \pm \num{2.0735e-10}$ & $\num{1.2637e-09} \pm \num{1.0370e-11}$ & $\num{1.2694e-09}$ \\ \hline
SGDM & $\num{1.2254e-05} \pm \num{9.6173e-08}$ & $\num{1.2363e-05} \pm \num{7.4615e-08}$ & $\num{1.2274e-05}$
\\ \hline
SGDM-RR & $\num{3.0378e-08} \pm \num{1.1153e-11}$ &  $\num{2.4353e-08} \pm \num{5.5095e-12}$ & $\num{2.4359e-08}$
\\ \hline
SGDM-SO & $\num{9.6012e-09} \pm \num{3.2920e-10}$ &  $\num{2.6070e-09} \pm \num{3.3994e-11}$ & $\num{2.6262e-09}$ 
\\ \hline
SNAG & $\num{1.2231e-05} \pm \num{9.3367e-08}$ & $\num{1.2346e-05} \pm \num{7.4216e-08}$ & $\num{1.2260e-05}$
\\ \hline
SNAG-RR & $\num{3.2857e-08} \pm \num{1.3083e-11}$ &  $\num{2.4298e-08} \pm \num{5.4395e-12}$ & $\num{2.4303e-08}$
\\ \hline
SNAG-SO & $\num{1.8442e-08} \pm \num{7.8794e-10}$ & $\num{2.6070e-09} \pm \num{3.3993e-11}$ & $\num{2.6263e-09}$
\\ \hline
\end{tabular}
\normalsize
\vspace{3pt}
\caption{Mean squared errors $\mathbb{E}[\|x_k - x^{\star}\|^2]$ for $n = 10$, $d = 5$;
$\gamma = 0.002$, $\alpha = 0.8$ (seed 38). The estimates and and their standard deviations are computed over 10 runs of \num{6e6} iterations each. The eigenvalues $\lambda_i$ are $0.0074$, $0.0947$, $0.1322$, $0.2763$, $0.3422$. The values $u_i^{\top} \Sigma u_i$ are, in the same order, $\num{4.7496e-06}$, $\num{1.7296e-05}$, $\num{1.4069e-04}$, $\num{4.8573e-05}$, $\num{1.3472e-04}$. The theoretical errors for SGDM and SNAG given by the approximations \eqref{eq:replacement_limit} and \eqref{eq:replacement_limit_snag} are $\num{1.2273e-05}$.}
\label{table:4}
\end{table}

Finally, \autoref{table:4_0.9} has the same configuration as \autoref{table:4}, but changing $\alpha = 0.8$ to $\alpha = 0.9$. We observe the same comparative behavior as in \autoref{table:3_0.9} vs. \autoref{table:3}.

\begin{table}[]
\small
\centering
\begin{tabular}{|l|l|l|l|}
\hline
Algorithm & Full noise & 0th order noise & Theory  \\ \hline
SGDM & $\num{2.4631e-05} \pm \num{1.3797e-07}$ & $\num{2.4698e-05} \pm \num{1.1961e-07}$ & $\num{2.4548e-05}$
\\ \hline
SGDM-RR & $\num{6.3914e-08} \pm \num{2.4422e-11}$ &  $\num{5.5242e-08} \pm \num{1.7809e-11}$ & $\num{5.5260e-08}$
\\ \hline
SGDM-SO & $\num{8.8167e-09} \pm \num{2.8922e-10}$ &  $\num{2.5271e-09} \pm \num{3.3276e-11}$ & $\num{2.5459e-09}$ 
\\ \hline
SNAG & $\num{2.4559e-05} \pm \num{1.3481e-07}$ & $\num{2.4633e-05} \pm \num{1.2221e-07}$ & $\num{2.4486e-05}$
\\ \hline
SNAG-RR & $\num{6.8010e-08} \pm \num{2.7241e-11}$ &  $\num{5.4985e-08} \pm \num{1.7664e-11}$ & $\num{5.5002e-08}$
\\ \hline
SNAG-SO & $\num{1.8086e-08} \pm \num{7.5204e-10}$ & $\num{2.5277e-09} \pm \num{3.3283e-11}$ & $\num{2.5465e-09}$
\\ \hline
\end{tabular}
\normalsize
\vspace{3pt}
\caption{Mean squared errors $\mathbb{E}[\|x_k - x^{\star}\|^2]$ for $n = 10$, $d = 5$;
$\gamma = 0.002$, $\alpha = 0.9$ (seed 38). The estimates and and their standard deviations are computed over 10 runs of \num{6e6} iterations each. The eigenvalues $\lambda_i$ are $0.0074$, $0.0947$, $0.1322$, $0.2763$, $0.3422$. The values $u_i^{\top} \Sigma u_i$ are, in the same order, $\num{4.7496e-06}$, $\num{1.7296e-05}$, $\num{1.4069e-04}$, $\num{4.8573e-05}$, $\num{1.3472e-04}$. The theoretical errors for SGDM and SNAG given by the approximations \eqref{eq:replacement_limit} and \eqref{eq:replacement_limit_snag} are $\num{2.4546e-05}$.}
\label{table:4_0.9}
\end{table}

\begin{figure}[H]
    \centering
    \includegraphics[width=0.99\textwidth]{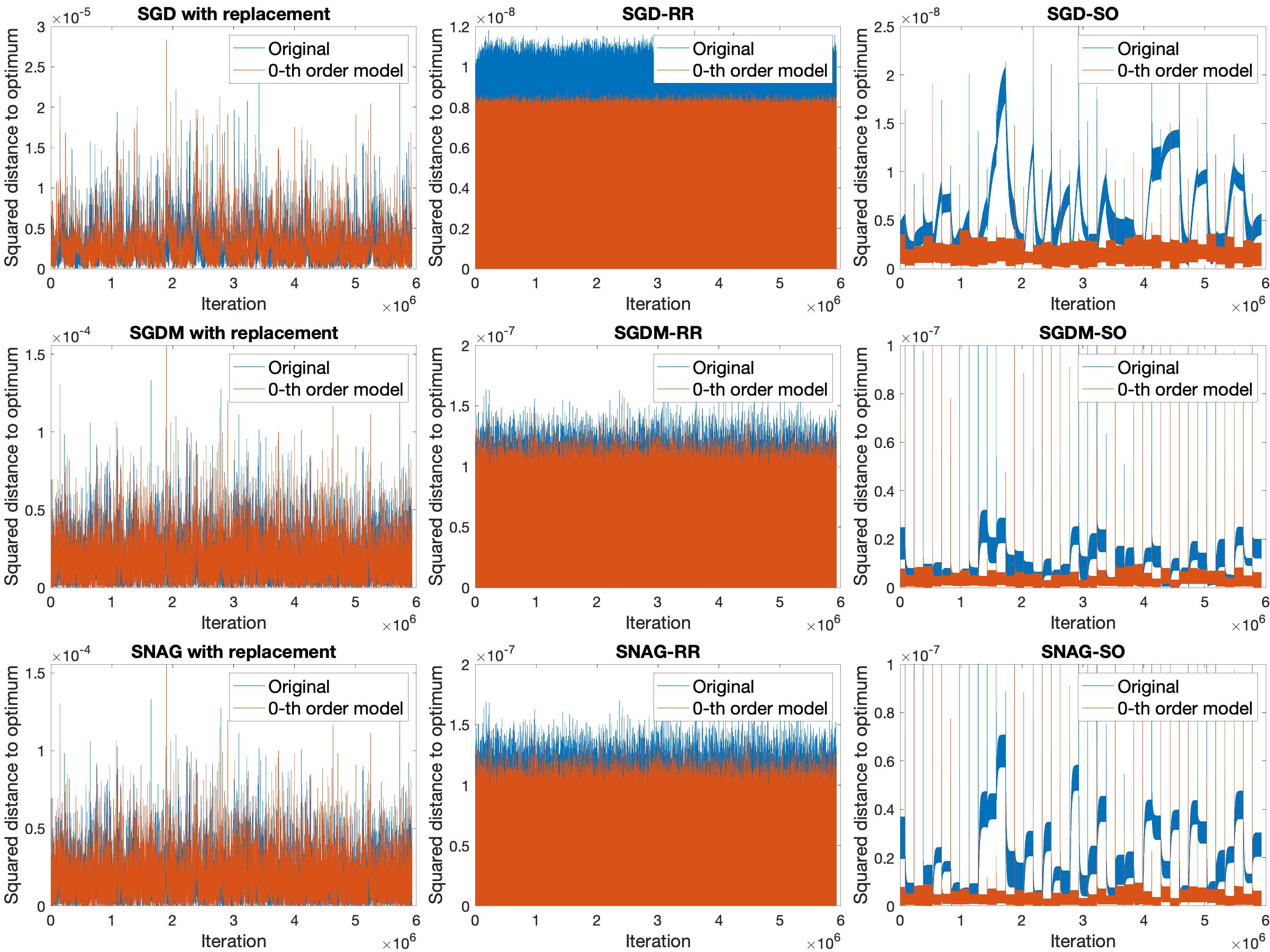}
    \caption{Plots for SGD, SNAG and SGDM with the three shuffling schemes, in the setting of \autoref{table:4}.}
    \label{fig:table4}
\end{figure}

\subsection{Additional experimental details} \label{subsec:additional_exp_details}
The code for the experiments, which can be found in a folder in the supplementary material, is in MATLAB. There is a README file in the code folder. The code was run in a personal laptop. All the experiments required for \autoref{table:1} took under 6 hours to run, and experiments in the other tables took a proportional time. As shown in the figures, all the runs are \num{6e6} iterations long. Since we care about the stationary error, for the purpose of computing the mean squared errors shown in the tables we discard the first \num{1.2e5} iterations of each run.

The experiments for the SO scheme need further explanation. Since the theoretical analysis of SO we average over the $n!$ permutations of the $n$ functions, in our experiments we need to average over the permutations as well. We divide every run of \num{6e6} iterations into 40 periods of \num{1.5e5} iterations each, and we use one random permutation in each period. When we switch from one permutation to the next, the system goes through a transient period until it stabilizes again. This can be seen in the SO plots of \autoref{fig:table1_1}, \autoref{fig:table1_2} and \autoref{fig:table4}. Since the focus is on the stationary error, for the purpose of computing the mean squared errors we discard the first \num{1.2e5} iterations of each period and use just the last \num{3e4} iterations of the period.

\end{document}